\documentclass[11pt]{amsproc}
\pdfoutput=1
\usepackage{amsmath,amssymb,amsthm,eucal, eufrak,amscd,mathtools}
\usepackage[shortlabels]{enumitem}
\usepackage{xcolor}
\usepackage[shortalphabetic]{amsrefs}
\usepackage{tikz-cd}
\usepackage{xypic}
\definecolor{allrefcolors}{rgb}{0,0.2,0.5}
\usepackage[linktocpage=true,colorlinks=true,allcolors=allrefcolors,bookmarksopen,bookmarksdepth=3]{hyperref}
\usepackage[margin=1.25in]{geometry}

\newcommand{\D}{\mathbf{D}}
\newcommand{\PSS}{\operatorname{PSS}}
\newcommand{\PSSlog}{\operatorname{PSS}_{log}}

\newcommand{\PD}{\operatorname{PD}}

\newcommand{\hatX}{\hat{\bar{X}}}

\newcommand{\DIo}{\mathring{D}_I}
\newcommand{\NIo}{\mathring{N}_I}

\newcommand{\SIo}{\mathring{S}_I}
\newcommand{\Etop}{E_{\operatorname{top}}}
\newcommand{\Egeo}{E_{\operatorname{geo}}}
\newcommand{\Evzo}{\operatorname{Ev}^{\vec{\v}}_{z_{0}}}
\newcommand {\Evo}{\operatorname{Ev}^{\vec{\v}}_{0}}
 
\newcommand{\GWv}{GW_{\vec{\v}_I}(\alpha)}
\newcommand{\GWvc}{\underline{GW_{\vec{\v}_I}(\alpha_c)}}

\newcommand{\vi}{v_i}
\def\bd{\partial}
\def\ra{\rightarrow}
\def\lra{\longrightarrow}
\def\Z{{\mathbb Z}}
\def\R{{\mathbb R}}
\def\C{{\mathbb C}}
\def\P{{\mathbb P}}
\def\K{{\mathbf{k}}}

\def\e{\epsilon}

\newcommand{\QH}{H_{log}}

\def\o{\omega}
\def\e{\epsilon}
\def\a{\alpha}
\def\b{\beta}
\def\r#1{\mathrm{#1}}
\def\rb#1{\mathrm{\mathbf{#1}}}
\def\c#1{\mathcal{#1}}
\def\mc#1{\mathcal{#1}}
\def\ol#1{\overline{#1}}

\def\k{\kappa}
\def\d{\Delta}
\def\M{\mathcal{M}}
\def\xo{(X,\o)}
\def\ainf{A_\infty}
\def\f{\c{F}}
\def\sh{SH^*(M)}
\def\bd{\partial}
\def\z2{\Z / 2\Z}
\def\y{\mc{Y}}
\def\endo{\mathrm{End}}
\def\id{\mathrm{id}}
\def\mf#1{\mathfrak{#1}}

\def\a{\mc{A}}
\def\o{\mc{O}}
\def\f{\mc{F}}

\def\z{\mc{Z}}

\def\b{\mc{B}}

\def\v{\mathbf{v}}

\newtheorem{lem}{Lemma}[section]
\newtheorem{prop}[lem]{Proposition}
\newtheorem{thm}[lem]{Theorem}
\newtheorem{cor}[lem]{Corollary}

\newtheorem{defn}[lem]{Definition}

\newtheorem{ques}[lem]{Question}
\newtheorem{rem}[lem]{Remark}
\def\e{\epsilon}

\theoremstyle{remark}

\newtheorem{example}{Example}[section]
\numberwithin{equation}{section}
\begin{document}
\begin{abstract}

    Let $M$ be a smooth projective variety and $\mathbf{D}$ an ample normal crossings divisor. From topological data associated to the pair $(M, \mathbf{D})$, we construct, under assumptions on Gromov-Witten invariants, a series of distinguished classes in symplectic cohomology of the complement $X = M \backslash \mathbf{D}$. Under further ``topological'' assumptions on the pair, these classes can be organized into a {\em Log(arithmic) PSS morphism}, from a vector space which we term the {\em logarithmic cohomology} of $(M, \mathbf{D})$ to symplectic cohomology.  Turning to applications, we show that these methods and some knowledge of Gromov-Witten invariants can be used to produce {\em dilations} and {\em quasi-dilations} (in the sense of Seidel-Solomon \cite{Seidel:2010uq}) in examples such as conic bundles. In turn, the existence of such elements imposes strong restrictions on exact Lagrangian embeddings, especially in dimension 3. For instance, we prove that any exact Lagrangian in any complex 3-dimensional conic bundle must be diffeomorphic to a product $S^1 \times \Sigma_g$ or a connect sum $\#^n S^1 \times S^2$.
\end{abstract}
\title{A Log PSS morphism with applications to Lagrangian embeddings}
\author{Sheel Ganatra and Daniel Pomerleano}
\thanks{S.~G.~  was partially supported by the National Science Foundation through a postdoctoral fellowship --- grant number DMS-1204393 --- and agreement number DMS-1128155. Any opinions, findings and conclusions or recommendations expressed in this material are those of the author(s) and do not necessarily reflect the views of the National Science Foundation.\\
    \mbox{ }\mbox{ }\mbox{ }\mbox{ }\mbox{ }D.~P.~was supported by Kavli IPMU, EPSRC,
and Imperial College.}
\maketitle

\section{Introduction}

Let $M$ be a smooth projective variety, $\mathbf{D}$ an ample strict (or simple) normal
crossings divisor and $X=M \setminus \mathbf{D}$ its complement. The affine
variety $X$ has the structure of an exact finite-type convex symplectic manifold and
hence one can associate to $X$ its {\em symplectic cohomology}, $SH^*(X)$
\cites{Floer:1994uq, Cieliebak:1995fk, Viterbo:1999fk}, a version of classical
Hamiltonian Floer cohomology defined for such manifolds.  Symplectic cohomology
has a rich TQFT structure and plays a central role in understanding the Fukaya
categories of such $X$
\cite{Seidel:2002ys, Abouzaid:2010kx} as well as various approaches to mirror
symmetry \cite{MR3415066}. 

While there are very few complete computations of symplectic cohomology in the
literature, Seidel \cite{Seidel:2002ys, Seidel:2010fk} has suggested that
symplectic cohomology for $X$ as above should be computable in terms of
topological data associated to the compactification $(M,\mathbf{D})$ and
(relative) Gromov-Witten invariants of $(M,\mathbf{D})$. In the case when
$\mathbf{D}$ is smooth, this proposal has been explored 
in some depth by Diogo and Diogo-Lisi \cite{diogothesis, diogolisi}.

The first part of this work introduces an approach to studying the symplectic
cohomology $SH^*(X)$ in terms of the relative geometry of the compactification
$(M,\mathbf{D})$ (when $\mathbf{D}$ is normal crossings). In \S
\ref{subsec:logcoh} we introduce an abelian group which captures the
combinatorics and topology of the normal crossings compactification, called the
{\em log(arithmic) cohomology of $(M, \mathbf{D})$}
\begin{equation}
    \QH^*(M,\mathbf{D}).
\end{equation} 
In the special case that $\mathbf{D} = D$ is a smooth divisor, the
log cohomology has a simple form:
\begin{equation} 
    \QH^*(M, \mathbf{D}) = H^*(M  \setminus \mathbf{D}) \oplus t H^*(S_D)[t]
\end{equation}
where $S_D$ denotes the unit normal bundle to $D$. In general, $\QH^*(M,
\mathbf{D})$ is generated by classes of the form $\alpha t^{\vec{\v}}$, where $\alpha$ lies in the cohomology of certain torus bundles over open strata of $\mathbf{D}$ and $\vec{\v}$ is a multiplicity vector. 

Our approach is inspired by a map introduced by Piunikhin, Salamon, and Schwarz
\cite{Piunikhin:1996aa} relating the (quantum) cohomology of $M$ and the
Hamiltonian Floer cohomology of a non-degenerate Hamiltonian. Namely, we define
a linear subspace of \emph{admissible classes} $H^*_{log}(M,\D)^{ad} \subset
\QH^*(M,\D)$ together with a map 
\begin{align} \label{eq: PSSminus}
    \PSSlog^{+}: \QH^*(M,\D)^{ad} \to SH_{+}^*(X) 
\end{align} where
$SH_{+}^*$ is the {\em positive} or {\em high energy part} of symplectic
cohomology, a canonically defined quotient of the complex defining $SH^*$ by
``low energy (cohomological) generators.''  

There is a wide class  of \emph{topological pairs} $(M, \D)$ for which the map
\eqref{eq: PSSminus} is particularly well behaved (e.g., take $M$ any
projective variety and $\mathbf{D}$ the union of $\geq \dim_{\C} M +1$ generic hyperplane
sections).  For a topological pair, $\QH^*(M,\D)^{ad}=\QH^*(M,\D)$, and there is
a canonical lifting of \eqref{eq: PSSminus} to a map we term the {\em Log PSS
morphism}
\begin{equation}
    \label{eq: PSStopintro} 
    \operatorname{PSS}_{log}: \QH^*(M,\mathbf{D}) \to SH^*(X)
\end{equation}
The key feature of topological pairs is that relevant moduli spaces of {\em
relative} pseudo-holomorphic curves are all generically empty; for example,
there should be no holomorphic spheres in $M$ which intersect a single
component of $\mathbf{D}$ in one point. See \S
\ref{subsec:topological} for the precise definition of a topological pair.

\begin{rem}
 In a sequel to this paper \cite{DPSGII}, we define a natural
ring structure on $H^*_{log}(M, \D)$ and show that the Log PSS morphism \eqref{eq: PSStopintro} is a
ring isomorphism for all topological pairs $(M, \D)$ (rather those pairs satisfying a
strengthening of the topological condition called being {\em multiplicatively
topological}), leading to many explicit new computations of symplectic
cohomology rings. 
\end{rem}

When the pair $(M,\D)$ is not topological and $\sigma \in \QH^*(M, \D)^{ad}$ is
an admissible class, we formulate the obstruction to lifting
$\PSSlog^{+}(\sigma)$ from $SH^*_+(X)$ to $SH^*(X)$ in terms of a Gromov-Witten
invariant associated to $\sigma$. 
When this obstruction vanishes, this provides a way to produce
\emph{distinguished classes} in $SH^*(X)$ which we will show may be applied to
study the symplectic topology of $X$. 

It should be noted that the construction of the Log PSS morphism \eqref{eq:
PSStopintro} and the proof of the chain equation are considerably more involved
than for its classical analogue.  We conjecture that 
for general $(M, \mathbf{D})$, a further analysis of the moduli spaces
appearing in \eqref{eq: PSStopintro} could produce an isomorphism
between
$SH^*(X)$ and the cohomology of a cochain model of log cohomology $C_{log}^*(M, \D)$
equipped with an extra differential accounting for non-vanishing relative
Gromov-Witten moduli spaces. Such analysis would require stronger analytic
results to deal with the subtle compactness and transversality questions which
arise (related issues already arise in the study of relative Gromov-Witten
invariants relative a normal crossings divisor). 
Our obstruction analysis is
indeed inspired by this conjectural theory, and can be viewed as an elementary
special case.   

The rest of this paper explores applications of our construction to Lagrangian
embeddings. 
An interesting class of examples can be constructed as follows. Let $X^o$ be an affine variety with trivial canonical bundle and consider a regular function $f: X^o \to \mathbb{C}$ whose zero set $Z^o$ is smooth. We define the \emph{affine conic bundle} 
to be the affine variety $X$ defined by the equation: \begin{align} 
    \label{eq: conicbundle} X=\lbrace (u,v, \bar{x}) \in \mathbb{C}^2 \times X^o |\; uv=f(\bar{x}) \rbrace 
\end{align}  
The symplectic topology of these varieties 
is very rich and can be approached from different perspectives, see for
instance \cite{AAK, catdyn, KS} (these references generally take the base to be either $(\mathbb{C}^*)^{n-1}$ or $\mathbb{C}^{n-1}$). For example, there is a standard construction of Lagrangian spheres in $X$ given by
taking a suitable Lagrangian disc $j: D^{n-1} \to (\mathbb{C}^*)^{n-1}$ with boundary
on the discriminant locus and ``suspending" it to a Lagrangian $S^n \hookrightarrow X$ \cite{HIV, GM, SeidelSus}. 

It is natural to ask: \emph{what are the possible topologies of exact
Lagrangians in these conic bundles}? The suspension construction typically
provides a large collection of  Lagrangian spheres and Seidel
\cite{catdyn}*{Chapter 11}  has provided
constructions of exact Lagrangian tori in certain examples. Our first
application is the following classification result concerning the diffeomorphism types
of exact Lagrangian submanifolds in 3-dimensional conic bundles over
affine surfaces $X^o$ with trivial canonical bundle. 

\begin{thm} 
    \label{thm: app2gen}(See Theorem \ref{thm: app2sec6gen}) Let $X$ be a
    3-dimensional affine conic bundle of the form \eqref{eq: conicbundle} over
    an affine surface $X^o$ with trivial canonical bundle. Then, any closed,
    oriented, exact Lagrangian submanifold $Q \hookrightarrow X$ must be diffeomorphic to one of the
    following 3-manifolds:
 \begin{itemize}
     \item  a product $S^1 \times B$ for $B$ a Riemann surface of genus $\geq 1$; or
     \item a connected sum $\#_n S^1 \times S^2$ for some $n \geq 0$.
 \end{itemize}
 (by convention, the case $n=0$ corresponds to $S^3$).  
\end{thm}

Theorem \ref{thm: app2sec6gen} represents a relatively sharp classification: following Proposition \ref{thm: app2sec6},  we provide examples of exact Lagrangian embeddings  $\#_n S^1 \times S^2 \to X$ for $n\geq 2$ (these arise by essentially  replacing the disc in the suspension construction by a suitable exact embedding of $C_n$, a sphere minus $n+1$-discs). It also suggests a close connection between exact Lagrangian surfaces in $X^o$ (with boundary on $Z^o$) and exact Lagrangian 3-folds in the affine conic bundle $X$. To illustrate this, it is well-known that there are no exact Lagrangian embeddings of closed surfaces of genus $g \geq 2$ in $(\mathbb{C}^*)^2$. Motivated by this, we prove that when $X^o=(\mathbb{C}^*)^2$, one can strengthen Theorem \ref{thm: app2gen} by excluding exact Lagrangian products of surfaces of genus $g \geq 2$ with $S^1$: 

\begin{prop} \label{thm: app2} (Proposition \ref{thm: app2sec6}) 
    Let $X$ be a 3-dimensional conic bundle
    over $(\mathbb{C}^*)^2$ of the form \eqref{eq: conicbundle} and such that the discriminant locus $Z^o$ is connected. Let $j: Q
    \hookrightarrow X$ be a closed, oriented, exact Lagrangian submanifold of
    $X$. Then $Q$ is diffeomorphic to either $T^3$ or $\#_n S^1 \times S^2$. 
\end{prop}

By combining our methods with those from \cite{Seidel:2014aa}, we also prove
the following result concerning disjoinability of Lagrangian spheres:

\begin{thm} 
    \label{thm: app1}
    (See Theorem \ref{thm: app1sec6}) 
    Let $\mathbf{k}$ be a field and $n \geq 3$ be an odd integer. Suppose that
    $X$ is a conic bundle of the form \eqref{eq: conicbundle} of total
    dimension $n$ over an affine variety $X^o$ with trivial canonical bundle and that $Q_1,\cdots, Q_r$ is a
    collection of embedded Lagrangian spheres which are pairwise disjoinable.
    Then the classes $[Q_1],\cdots,[Q_r]$ span a subspace of
    $H_n(X,\mathbf{k})$ which has rank at least $r/2$.  
\end{thm}

An important special case is when $r=1$, in which case the theorem immediately implies that for any embedded Lagrangian sphere $Q \hookrightarrow X$, the
class $[Q] \in H_n(X,\mathbb{Z})$ is non-zero and primitive  (Corollary
\ref{cor: primitiveclass}). 

For our final application, we study a question posed by Smith concerning the
``persistence (or rigidity) of unknottedness of Lagrangian
intersections.''
\begin{ques}[Smith]\label{ques:smith}
In a symplectic manifold $X^6$, if a pair of Lagrangian 3-spheres $S_1$
and $S_2$ meet cleanly along a circle, can the isotopy class of this knot (in
$S_1$ and $S_2$) change under (nearby) Hamiltonian isotopy?  
\end{ques}
Note that there is no smooth obstruction to changing the knot types. 
In \cite{evanssmithwemyss}, Evans-Smith-Wemyss study the case where $S^1=S_1
\cap S_2$ is unknotted in both factors under an additional assumption on the
identification of normal bundles $$ \eta :\nu_{S^1}S_1 \cong
\nu_{S^1}S_2 $$ induced by the symplectic form on $X$. They prove in this
setting
that if there is a nearby Hamiltonian isotopy $\tilde{S}_1$ of
$S_1$ and $\tilde{S}_2$ of $S_2$ so that $\tilde{S}_1$ continues to meet
$\tilde{S}_2$ cleanly along a knot, then this knot must be the unknot in one
component and the unknot or trefoil in the other; see Proposition
\ref{prop:esw} for a precise statement. In \S \ref{sec:knottedness}, under the
same assumptions on the intersection of $S_1$ and $S_2$, we rule out the
remaining trefoil case, while also giving an alternative proof of Proposition
\ref{prop:esw}, thereby answering Question \ref{ques:smith} negatively in this
case. Our main result, Proposition \ref{unknottedness}, is somewhat
stronger and rules out any pair of Lagrangian 3-spheres meeting cleanly in a
non-trivial knot in a standard ``plumbing'' neighborhood of
$S_1 \cup S_2$.

To explain how the results about conic bundles and knottedness above connect to
our discussion of Log PSS, note that the $\operatorname{PSS}_{log}$ morphism
gives a method for producing distinguished classes in $SH^*(X)$, and
\emph{algebraic relations between
them}\footnote{Much like the classical PSS morphism,
$\operatorname{PSS}_{log}$ in many cases intertwines topological/GW-type
relations in $\QH^*(M, \mathbf{D})$ with TQFT-structure relations in
$SH^*(X)$. This is explored further in the sequel \cite{DPSGII}.}. It is well-understood, going back to ideas of Viterbo, that the
existence of solutions in $SH^*(X)$ satisfying certain relations often
produce strong restrictions on Lagrangian embeddings \cite{Viterbo:1996kx,
Seidel:2010uq, Seidel:2014aa}.
Concretely, let $Q \hookrightarrow X$ denote an exact Lagrangian embedding of a
closed Spin manifold. The well-known {\em Viterbo transfer map} \cite{Viterbo:1996kx, Ritter}  gives a
unital BV morphism from $SH^*(X)$ to the symplectic cohomology of a
neighborhood of $Q$, $SH^*(T^{*}Q)$, which in turn can be identified \cite{Viterbo:1996kx, Salamon:2006ys, AbSch1, AbSch2,Abouzaid:2015ad}
with the {\em string topology} BV algebra of $Q$, $H_{n-*}(\mc{L} Q)$, a
topological invariant \cite{Chas:aa}. Thus, one can transfer distinguished
elements in $SH^*(X)$ to distinguished elements in $H_{n-*}(\mc{L} Q)$, the
existence of which in turn restricts
the topology of $Q$, especially in dimension 3. The broad idea of producing
classes/relations in loop homology in order to constrain the 3-manifold
topology of a Lagrangian embedding first appears in Fukaya's pioneering
work \cite{FukayaLoopSpace1, FukayaLoopSpace2} 
on Lagrangian embeddings in
$\C^3$; in contrast the types of distinguished classes we produce are
different, as are both the means by which we obtain them and our method of
deducing topological restrictions from them; see Remark \ref{fukayacontrast}.

As an example,  a {\em dilation} \cite{Seidel:2010uq}  is
an element $B \in SH^1(X)$ satisfying $\Delta B = 1$, where $\Delta$ denotes
the BV operator.  A straightforward application of Viterbo's principle (along
with a string topology computation) implies that
if $SH^*(X)$ admits a dilation, $X$ does not contain an exact Lagrangian $K(\pi, 1)$ \cite[Cor. 6.3]{Seidel:2010uq}. It should be noted that the existence of dilations has other implications for Lagrangian embeddings as well, for example  \cite{Seidel:2014aa} has used dilations to study disjoinability questions for Lagrangian spheres. Weakening the above condition, a {\em quasi-dilation} (\cite{catdyn}*{page 194}, where the definition is attributed to Seidel-Solomon) is a pair $(\Psi_S, \alpha) \in SH^1(X) \times SH^0(X)$ where $\alpha$ is invertible and $\Delta (\alpha\Psi_S) = \alpha$. 
The existence of quasi-dilations imposes slightly weaker, but nevertheless quite strong
constraints on the topology of $Q$. We make these constraints very explicit 
when $\dim_{\R} Q = 3$, using techniques from 3-manifold topology:
\begin{prop}\label{prop:3manifoldquasidilation} (Corollary \ref{cor: Vitdil})
      Let $Q^3 \hookrightarrow X$ be an exact Lagrangian embedding of a closed
      oriented 3-manifold. If $SH^*(X,\mathbb{Z})$ admits a quasi-dilation,
      then $Q$ is either diffeomorphic to $S^1 \times B$ with $B$ a Riemann
      surface or diffeomorphic to a connected sum  $\#_n S^1 \times S^2$.  
  \end{prop}
 There are variants of Proposition \ref{prop:3manifoldquasidilation} for quasi-dilations in $SH^*(X,\mathbb{Q})$; see Corollary \ref{cor: Vitdil}. 
As usual, the chief difficulty in applying Proposition
\ref{prop:3manifoldquasidilation} and its variants to restrict exact
Lagrangian embeddings is performing partial computations of the BV algebra
structure in symplectic cohomology in order to produce quasi-dilations. 
The Log PSS morphism developed in this paper gives a method of performing such
computations, which we use to prove the following result,
suggested by Seidel using mirror symmetric
considerations \cite{catdyn}*{first paragraph of Lecture 19}:
\begin{thm} 
    \label{thm:intro3} (Theorem \ref{thm: generalquasi}) Let $X$ be the conic
    bundle appearing in \eqref{eq: conicbundle}. Then $SH^*(X, \mathbb{Z})$ admits a
    quasi-dilation.   
\end{thm}
When combined with Proposition \ref{prop:3manifoldquasidilation}, this
immediately implies Theorem \ref{thm: app2gen}. Proposition \ref{thm: app2} is then
proven by combining Theorem \ref{thm:intro3} with some additional algebraic
relations which exist in the zeroth symplectic cohomology group $SH^0(X)$ when
$X^o = (\mathbb{C}^*)^2$,
which enable us to further rule out Lagrangians in $X$ of the form $S^1 \times
B$ for $g(B)>1$.

  Seidel \cite{catdyn}*{ Lecture 19} has shown that quasi-dilations have similar implications for disjoinability  questions (in the sense of \cite{Seidel:2014aa}) as dilations. Theorem \ref{thm: app1} therefore follows by combining Seidel's results from {\em loc. cit.} with Theorem \ref{thm:intro3}. An essential ingredient in Theorem \ref{thm: app1} is the specific geometry of the construction of
the quasi-dilation, which is what enables us to interpret the constraints
imposed by disjoinability in terms of the topology of $X$.

We expect that our techniques apply in many other geometric situations as well. For example, we also study the case where $M$ is a Fano variety of dimension $n \geq 3$  and $D$ is a smooth very ample divisor satisfying $H^2(M,\mathbf{k})= \mathbf{k} \langle \PD(D) \rangle$ and $K_M^{-1}=\mathcal{O}(mD)$ with $m \geq \frac{n}{2}$. Under
hypotheses on the cohomology and Gromov-Witten invariants of such $(M, D)$, we
show using $\PSS_{log}$ that $SH^*(M \backslash D)$ admits a dilation, see
Theorem \ref{thm: dilationcrit} below. The theorem is closely related to a question of Seidel
\cite{catdyn}*{Conjecture 18.6}, see Remark \ref{rem:Seidelconj}. 

\subsection*{Related work}
We note that there are several independent recent and/or ongoing works
which have non-trivial conceptual overlap with the present paper. We have
already mentioned the work 
Diogo-Lisi \cite{diogothesis, diogolisi}, who use techniques from symplectic
field theory to give a formula for the abelian group $SH^*(X)$ in terms of
topological data and relative Gromov-Witten invariants for a large class of
pairs $(M,\D)$ with $\D=D$ smooth. Closer in spirit to the present work, a
forthcoming paper of Borman and Sheridan \cite{bormansheridan} constructs a
particular distinguished class in
 $SH^*(M \backslash \D)$ for suitable pairs $(M,\D)$ (see
Remark \ref{rem: BormanSheridan} for more details about how their class fits
into our framework) and uses this to develop a relationship between $SH^*(X)$
and $QH^*(M)$, the quantum cohomology of the compactification $M$. A modified
version of their construction in the setting of Lefschetz pencils also appears
in a recent preprint of Seidel \cite{SLef}. \vskip 5 pt 

\subsection*{Acknowledgments} 
Paul Seidel's work has decisively influenced many aspects of this paper; he
also answered the author's questions concerning \cite{catdyn} and gave comments
on an earlier draft. The authors had some inspiring conversations with Strom
Borman, Luis Diogo, Yakov Eliashberg, Eleny Ionel, Mark McLean, Sam Lisi, and
Nick Sheridan about symplectic cohomology and normal crossings
compactifications at an early stage of this project. The examples following
Theorem \ref{thm: app2sec6} are based upon a suggestion of Jonny Evans. We
thank Jonny Evans, Ivan Smith, and Michael Wemyss for sharing with us a preliminary
copy of their work \cite{evanssmithwemyss}, and Ivan Smith for conversations about how to apply 
our results to study Lagrangians intersecting cleanly along a circle.  Alessio
Corti suggested the $IG(2,2n+2)$ examples which were studied in the first version of this paper and also patiently answered questions
about conic bundles. We thank them all for their assistance. 
Finally, we would also like to thank an anonymous referee for detailed remarks and
comments which greatly improved the exposition of this article.
\vskip 5 pt

\subsection*{Conventions}\label{conventions} Throughout this paper the ring $\mathbf{k}$ will be
either $\Z$, $\mathbb{Q}$, or $\C$ unless explicitly stated otherwise (see Remark \ref{rem:groundring}).
All grading conventions in this paper are cohomological,
and our conventions for grading
symplectic cohomology follows \cite{Abouzaid:2010kx}.

\section{Symplectic cohomology}

\subsection{The definition}

Let $(\bar{X}, \theta)$ be a {\em Liouville domain}; that is, a $2n$ dimensional manifold $\bar{X}$ with boundary $\partial \bar{X}$ with one form
$\theta$ such that $\omega:= d\theta$ is symplectic and the Liouville vector
field $Z$ (defined as the $\omega$ dual of $\theta$) is outward pointing along
$\partial \bar{X}$. More precisely, flowing along $-Z$ is defined for all times, giving rise to a canonical embedding of the
negative half of the symplectization 
\begin{align} 
    j: (0,1]_R \times \partial \bar{X} &\to \bar{X}\\
    j^*(\omega) &= d(R\theta|_{\partial \bar{X}})
\end{align} The {\em (Liouville) completion} of $\bar{X}$ is formed by
attaching (using the aforementioned identification) a cylindrical end to the
boundary of $\bar{X}$:
\begin{equation}\label{eq:conicalstructure}
    \hatX:=\bar{X} \cup_{\partial\bar{X}} [1,\infty)_R \times \partial \bar{X}.
\end{equation}
The Liouville completion $\hatX$ is also exact symplectic, equipped with one form
$\hat{\theta}$ with $\hat{\omega} = d\hat{\theta}$ symplectic. Explicitly,
$\hat{\theta}$ is equal to $\theta$ on $\bar{X}$ and  $R \theta$ on the
cylindrical end, where $R$ denotes the canonical $[1,\infty)$ coordinate.
Recall that in the above situation, $(\partial \bar{X}, \alpha:= \theta|_{\partial
\bar{X}})$ is a {\em contact manifold}; it possesses a canonical {\em Reeb
vector field} $X_{Reeb}$ determined by $\alpha(X_{Reeb}) = 1$,
$d\alpha(X_{Reeb}, -) = 0$. 
The {\em spectrum of $\partial \bar{X}$}, denoted $\operatorname{Spec}(\partial \bar{X},\theta)$ 
is the set of real numbers which are lengths of Reeb orbits in $\partial
\bar{X}$. We will assume throughout that this spectrum is discrete (which
follows from choosing $\alpha$ sufficiently generic).

Observe that after choosing a compatible almost-complex structure then
(negative) the first Chern class $-c_1(\hatX)$ is represented by the complex line
bundle $\mathcal{K}=\det_{\C}(T\hatX)^{\vee}$.  For simplicity (and in order to set up
$\mathbb{Z}$-gradings), we will assume that $c_1(\hatX) = 0$ and moreover that we
have fixed a choice of trivialization of $\mathcal{K}$.  

\begin{defn} \label{def:admissibleham}
     A (time-dependent) Hamiltonian $H: S^1 \times \hatX \to \mathbb{R}$ is {\em admissible with
     slope $\lambda$} if $H(t,x )=\lambda R+C_0$ near infinity, where $R$ is the
     cylindrical coordinate near $\infty$ (note $R$ only depends on $x \in \hatX$) and $C_0$ is an arbitrary constant.
\end{defn} 

 For each positive real number $\lambda$ which is not in the spectrum of
 $\partial \bar{X}$, we choose an admissible Hamiltonian of slope $\lambda$,
 denoted $H^{\lambda}$. Each Hamiltonian $H$ determines an {\em action functional} on the free loop
 space of $\hatX$, whose value on a free loop $x:  S^1
 \to \hatX$ is given by 
 \[ 
 \mathcal{A}_{H}(x):= -\int_{S^1} x^*\hat{\theta} + \int_{0}^1 H(t, x(t)) dt  
 \]
By definition (and partly convention), the {\em Hamiltonian vector-field} of
$H$ with respect to $\hat{\omega}$ is the unique ($S^1$-dependent) vector-field
$X_H$ on $\hatX$ such that $\hat{\omega}(-, X_H)=dH$. Critical points of
$\mathcal{A}_H$ are time-1 orbits of the Hamiltonian vector-field $X_H$ with
respect to the symplectic form $\hat{\omega}=d\hat{\theta}$. It is well-known
that we may perturb the functions $H^\lambda$ so that all 1-periodic orbits are non-degenerate, while keeping the functions admissible (for a nice reference
directly applicable to our setting see \cite{McLean:2012ab}*{Lemma 2.2}).  

Let $\mathcal{X}(H^{\lambda})$ denote the set of time-1 orbits of
$X_{H^{\lambda}}$. To each orbit $x \in \mathcal{X}(H^\lambda)$, there is an
associated
one-dimensional real vector space $\mathfrak{o}_x$ called the {\em orientation
line}, defined via index theory as the {\em determinant line} of a local
operator $D_x$ associated to $x$ (this goes back to \cite{Floer:1995fk} but our
treatment follows e.g., \cite{Abouzaid:2010kx}*{\S C.6}; note particularly that
the choice of local operator is fixed by the choice of trivialization of
$\mathcal{K}$ chosen above).  Over a ring $\K$,
the {\em Floer co-chain complex} is, as a vector space
\begin{equation} \label{eq:Floercomplexvec} 
    CF^*(\hatX,H^\lambda):= \bigoplus_{x \in \mathcal{X}(H^\lambda)}|\mathfrak{o}_x|_{\mathbf{k}} \end{equation}
where for any real one-dimensional vector space $V$, its {\em $\K$-normalization}
\begin{equation}\label{eq:normalization}
    |V|_{\K}
\end{equation}
is the rank one free $\K$-module generated by possible orientation of $V$,
modulo the relation that the sum of opposite orientations vanishes. To define a
differential on \eqref{eq:Floercomplexvec}, pick a compatible (potentially
$t\in S^1$-dependent) almost complex structure $J_t$ which is 
{\em of contact type}, meaning on the conical end, it satisfies
\begin{align}
    \label{eq: contacttype}  
    \hat{\theta} \circ J_t= dR. 
\end{align}
Given a pair of orbits $x_1$ and $x_0$, a {\em Floer trajectory} between $x_1$ and $x_0$ is formally a gradient flowline for $\mc{A}_H$ between $x_1$ and $x_0$ using the metric on $\mc{L} \hatX$ induced by $J_t$, or equivalently a map $u: \mathbb{R} \times
S^1 \to \hatX$, asymptotic to $x_\pm$ at $\pm \infty$
satisfying a PDE called {\em Floer's equation}:
\begin{align} \label{eq:FloerRinv}
\left\{
\begin{aligned}
 & u \colon \mathbb{R} \times S^1 \to \hatX, \\
& \lim_{s \to -\infty} u(s, -) = x_0\\
& \lim_{s \to +\infty} u(s, -) = x_1 \\
 &  \partial_s u + J_{t}(\partial_tu-X_{H^{\lambda}_{t}})=0.
\end{aligned}
\right.
\end{align}
Denote by 
\begin{equation}\label{eq:floertraj}
    \widetilde{\mathcal{M}}(x_0,x_1)
\end{equation}
the moduli space of Floer trajectories between $x_1$ and $x_0$, or solutions to
\eqref{eq:FloerRinv}.
For generic $J_t$, \eqref{eq:floertraj} is a manifold of dimension
\[ 
    |x_0| - |x_1| 
\]
where $|x|$ (or equivalently $\deg(x)$), the index of the local operator $D_x$ associated to $x$, is
equal to $n - CZ(x)$, where $CZ(x)$ is the Conley-Zehnder index of
$x$ (with respect to the symplectic trivialization of $x^* T\hatX$ induced by the
trivialization of $\mathcal{K}$) \cite{Cieliebak:1995fk, Floer:1995fk}. 
There is an induced $\mathbb{R}$-action on the
moduli space $\widetilde{\mathcal{M}}(x_0,x_1)$ given by translation in the $s$-direction,
which is free for non-constant solutions. Whenever $|x_0| - |x_1| \geq 1$, 
for a generic $J_t$ the quotient space 
\begin{equation}\label{eq:modulispacetraj}
    \mathcal{M}(x_0, x_1) := \widetilde{\mathcal{M}}(x_0,x_1)/ \mathbb{R}
\end{equation}
is a manifold of dimension $|x_0| - |x_1| - 1$.

The most delicate part of setting up such a theory for exact (non-compact)
symplectic manifolds, where much of the analysis is otherwise simplified, is an
{\em a priori compactness result} ensuring that solutions to the PDE cannot escape to
$\infty$ in the target (this is vital input into the usual Gromov compactness
theorems for moduli spaces, which apply to maps into a compact target).
Technically, one can ensure a priori compactness by 
carefully choosing the behavior of $J_t$ and $H$ near $\infty$ so that a sort
of {\em maximum principle} holds for solutions to Floer's equation
(at least outside a compact set); see e.g., \cite{Seidel:2010fk}. The maximum principle
ensures that any solution $u$ to Floer's equation remains in some compact
subset of the target $Z \subset \hatX$, which depends only on the asymptotics of
$u$ (and possibly $\hatX$, $H$, $J_t$).

Given that input, a  version of Gromov compactness ensures that for
generic choices, whenever $|x_0| - |x_1| = 1$,
the moduli space
\eqref{eq:modulispacetraj} is compact of dimension 0. Moreover,
orientation theory associates, to every rigid element 
$u \in \mathcal{M}(x_0,x_1)$
an isomorphism of orientation lines $\mu_u:
\mathfrak{o}_{x_1} \stackrel{\cong}{\ra} \mathfrak{o}_{x_0}$ and hence an induced map $\mu_u:
|\mathfrak{o}_{x_1}|_{\K} \ra |\mathfrak{o}_{x_0}|_{\K}$. 
Using this, one defines the $|\mathfrak{o}_{x_1}|_{\K}\!-\! |\mathfrak{o}_{x_0}|_{\K}$
component of the differential
\begin{equation} 
    (\partial_{CF})_{x_1, x_0}  = \sum_{u \in \mathcal{M}(x_0,x_1)} \mu_u
\end{equation} 
whenever $|x_0| = |x_1| + 1$ (and 0 otherwise).
When applied to 2 dimensional moduli spaces (which are 1 dimensional after
quotienting by $\R$), Gromov compactness and gluing analysis imply that:
\begin{lem} \label{lem:Floercomp} 
    Given $x_0,x_1 \in \mathcal{X}(H^{\lambda})$ such that $|x_0| = |x_1| + 2$
    for generic $J_t$,
    $\mathcal{M}(x_0,x_1)$
    admits a compactification $\overline{\mc{M}}(x_0,x_1)$ with
    \[
        \partial \overline{\mc{M}}(x_0,x_1) := \bigsqcup_{y \in \mathcal{X}(H^{\lambda}), |y| = |x_1| + 1}
 \mathcal{M}(x_0,y)  \times \mathcal{M}(y,x_1).\]
\end{lem}
By a standard argument this implies that $\partial_{CF}^2=0$ and hence that the cohomology of $\partial_{CF}$ is well-defined.
We will denote by 
\begin{equation}
    HF^*(\hatX,H^\lambda):=H^*(CF^*(\hatX,H^\lambda),\partial_{CF});
\end{equation}  
standard techniques involving continuation maps shows that this group only
depends on $\lambda \in \R$. 
 Furthermore, whenever $\lambda_2 \geq \lambda_1$ there are continuation maps \cite{Seidel:2010fk}
 \begin{equation}\label{eq:continuation}
 \mathfrak{c}_{\lambda_1,\lambda_2}: HF^*(\hatX, H^{\lambda_1})  \to HF^*(\hatX,H^{\lambda_2}) 
 \end{equation}
 defined as follows: Let $H_{s,t}$ be a map from $\R \times S^1 \ra C^{\infty}(\hatX;
 \R)$ which agrees with $H^{\lambda_1}$ near $+\infty$ and $H^{\lambda_2}$ near
 $-\infty$ and which is monotonic, meaning (at least away from a compact set in
 $X$) $\partial_s H \leq 0$. Let $J_{s,t}$ be a compatible $\R \times S^1$ dependent almost
 complex structure agreeing with a choice of $J_{\lambda_1}$ used to define
 $CF^*(\hatX, H^{\lambda_1})$ near $+\infty$ and a choice of $J_{\lambda_2}$ used
 to define $CF^*(\hatX, H^{\lambda_2})$ near $-\infty$. 
Then, if $x_1$ is an orbit of $H^{\lambda_1}$ and $x_2$ is an orbit of
$H^{\lambda_2}$, the  $\mathfrak{o}_{x_1}\! -\! \mathfrak{o}_{x_2}$ component
of the map \eqref{eq:continuation} is defined on the chain level
by counting maps $u: \R \times S^1 \ra \hatX$ satisfying Floer's equation for 
$H_{s,t}$ and $J_{s,t}$:
\[
    \partial_s u + J_{s,t} (\partial_t u -  X_{H_{s,t}}) = 0
\]
which in addition satisfy requisite asymptotics:
\begin{align} \label{eq: continuationmaps}
\left\{
\begin{aligned}
 & \lim_{s \to -\infty} u(s, -) = x_2\\
 & \lim_{s \to +\infty} u(s, -) = x_1.
\end{aligned}
\right.
\end{align}  
(A standard transversality, compactness, and gluing argument ensures that the
chain level map is in fact a chain map, as long as a maximum principle holds
for elements of the moduli space. This maximum principle is where one requires $\lambda_2 \geq
\lambda_1$.)

Define {\em symplectic cohomology} to be the colimit of this directed system 
\begin{equation} 
    SH^*(\hatX) :=\varinjlim_{\lambda} HF^*(\hatX, H^{\lambda}).
\end{equation}
The symplectic cohomology of a Liouville domain is by definition the symplectic cohomology of its completion
\begin{equation}
    SH^*(\bar{X}):= SH^*(\hatX).
\end{equation}
It is not hard to prove that these definitions are independent of the various
choices made along the way \cite{Seidel:2010fk}. By construction there are
canonical morphisms, for each $\lambda$
\begin{align} 
    \mathfrak{c}_{\lambda,\infty}: HF^*(\hatX, H^{\lambda}) \to SH^*(\hatX).
\end{align} 
In particular, given that 
there is a canonical isomorphism $HF^*(\hatX, H^{\lambda}) \cong H^*(\hatX)$
when $0 < \lambda \ll 1$ is smaller than the period of
any Reeb orbit on $\partial \bar{X}$ (see e.g., \cite{Ritter}*{\S 5}), 
there is a canonical map $H^*(\bar{X}) \to SH^*(\bar{X})$ \cite{Viterbo:1999fk}. 
The cone of (a chain-level version of) this map is often called {\em high
energy} (or {\em positive}) {\em symplectic cohomology} and denoted
$SH^*_+(\bar{X})$  (Recall from the introduction that the most canonical formulation
of the Log PSS map in the presence of holomorphic spheres is in terms of
$SH^*_+(\bar{X})$).

Let us review a convenient alternate construction of high energy symplectic
cohomology $SH^*_+(\bar{X})$, which makes use of a slightly more restrictive choice
of Hamiltonians (but manages to avoid using chain-level colimit constructions).
Consider Hamiltonians $H^{\lambda}$ on $\hatX$ for which: 
\begin{itemize}
    \item $H^{\lambda}=0$ on $\bar{X}$;
    \item Over the collar region satisfy $H^{\lambda}=h_{\lambda}(r)$, with $h_{\lambda}'(r) \geq 0$ and $h_{\lambda}''(r) \geq 0$;
\item For some $R_H$ near 1, we have that $h_{\lambda}(r)=\lambda (r-1)$ for $r \geq R_H$. 
\end{itemize}
After taking
a suitable small perturbation, which for simplicity we may assume is time
independent in the interior of $\bar{X}$, action considerations show that the
(orientation lines associated to) orbits in the interior generate a subcomplex 
$CF_{-}^*(\hatX, H^{\lambda}) \subset CF^*(\hatX,H^{\lambda})$.  Set \begin{align} \label{eq:posySH} CF^*_{+}(\hatX,
    H^{\lambda}):=\frac{CF^*(\hatX,H^{\lambda})} {CF_{-}^*(\hatX,H^{\lambda})}
\end{align} This construction passes to direct limits, giving rise to 
\begin{align} 
    SH^*_{+}(\bar{X}):= \varinjlim_\lambda HF_{+}^*(\hatX,H^{\lambda}) 
\end{align} 
It is not difficult to see that there is a long exact sequence 
\begin{align} 
    \cdots \to H^*(\bar{X}) \to SH^*(\bar{X}) \to SH^*_{+}(\bar{X}) \to H^{\ast+1}(\bar{X}) \to \cdots 
\end{align}
Finally, we observe that for such Hamiltonians $H^{\lambda}$, the integrated maximum principle of
\cite{Abouzaid:2010ly}*{Lemma 7.2} implies that all Floer trajectories and continuation maps (for $\lambda \leq \lambda'$) betweeen orbits actually lie in the compact region $\bar{X} \cup \{R \leq R_H\}$ (seeing as all orbits of $H^{\lambda}$ and $H^{\lambda'}$ themselves lie in this region).

\subsection{Algebraic structures}\label{subsec:algstructures}
Among its many other TQFT structures, symplectic cohomology is a {\em BV-algebra}; in particular it has a {\em pair of pants product} and a  {\em BV operator} which we now define. 

Recall that a {\em negative cylindrical end}, resp. a {\em positive cylindrical end} near a puncture $z$ of a Riemann surface $\Sigma$ consists of a proper holomorphic embedding
\begin{equation}
  \begin{split}    
   \epsilon_-: (-\infty,0] \times S^1 &\rightarrow \Sigma  \label{eq:negstrip} \\
   \end{split}
 \end{equation}
 resp. a proper holomorphic embedding
 \begin{equation} \label{eq:positivecylinder}
    \begin{split} \epsilon_+: [0,\infty) \times S^1 &\rightarrow
            \Sigma\\ \end{split} 
\end{equation} 
asymptotic to $z$. We use standard product coordinates $(s,t)$ on both of these
semi-infinite cylinders inherited from their embedding in $\R_s \times S^1_t =
\R \times (\R_t / \Z)$, with associated standard complex structure $j
\partial_s = \partial_t$.  Let $\Sigma$ be a Riemann surface equipped with
suitable cylindrical ends $\epsilon_i$.  Suppose to each cylindrical end we
have associated a time-dependent Hamiltonian $H_i$. Let $K \in
\Omega^1(\Sigma,C^{\infty}(\hatX))$ be a 1-form on $\Sigma$ with values in smooth
functions on $\hatX$ which, along the cylindrical ends, satisfies:
$$\epsilon_i^*(K)=H_idt \textrm{ for $|s| \gg 0$}$$ for some functions $H_i$.  To such a $K$, which we call a {\em perturbation 1-form}, we may
associate a Hamiltonian vector field-valued  1-form $X_K \in
\Omega^1(\Sigma,C^{\infty}(T\hatX))$, characterized by the property that for any
tangent vector $\vec{r}_z \in T_z \Sigma$, $X_K(\vec{r}_z)$ is the
Hamiltonian vector-field of the Hamiltonian function $K(\vec{r}_z)$ (on the
ends, $\epsilon_i^*(X_K) = X_{H_i} \otimes dt$).  
    
In order to define Floer-theoretic operations, we fix the following additional data on $\Sigma$: 
\begin{itemize}
    \item a (surface-dependent) family of {\em admissible} $J$, meaning $J$ should be of {\em contact type}.  Further, when restricted to cylindrical ends $J$ should depend only on $t$. 
    \item a subclosed 1-form $\beta$ (meaning $d \beta \leq 0$ pointwise, where positivity resp. negativity is detected by comparison with the standard complex orientation), which restricts to  $d_i \cdot dt$, for $d_i \in \mathbb{Z}^{+}$, when restricted to the cylindrical ends. 
    \item A perturbation 1-form $K$ (as defined above) restricting to $H^{d_i \lambda} dt$ on the cylindrical ends (in other words, in terms of the discussion above, we are now associating the Hamiltonian $H_i:= H^{d_i \lambda}$ to the $i$th cylindrical end). We further require that outside of a compact set on $\hatX$, 
    \begin{equation} \label{eq: simplepert} 
        K = H^{\lambda} \beta,
    \end{equation}
    where $H^{\lambda}$ is the choice of admissible Hamiltonian (in the sense of Definition \ref{def:admissibleham}) used to define $HF^*(\hatX; H^{\lambda})$.
\end{itemize}

The most general form of Floer's equation that we will be studying in this paper is: 
 \begin{align} \label{eq:generalFloer}
\left\{
\begin{aligned}
 & u \colon \Sigma \to \hatX, \\
 & (d u - X_{K})^{0,1} = 0.
\end{aligned}
\right.
\end{align} 

To such $u$ we can associate the {\em geometric energy}
\begin{equation}
    \Egeo(u) :=  \frac{1}{2}\int_{\Sigma}|| du - X_{K}||^2
 \end{equation} 
as well as the {\em topological energy}
\begin{equation}\label{eq:Etop}
    \Etop(u) = \int_{\Sigma} u^* \omega - d(u^*K).
\end{equation}

For solutions $u$ of \eqref{eq:generalFloer}, we have a relationship
\begin{equation}\label{generalenergyrelationship}
    \Egeo(u)=\Etop(u)+ \int_{\Sigma} u^*\Omega_K, 
\end{equation} 
where the curvature $\Omega_K$ of a perturbation 1-form $K$ is the exterior derivative of $K$ in the $\Sigma$  direction; this is a 2-form on $\Sigma$ with values in $C^{\infty}(\hatX)$ or equivalently a section of $\pi^* (\Lambda^2 T^*\Sigma) \to \Sigma \times \hatX$, where $\pi: \Sigma \times \hatX \to \Sigma$ is the projection (note $u^*\Omega_K$ above denotes pullback by the graph of $u$). 
If we momentarily assume that $K$ is of the form \eqref{eq: simplepert} and $H^{\lambda}\geq 0$ on all of $\hatX$, then $\Omega_K=H^{\lambda}d\beta$. The non-negativity of $H^{\lambda}$ and subclosedness of $\beta$ would therefore imply that $\int_{\Sigma} \Omega_K \leq 0$, giving an inequality (under these assumptions)
\begin{equation}\label{eq:monotonicidentityonend}
        \Egeo(u) \leq \Etop(u).
\end{equation}
In particular, if $\Etop(u) \leq 0$, it would follow that $du = X_K$ everywhere.
Now under our assumptions on $K$ and $H^{\lambda}$ (see third bullet point above), \eqref{eq: simplepert} and $H^{\lambda}>0$ are only required to hold outside
of some compact set $R$ of $\hatX$. This nevertheless implies one can a priori control geometric energy in terms of topological energy (which in turn can be computed in terms of the actions of asymptotics via Stokes' theorem), as we now recall (compare e.g., \cite[\S 5c]{Seidel:2009fk}).
First for any solution $u$ of \eqref{eq:generalFloer}, if  $\hat{\Sigma}$ is the subset of the domain of $u$ mapping outside $R$, then $\int_{\hat{\Sigma}} u^*\Omega_K \leq 0$, hence $u|_{\hat{\Sigma}}$ satisfies \eqref{eq:monotonicidentityonend}. For $u$ itself, we note first that $\Omega_K = 0$ near the cylindrical ends (where $K$ is closed in the $\Sigma$ direction),  hence its support lies in $\bar{\Sigma} \times \hatX$ where $\bar{\Sigma} \subset \Sigma$ is some compact subset independent of $u$. Also, the restriction of $\Omega_K$ to 2-forms on $\bar{\Sigma}$ with values in $C^{\infty}(R)$ is (by compactness of $R$ and $\bar{\Sigma}$) necessarily bounded above (by say $A d\mathrm{vol}_{\Sigma}$ for some $A > 0$ and a nowhere vanishing volume form $d\mathrm{vol}_{\Sigma}$). It follows that, using the notation $\hat{\Sigma}^c:= \Sigma \setminus \hat{\Sigma}$ for the region mapping to $R$, the curvature satisfies
\begin{equation*}
    \begin{split}
        \int_{\Sigma} u^*\Omega_K &= \int_{\hat{\Sigma}} u^*\Omega_K +  \int_{\hat{\Sigma}^c} u^*\Omega_K \leq \int_{\hat{\Sigma}^c} u^*\Omega_K\\
        &= \int_{\hat{\Sigma}^c \cap \bar{\Sigma}} u^*\Omega_K  \ \ \textrm{as $\Omega_K = 0$ outside $\bar{\Sigma}$}\\
        & \leq \int_{\hat{\Sigma}^c \cap \bar{\Sigma}} A d \mathrm{vol}_{\Sigma} \ \ \textrm{as $(\Omega_K)_{(p, u(p))} \in \Omega^2(\hat{\Sigma}^c \cap \bar{\Sigma}; C^{\infty}(R))$ for $p \in \hat{\Sigma}^c \cap \bar{\Sigma}$}
    \end{split}
\end{equation*}
is bounded above independently of $u$ by $C:= A \cdot \mathrm{vol}(\bar{\Sigma})$ (note that while $\hat{\Sigma}^c$ depends on $u$, $\bar{\Sigma}$ is independent of $u$ and hence so is this bound).
Thus we obtain a bound on the geometric energy of a solution to
Floer's equation in terms of the topological energy and a constant $C$ depending
only on the perturbation 1-form $K$:
\begin{equation}\label{eq:monotonicidentity}
        \Egeo(u) \leq \Etop(u) + C
\end{equation}
 
We begin by defining the pair of pants product as an operation
\begin{align} \label{eq: prod} (-\cdot -):HF^*(H^{\lambda})\otimes HF^*(H^{\lambda}) \to HF^*(H^{2\lambda}); \end{align}
 to do so we specialize to the case where $\Sigma$ is the pair of pants, viewed as a sphere minus three points. 
Labeling the punctures of $\Sigma$ by $z_1$, $z_2$ and $z_{0}$, we equip
$\Sigma$ with positive cylindrical ends $\epsilon_+^1$, $\epsilon_+^2$ around $z_1$ and $z_2$ and a negative
end $\epsilon_-^0$ around $z_{0}$. We count solutions to Floer's equation \eqref{eq:generalFloer} such that
\begin{align} \label{eq:FloerS}
\left\{
\begin{aligned}
 & u \colon \Sigma \to \hatX, \\
 & \lim_{s \to -\infty} u(\epsilon_-^0(s, -)) = x_0\\
 & \lim_{s \to \infty} u(\epsilon_+^1(s, -)) = x_1\\
& \lim_{s \to \infty}u(\epsilon_+^2(s, -)) = x_2 \\
\end{aligned}
\right.
\end{align} 
The operation \eqref{eq: prod} is compatible with the continuation
maps $\mathfrak{c}_{\lambda_1,\lambda_2}$, and hence induces a product
$SH^*(\bar{X}) \otimes SH^*(\bar{X}) \to SH^{*}(\bar{X})$ which we call the pair of pants product.

We will need to consider certain parameterized operations on symplectic cohomology as well. In the simplest case, consider the cylinder $\mathbb{R}_s\times S^1_t$. We will consider operations parameterized by a value $r \in S^1$. Set $$ H^{(r)}_{t}(x)=H(t-r,x), \quad J^{(r)}_{t,x}= J_{t-r} $$
Choose a pair $(J_{s,t,r},H_{s,t,r})$ such that along the cylindrical ends 
\begin{align} \label{eq:FloerS1}
\left\{
\begin{aligned}
    & (J_{s,t,r},H_{s,t,r})=(J^{(r)},H^{(r)}) & \textrm{ for all }s \ll 0\\
& (J_{s,t,r},H_{s,t,r})=(J,H) & \textrm{ for all }s \gg 0.
\end{aligned}
\right.
\end{align}

We will denote the rotated Hamiltonian orbits of $H^{(r)}_{t}$ by $x^{(r)}$.
Note that these orbits are bijection with those of $H_{t}$. Under this
correspondence, Floer trajectories of $(H^{(r)}_t, J^{(r)}_{t,x})$ with
asymptotics $x_0^{(r)}$, $x_1^{(r)}$ are naturally in bijection with
trajectories of $(H_t, J_t)$ with asymptotics $x_0$, $x_1$; 
i.e., there is a canonical isomorphism between the Floer complexes of
$H^{(r)}_{t}$ and $H_{t}$. For later use, we denote this former moduli
space, the solutions to Floer's equation for $(H^{(r)}_t, J^{(r)}_{t,x})$ with
asymptotics $x_0^{(r)}$, $x_1^{(r)}$, by $\mathcal{M}_r(x_0,x_1)$.
For generic
$(J_{s,t,r},H_{s,t,r})$, the BV operator defines an operation 
$$ \Delta: HF^{*}(\hatX, H^{\lambda}) \to HF^{*-1}(\hatX,H^{\lambda})$$
by counting solutions to the parameterized Floer equation: 
\begin{align} \label{eq:FloerSBV}
\left\{
\begin{aligned}
    & r \in S^{1},\\
 & u \colon \Sigma \to \hatX, \\
 & \lim_{s \to -\infty} u(s, -) = x^{(r)}_0\\
 & \lim_{s \to \infty} u(s, -) = x_1\\
&  \partial_s u + J_{s,t,r}(\partial_tu-X_{H_{s,t,r}})=0 
\end{aligned}
\right.
\end{align}  Furthermore this operation is compatible with the continuation
maps $\mathfrak{c}_{\lambda_1,\lambda_2}$, and hence induces an operator
$\Delta: SH^*(\bar{X}) \to SH^{*-1}(\bar{X})$ which we also call the BV operator. The BV
operator and the product together give $SH^*(\bar{X})$ the structure of a unital
\emph{BV-algebra} (meaning that there is a unit for the algebra structure,
$\Delta^2 = 0$,  and lastly there is a suitable compatibility relation between $\Delta$
and the product which we omit; see e.g., \cite{getzler1994} for a precise formula, 
which is not directly relevant to this paper).

The final general property of symplectic cohomology that we need concerns its
functoriality. Namely, let $j: \bar{W} \subset \bar{X}$ be a sub-Liouville domain. We
have a Viterbo functoriality map 
\begin{align} \label{eq:Viterbo} j^{!}:SH^*(\bar{X})
    \to SH^*(\bar{W})  
\end{align} 
which respects the BV structures on both sides:
\begin{lem} \label{lem:Viterboalgebraic}
    The Viterbo functoriality map $j^{!}$ is a morphism of unital BV algebras (and in particular preserves the BV operator and the product).
\end{lem}

\def\o{\omega}
\def\e{\epsilon}
\def\D{ {\mathbb D}}
\def\a{\alpha}
\def\b{\beta}
\def\r#1{\mathrm{#1}}
\def\rb#1{\mathrm{\mathbf{#1}}}
\def\c#1{\mathcal{#1}}
\def\mc#1{\mathcal{#1}}
\def\ol#1{\overline{#1}}
\def\k{\kappa}
\def\d{\Delta}
\def\M{\mathcal{M}}
\def\xo{(X,\o)}
\def\ainf{A_\infty}
\def\f{\c{F}}
\def\sh{SH^*(M)}
\def\bd{\partial}
\def\y{\mc{Y}}
\def\endo{\mathrm{End}}
\def\id{\mathrm{id}}
\def\mf#1{\mathfrak{#1}}
\def\D{\mathbf{D}}

\section{Complements of normal crossings divisors}\label{sec:ncgeometry}
\subsection{Nice symplectic and almost-complex structures}\label{subsec:nicesymplectic}

\begin{defn}
    \label{def:logsmooth} 
    A {\em log-smooth compactification} of a smooth complex $n$-dimensional
    affine variety $X$ is a pair $(M,\mathbf{D})$ with $M$ a smooth, projective
    $n$-dimensional variety and $\mathbf{D} \subset M$ a divisor satisfying
\begin{align}
 &X = M \backslash \mathbf{D};\\
 &\textrm{The divisor $\mathbf{D}$ is normal crossings in the strict sense, e.g.,}\\
 &\nonumber\mathbf{D} := D_1 \cup \cdots \cup D_i \cup \cdots \cup D_k\ \textrm{where $D_i$ are smooth components of $\mathbf{D}$; and}\\
 &\label{eq:kappai}\textrm{There is an ample line bundle $\mathcal{L}$ on $M$ together with a section $s \in H^0(\mathcal{L})$ whose }\\
 &\nonumber\textrm{divisor of zeroes is $\sum_i \kappa_i D_i$ with $\kappa_i>0$}.
\end{align}
\end{defn}

It follows from Hironaka's embedded resolution of singularities--- see \cite[Main Theorem I]{MR0199184}, Theorem 1.6 of \cite{MR1440306} for a constructive proof and \cite{resofsing} for 
an expository reference--- that any smooth affine variety admits a log-smooth compactification. While this is well-known (see \cite{Seidel:2010fk}*{Lemma 4.4}), we briefly recall the argument as we will need similar arguments in the proof of Lemma \ref{defn: goodcomp}. The precise version of Hironaka's result that we will use is as follows: 

\begin{thm}(Embedded resolution of singularities)  \label{thm:Hironaka} Let $Y \to W$ be a closed embedding of a reduced variety (= finite-type scheme over $\mathbb{C}$) into a smooth variety $W.$ Then there exists a birational morphism $\pi_r: W' \to W$ given by a sequence of blow ups at smooth centers:
\begin{align} \label{eq:stages} W' := W_r \to \cdots \to W_j \to \cdots  \to W_1 \to W \end{align} 
 such that letting $E_j$ denote the exceptional divisors of $\pi_j: W_j \to W$ (so $E_r$ is the exceptional locus of $W' \to W$):  \begin{itemize} \item \begin{enumerate} \item the proper transform $Y'$ of $Y$ in $W'$ is smooth \item $Y'$ and $E_r$ simultaneously have only normal crossings.\footnote{This means that $E_r$ is normal crossings and meets $Y'$ transversely or equivalently that $E_r \cup Y'$ is a (not-necessarily equidimensional) normal crossings scheme (see also the fourth paragraph of page 335 of \cite{resofsing}).} \end{enumerate} \item For each $j$, letting $Y_j \subset W_j$ denote the proper transform of $Y$ to $W_j$, we have that either: \begin{enumerate} \item the blow-up center $N_j$ lies in the singular locus of $Y_j$ or \item $Y_j$ is smooth and $N_j \subset E_j \cap Y_j.$ \end{enumerate} \end{itemize} \end{thm} 

We will also need the following statement (both for the argument below and in Section 6): 

\begin{lem} \cite[Exercise 7.14(b) of Chapter 2]{Hartshorne} \label{lem:randomhart} Let $W$ be a smooth variety with ample line bundle $\mathcal{L}$. Let $\pi: W' \to W$ denote the blow up of $W$ along a smooth subvariety $N$ and let $E$ denote the exceptional divisor. Then there exists an $n_0$ such that $\pi^*(\mathcal{L})^{\otimes n} \otimes \mathcal{O}(-E)$ is (very) ample for all $n>n_0.$   \end{lem}
 
We are now in a position to prove: 
\begin{lem} \label{lem:smoothcomp} Any smooth affine variety $X$ admits a log-smooth compactification $(M,\mathbf{D}).$ \end{lem} 
\begin{proof} Compactify $X \subset \mathbb{C}^n$ inside of $\mathbb{P}^n$ to a variety $\tilde{M}$. Notice that $\tilde{M} \setminus X$ supports an ample divisor $\tilde{D}$ (the scheme-theoretic intersection of $\tilde{M}$ with the hyperplane at infinity). Apply embedded resolution of singularities (first to the variety $\tilde{M}$ to obtain a smooth $M$ and then to the compactifying divisors inside of $M$) to obtain a normal crossings compactification $(M,\D)$ of $X$. 
 
 It remains to show that the pair $(M,\D)$ satisfies \eqref{eq:kappai}. Let $M_1$ denote the first stage of the embedded resolution of singularities (i.e. the proper transform of $\tilde{M}$ inside of the resolution $W_1$). Then by Lemma \ref{lem:randomhart}, for $n$ sufficiently large, $n\pi^*(\tilde{D})-E_1$ defines an effective ample divisor on $W_1$ (it is effective because the pull-back of the section defining $\tilde{D}$ vanishes along $E_1.$). The restriction to $M_1$ therefore defines an effective ample divisor supported on $M_1 \setminus X.$  Continuing in this way, we see inductively that $\D$ supports an effective ample divisor $F$. To get one which has positive coefficients on each component $D_i$, one can take $\sum_i D_i + mF$ which is ample for $m$ sufficiently large by \cite[Exercise 7.5(b) of Chapter 2]{Hartshorne}. \end{proof} 
For all pairs $(M, \mathbf{D})$ in this paper, we will assume that the canonical
bundle is supported on $\D$
\begin{align} 
    & 
    \wedge^n_{\C} T^*M
    \cong \mathcal{O}(\sum_{i=1}^k -a_iD_i) \textrm{ for some integers $a_i \in \Z$}, \label{eq:volform} 
\end{align}  
and choose a meromorphic volume form $\Omega_{M, \mathbf{D}}$ on $M$ which is non-vanishing on $X$ and has poles of order $a_i$ along $D_i$ as in \eqref{eq:volform}. 

Given a subset $I \in \lbrace 1,\cdots, k \rbrace$ define
\begin{equation}
    D_I:= \cap_{i \in I} D_i;
\end{equation}    
we refer to the associated open parts of the stratification induced by
$\mathbf{D}$ as 
\begin{equation}
    \DIo = D_I \backslash \cup_{j \notin I} D_j
\end{equation}
By convention, we set $D_{\emptyset} := M$ and $\mathring{D}_{\emptyset} = X$.
Denote by $\mathfrak{i}: X \hookrightarrow M$ the natural inclusion map. 
\begin{defn} Let $X$ be a symplectic manifold equipped with a one-form $\theta_X$ such that $\omega_X=d\theta_X$. 
Let $Y$ denote the $\omega_X$-dual of $\theta_X$. We say that $(X,\theta_X)$ is a {\em finite-type convex symplectic manifold} if there exists an exhausting function $f_X: X \ra [0, \infty)$ together with a $c_0 \in \mathbb{R}^{\geq 0}$ such that $df_X(Y) > 0$ over all of $ f_X^{-1}[c_0,\infty).$ 
 
\end{defn}

See \cite{McLean:2012ab}*{\S A} for a comprehensive survey of these structures
including the notion of deformation of these structures that we will use. Note
in particular that any finite-type convex symplectic manifold $X$ has a
well-defined symplectic cohomology group (in the sense of the previous section)
defined as follows:  
pick an $f_X$ as above, note that the sub-level set $\bar{X}:= f_X^{-1}([N,\infty))$ is a
Liouville domain for any $N > c_0$ (the completion $\hat{\bar{X}}$ of which is
convex deformation equivalent to $X$ and in particular independent of choices
up to deformation equivalence), and define  
\begin{equation}
    SH^*(X):= SH^*(\bar{X}) 
\end{equation}
for any such $N > c_0$.

Any affine variety has a canonical (up to deformation equivalence) structure of
a finite-type convex symplectic manifold constructed as follows:
\vskip 10 pt
\begin{example}[Stein symplectic structure] \label{ex:stein} Pick a holomorphic embedding $i: X \to \C^N$ and
    equip $X$ with \begin{itemize}
        \item the one-form $\theta=i^*(\sum_{k=1}^N \frac{r_k^2}{2}d\theta_k)$
    (in terms of polar coordinates $(r_k,\theta_k)$ on $\mathbb{C}$); and 
    \item the
    exhausting function $f_X= i^*(\sum_k|z_k|^2)$. 
    \end{itemize}
    Up to deformation equivalence, the resulting convex symplectic structure is independent of
    choice of $i: X \hookrightarrow \C^N$.
\end{example} 
There is a different construction of a (deformation equivalent) convex symplectic
structure which is more natural from the point of view of normal crossings
compactifications.  
\begin{example}[Logarithmic symplectic structure]\label{ex:logsymplectic}  Let
    If $(M,\mathbf{D})$ be a log-smooth compactification of $X$ as above, and $s \in H^0(\mc{L})$ a section of a line bundle cutting out $\mathbf{D}$ as in Definition \ref{def:logsmooth}, we can equip $X$ with  
    \begin{itemize}
        \item the one-form $\theta=d^c\operatorname{log}|s|$. 

        \item the exhausting function is given by
        $f_X= -\operatorname{log}|s|$,
    \end{itemize}
    where $| \cdot |$ is any choice of positive Hermitian metric on the line bundle
    $\mc{L}$ (once more, the result is independent of $|\cdot |$ up to deformation
    equivalence).  
\end{example}
The deformation equivalence of the above two convex symplectic structures for a given $X$ is proven in \cite{McLean:2012ab}*{Lemma 5.18} (compare \cite{Seidel:2010fk}*{Lemma 4.4}). This implies in particular that the deformation class of convex symplectic structure on $X$ induced by a log-smooth compactification is independent of $(M,\D)$.
Here we recall a further deformation of the convex
symplectic structure on $X$, due to McLean \cite{McLean:2012ab} (see
\cite{Seidel:2010fk} for the case $\dim_{\C} X = 2$), with ``nice" properties
at infinity with respect to a given compactification $(M, \mathbf{D})$.  
To begin, after a deformation (as symplectic submanifolds) we assume that the smooth components of
$\mathbf{D}$ intersect orthogonally in $M$:
\begin{thm}[\cite{McLean:2012ab}*{Lemma 5.3, 5.15}] \label{thm:makingdivisorsorthogonal}
    There exists a deformation of the divisors $\mathbf{D}$ (through symplectic divisors) such that they
    intersect orthogonally with respect to the symplectic structure on $M$.
    This does not change the symplectomorphism type of the complement $X = M
    \backslash \mathbf{D}$, or the deformation class of its convex symplectic
    structure.  
\end{thm}
\begin{defn}[\cite{McLean:2012ab}*{Lemma 5.14}] \label{defn:nice}
    A convex symplectic structure $(W,\theta)$ constructed from a log-smooth compactification
    $(M,\mathbf{D})$ is {\em nice} if (after first implicitly applying a deformation to make the divisors symplectically orthogonal as in Theorem \ref{thm:makingdivisorsorthogonal}) 
          there exist tubular neighborhoods $U_i$ of $D_i$ with symplectic
          projection maps 
          \[
            \pi_i: U_i \ra D_i
           \] 
           such that on a $|I|$-fold intersection of tubular neighborhoods 
           \[
           U_I: = \cap_{i \in I} U_i = U_{i_1} \cap \cdots \cap U_{i_{|I|}},\] 
           iterated projection $|I|$ times 
           \[
               \pi_I:= \pi_{i_1} \circ \pi_{i_2} \circ \cdots \circ \pi_{i_{|I|}}: U_I \ra D_I
           \]
           is a symplectic fibration with structure group $U(1)^{|I|}$ and
           with fibers symplectomorphic to a product of standard symplectic discs $\prod_{i
           \in I} \D_\e$ of some radius $\epsilon$ and such that $\theta$ restricts to
           \begin{equation}\label{thetafiber}
               \sum_{i \in I} (\frac{1}{2} r_i^2 - \kappa_i) d\varphi_i,
           \end{equation}
           on each fiber, where $(r_i, \varphi_i)$ are standard polar coordinates for  
           $\D_\e$ and the $\kappa_i$ are as in \eqref{eq:kappai}. 
           Moreover, each $\pi_i$ for
           $i \in I$ is fiber-preserving, sending
           \[
           \prod_{j \in I} \D_\e \ra \prod_{j \in I | j \neq i} \D_\e.
           \]
\end{defn}
\begin{thm}[\cite{McLean:2012ab}*{Theorem 5.20}]  \label{thm:McLean520}
    Given an affine variety $X$ equipped with a log-smooth compactification
    $(M, \D)$ as above, there exists a convex symplectic structure $(W,\tilde{\theta})$
    deformation equivalent to the canonical (up to deformation equivalence)
    convex symplectic structure $(X,\theta)$ which is nice.
\end{thm}

Henceforth, we replace $(X,\theta)$ by the corresponding nice structure. 
By shrinking the tubular neighborhoods in Definition \ref{defn:nice} if
necessary, we suppose their size $\epsilon>0$ is sufficiently small so that 
    $\epsilon^2 \ll \kappa_i \textrm{ for each }i \in \{1, \ldots, k\}$.
In this setting, there is a nice Liouville domain $\bar{X} \subset X$, obtained by
smoothing the complement $M \backslash \cup_{i} U_i$ of the union of the tubular
neighborhoods around smooth divisors.
Below, we describe an explicit model of $\bar{X}$, after first recalling some additional relevant notation and consequences of having a nice structure.

Observe that the coordinates $r_i$ in Definition \ref{defn:nice} induce smooth functions
\begin{equation}
    r_i^2: U_i \to \mathbb{R}
\end{equation}
on the neighborhoods $U_i$. On the intersections $U_I$, these functions give
rise to commuting Hamiltonian $S^1$ actions. Denote by $UD$ the union
$$UD:=\cup_iU_i.$$ 
A basic important consequence of the ``nice'' property is:
\begin{lem}\label{nicehamiltonianvectorfields}
    \begin{enumerate}
        \item[(1)] The symplectic orthogonal to the tangent space of any fiber $\pi^{-1}_i(p)$
            is contained in a level set of the radial function $r_i^2$.

\item[(2)] In particular, if $I = \{i_1, \ldots, i_s\}$, any smooth function
    $f$ of the corresponding radial functions $f(r_{i_1}^2, \ldots,
    r_{i_s}^2): U_I \to \R$ has Hamiltonian vector field $X_{f}$ tangent to
the fibers of $\pi_{I}$ at points $p \in U_I \setminus_{j \notin I} U_j$ in
    $U_I$ away from a deeper stratum), with Hamiltonian vector field of following
    the form: for 
    $F = \pi_{I}^{-1}(p)$ with $p \in U_I \backslash \cup_{j \notin I} U_j$,
    \begin{equation}\label{hamvectorfield}
        (X_f)|_{F} = X_{f |_{F}} = \sum_{i \in I} 2 \frac{\partial f}{\partial (r_i)^2}\partial_{\varphi_i}.
    \end{equation}

\item[(3)] Let $I = \{i_1, \ldots, i_s\} \subset \{1, \ldots, k\}$ be any
    subset. In the associated neighborhood $U_I$ of $D_I$ away from points in
    $U_j$ for $j \notin I$, any two functions of the radial coordinates
    $f(r_{i_1}^2, \ldots, r_{i_s}^2)$, $g(r_{i_1}^2,
    \ldots, r_{i_s}^2)$ have commuting Hamiltonian vector fields: $\omega(X_f,
    X_g) = df(X_g) = -dg(X_f) = 0$.  In particular, $dr_{i_t}^2(X_f)$ = 0 for
    any $t$ and $f$ as above.
\end{enumerate}
\end{lem}
\begin{proof}
    By definition any vector field $Y$ symplectically orthogonal
    to the fibers of $\pi_i$ preserves the function $r_i^2$, i.e., $d(r_i^2)(Y)
    = \omega(X_{r_i}^2, Y) = 0$. It follows that $X_{r_i}^2$ is tangent to the
    fibers of $\pi_i$, which implies (1) (as the symplectic orthogonal to the tangent
    space of a fiber is contained in the symplectic orthogonal to $X_{r_i}^2$
which is the kernel of $d(r_i^2)$). For (2), note that (1) implies that, on $U_I \backslash \cup_{j \notin I} U_j$ where $f$ depends only on $\rho_i$ for $i \in I$, $df =
    \omega(X_f, -)$ is zero on any vector simultaneously tangent to the
    level sets of all radial functions $r_i^2$ for $i \in I$, in particular on any vector that is
    symplectically orthogonal to all fiber directions in $U_I$, hence $X_f$ is tangent to the fibers.
     It follows also that $X_f|_{F} = X_{f |_{F}}$, from which the
    computation \eqref{hamvectorfield}, and hence the remainder of (2) is immediate (given our
    knowledge of the the fiberwise symplectic form as $d$ of
    \eqref{thetafiber}). (3) can be deduced from \eqref{hamvectorfield} and
the fact that $d(r_j^2)(\partial_{\varphi_i}) = 0$.
\end{proof}

\begin{defn}\label{def:normaltorusbundles}
    Define \[
        S_I 
    \] 
    to be the unit $T^{|I|}$ bundle around each $D_I$, e.g. the
    set of points in $U_I$ where $r_i=\delta$ for some small $\delta<\epsilon$
    and $i \in I$. We will use the notation $S_i:= S_{\{i\}}$ for the one
    dimensional torus bundle over $D_i$. Let 
    \[
    \mathring{S}_I
    \] 
    denote the restriction of the $T^{|I|}$ bundle to $\mathring{D}_I$.  
\end{defn}

 Let $J$ denote a subset of $\lbrace 1,\cdots, n \rbrace \setminus I$. There is
 a codimension $|J|$ stratum of $S_I$ corresponding to the restricted torus
 bundle $S_I|_{D_{I\cup J}}$. Near these strata, we have radial coordinates
 $r_j$. We let $UD_{I,j}$ denote set of points in $S_I$ where $r_j \leq
 \epsilon$. We let $\overline{S}_I$ denote the natural closure of the open
 manifold $ S_I \setminus \lbrace \cup_{j \notin I} UD_{I,j} \rbrace $. This is
 a manifold with corners whose boundary is the region
 \[
     \partial\overline{S}_I=\bigcup_{j \notin I} (\lbrace r_j=\epsilon\rbrace \setminus  \cup_{k \neq j} \lbrace r_k< \epsilon \rbrace). 
 \]
 We adopt the convention that $S_{\emptyset} = M$
 and
 \[
     \bar{S}_{\emptyset} := \overline{M \backslash UD}
 \]

For each $I$, let $N_I$ denote the normal bundle to the stratum $D_I$ and $\NIo$ the restriction of this bundle to $\mathring{D}_I$. These two bundles have been equipped with canonical $U(1)^{|I|}$ structures and there is a canonical identification of the associated torus bundles with $S_I$ and $\mathring{S}_I$ respectively.

In order to define $\bar{X}$, we first choose a smooth function  $q:[0,\epsilon^2) \to \mathbb{R}$ satisfying:
\begin{enumerate} 
    \item There exists $\epsilon_q \in (0,\epsilon)$ such that $q(s)=0$ if and only if $s \in [\epsilon_q^2,\epsilon^2)$
\item The derivative of $q(s)$ is strictly negative when $q(s) \neq 0$.
\item $q(s)=1-s^2$ near $s=0$. 
\item There is a unique $\bar{s} \in [0,\epsilon^2)$ with $q''(\bar{s})=0$ and $q(\bar{s}) \neq 0$. 
\end{enumerate}
(Compare \cite{McLean2}*{proof of Theorem 5.16}). Define 
\begin{align}\label{Sfunction}
    F: UD &\to \mathbb{R}\\
    F(x)&=\sum_{i=1}^k q(r_i^2),
\end{align}
where we implicitly smoothly extend $q(r_i^2)$ to be 0 outside of the region
$U_i$ where $r_i^2$ is defined.

\begin{lem}\label{lem:outwardpointing}
For $\delta>0$ sufficiently small, 
the Liouville vector field associated to $\theta$ is strictly outward pointing
along $F^{-1}(\delta)$. 
\end{lem}
\begin{proof}
    This computation is very close to that in \cite{McLean:2012ab}*{Lemma
    5.17} (also, compare \cite{McLean2}*{Theorem 5.16} for a related
    calculation in the case of concave boundaries); we
    include a sketch for convenience. First, note that if $\delta$ is small,
    then $F^{-1}(\delta) \subset UD$ and $F^{-1}(\delta) \cap \D= \emptyset$;
    the latter condition implies that
    the Liouville vector field is
    defined at all points of $F^{-1}(\delta)$. We need to show that at any
    $x \in F^{-1}(\delta)$, $dF(Z) = \omega(X_{F}, Z) = -\theta(X_{F}) > 0$,
    where $X_F$ denotes the Hamiltonian vector field of $F$. Let $X_{q(r_i^2)}$
    denote the Hamiltonian vector field of $q(r_i^2)$, so $X_{F} = \sum_{i=1}^k
    X_{q(r_i^2)}$. The condition that $(X,\theta)$ is nice implies (see Lemma
    \ref{nicehamiltonianvectorfields}) that
    $X_{q(r_i^2)}$ is 
    tangent to the fibers $\D_{\epsilon}$ of $\pi_i$ with the following form
    \[
        (X_{q(r_i^2)})|_{\D_{\epsilon}} = X_{q(r_i^2)|_{\D_{\epsilon}}} = -\frac{1}{2} q'(r_i^2) \frac{\partial}{\partial {\varphi_i}}.
    \]
    Applying $\theta$, we have that
    \begin{equation}\label{eq:thetaX}
        \theta(X_{q(r_i^2)}) = (\frac{1}{2} r_i^2 - \kappa_i)d\varphi_i(X_{q(r_i^2)}) = \frac{1}{2} q'(r_i^2) ( \kappa_i - \frac{1}{2} r_i^2)
    \end{equation}
    (this extends smoothly by zero outside $U_i$).
    Note that $q'(r_i^2) \leq 0$ with equality exactly when $r_i^2 \geq
    \epsilon_q^2$, and also note that $\kappa_i > \frac{1}{2} r_i^2$ for $r_i^2
    \leq \epsilon_q^2$ (recall that $\epsilon^2 \ll \kappa_i$ for all $i$). It
    follows that \eqref{eq:thetaX} is non-positive, and strictly negative if
    $r_i^2$ is sufficiently small.
    Hence $\theta(X_F) = \sum_{i=1}^k \theta(X_{q(r_i)^2}) =
    \sum_i \frac{1}{2} q'(r_i^2) ( \kappa_i - \frac{1}{2} r_i^2) \leq 0$ as
    well. Also condition (2) implies that at least one $q'(r_i^2) < 0$ for $x
    \in F^{-1}(\delta)$ as $ \delta \neq 0$, so we are done. 
\end{proof}
Fixing such a $\delta$, let $\partial \bar{X}=F^{-1}(\delta)$ and 
\begin{defn}\label{def:liouvilledomain}
    Define $\bar{X}$ to be the compact region in $X$ bounded by $\partial \bar{X}$ defined above.
\end{defn} 

\begin{rem}
Conditions (1)-(4) for the function $q$ are stronger than required for this
paper, but
are helpful for further analysis of $\partial \bar{X}$ in the sequel \cite{DPSGII}.
\end{rem}

Flowing by $Z$ for some small negative time $t_0$
defines a collar neighborhood of the boundary 
$\partial{\bar{X}} \times (-t_0, 0] \cong C(\partial{\bar{X}}) \subset
\bar{X}$. Letting $X^{o}$ denote the complement of this collar in $\bar{X}$ 
\[
X^o = \bar{X} \setminus C(\partial{\bar{X}}),
\]
we obtain a Liouville coordinate on points in $X$ outside $X^o$
\[
R: X\setminus X^{o} \to \mathbb{R},
\]
defined as usual by $R(x) = e^t$, where 
$t$ is the time required to flow by $-Z$ from $x$ to the
hypersurface $\partial{\bar{X}}$.

Setting
\begin{equation}\label{eq:rho}
    \rho_i := \frac{1}{2} r_i^2,
\end{equation}
$F$ can be written as $\sum_{i=1}^k q(2\rho_i)$.
Our first claim is that
\begin{lem}\label{lem:Rdependsonrho}
    The Liouville coordinate $R$ defined above is a function of $\rho_1,\cdots,\rho_k$ on $X \setminus \mathring{\bar{X}} = \{R \geq 1\}$,  i.e.,
    $R=R(\rho_1,\rho_2,\cdots,\rho_k)$ (or rather $R$ depends smoothly on whichever subset of $\{\rho_1, \ldots, \rho_k\}$ is defined near a given point in $(X \setminus \mathring{\bar{X}}) \subset UD$).
\end{lem}
\begin{proof}
    Let $H_I = F^{-1}(\delta) \cap (U_I \setminus \cup_{j \notin I} U_j)$ be
    the portion of the hypersurface $\partial \bar{X}$ which lies in $U_I
    \backslash \cup_{j \notin I} U_i$. 
    We denote by $V_I
    \subset U_I \cap (X\setminus \mathring{\bar{X}})$ the subset of points $x \in  U_I \cap X$ for which the time $-\log
    R(x)$ flow by $Z$ from $x$ lands in the subset $H_I \subset \partial \bar{X}$. 
    In $U_I \cap X$, the vertical component of $Z$ with
    respect to the symplectic fibraton $\pi_I: U_I \to D_I$ is given by 
    \begin{align} 
        \label{eq: Liouvillevec}
        Z_{vert}= \sum_{i\in I}(\rho_i-\kappa_i)\partial_{\rho_{i}};
    \end{align} 
    as $\kappa_i > \rho_i$ in $U_I$, each $\rho_i$ (where defined) decreases, resp. increases
     along positive, resp. negative flowlines of $Z$ (and $Z_{vert}$).  It follows that the
     positive flowlines of $Z$ (and $Z_{vert}$) from $H_I$ continue to lie
    in $U_I$ and negative flowlines from $x \in U_I \cap X$
    can only leave $U_i$ for $i \in I$ but cannot enter $U_j$ for $j \notin I$.
    Hence any point $x \in \{R \geq 1\}$ which flows by $-Z$ to $H_I$ must be
    contained in $V_I \subset U_I$, and every point $x \in U_I \cap \{R \geq
    1\}$ is contained in $V_J \subset U_J$ for some $J \subseteq I$, i.e., $\{R
    \geq 1\} = \cup_{J} V_J$.
    It therefore suffices to show that in a neighborhood of $x \in V_I \subset U_I$,
    $R$ depends smoothly on $\rho_i$ for $i \in I$ (note if $x \in V_J \cap U_I$ for
    $J \subseteq I$, it will follow that $R$ depends smoothly on $\{\rho_j |  j
    \in J\}$, hence  $\{\rho_i |  i \in I\}$ if $J \subseteq I$).

    Let $F_I: U_I \to \R$ denote the function $ \sum_{i \in I}
    q(2\rho_i)$ and we view $F_I^{-1}(\delta)$ as defining a hypersurface in $U_I$. Note that since $F_I = F$ on (a small neighborhood in
    $U_I$ containing) $(U_I \setminus \cup_{j \notin I} U_j)$, it follows that $H_I =
    F_I^{-1}(\delta) \setminus \cup_{j \notin I} U_j \subset F_I^{-1}(\delta)$. We define $V_I^{vert}$
    to be the locus of points in $U_I \cap \{R \geq 1\}$ which flow negatively
    by $Z_{vert}$ to $F_I^{-1}(\delta)$. Define 
    $R_{I,vert}(p)$ for $p \in V_I^{vert}$ to be $e^{t_{vert}}$, where $t_{vert}$ is the time it
    takes to flow along $-Z_{vert}$ from the point $p$ to $F_I^{-1}(\delta)$.
    Note evidently that $R_{I,vert}$ is a function of $\rho_{i_1}, \ldots,
    \rho_{i_{|I|}}$ (where $I = \{i_1, \ldots, i_{|I|}\}$).

    Now, observe that $Y
    = Z - Z_{vert}$ is by definition orthogonal to the fibers $\prod_{i \in I}
    \D_{\epsilon}$ of $\pi_I$, hence (by niceness) tangent to the level set of each
    $r_i$ (and thus $\rho_i$), for $i \in I$. In particular, if $\phi_X^t$ denotes the time $t$ flow of
    a vector field $X$ then for any $i\in I$ and $p \in U_I$ such that
    $\phi_Z^t(p) \in U_I$, $\phi_{Z^{vert}}(t) \in U_I$ too and
    $\rho_i(\phi_Z^t(p)) = \rho_i(\phi_{Z_{vert}}^t(p))$. Seeing as $H_I
    \subset F_I^{-1}(\delta)$, it therefore follows that $V_I \subset
    V_I^{vert}$ and $R = R_{I,vert}$ on $V_I$. Since $R_{I,vert}$ is a smooth
    function of $(\rho_{i_1}, \ldots, \rho_{i_{|I|}})$, we are done.

\end{proof}

Although $Z$ does not extend across $\D$, we observe that $R$ does:
\begin{lem}\label{lem:Rextends}
The function $R$ extends smoothly to a function $R_M : M\setminus X^{o} \to \R$.  
\end{lem}
\begin{proof}
    Letting $\rho_I := (\rho_{i_1}, \ldots, \rho_{i_{|I|}}): U_I \to \R^I$, 
it suffices to check the assertion for the functions
$R_{I,vert}(\rho_{i_{1}},\cdots, \rho_{i_{|I|}})$ defined in the previous Lemma, 
which we view as functions
defined on the domain $\Omega := \rho_I(V_{I}^{vert}) \subset 
\{ 0 < \rho_{i_j} < \frac{1}{2}\epsilon^2\} = \rho_I(U_I \cap X)$
in $\mathbb{R}_+^I
\setminus \cup_{i_{j} \in I} \lbrace \rho_{i_{j}}=0 \rbrace$; we'd like to show
this functions extends smoothly across $\rho_{i_j} = 0$ for $i_J \in I$ (or at least over the
portion of these planes meeting the closure of $\Omega$).
Let $H = F_I^{-1}(\delta) \subset \rho_I(U_I)$ 
and on $\rho_I(U_I \cap X)$ define $Z_{vert} = \sum_{j=1}^{|I|} (\rho_{i_j} - \kappa_{i_j})
\partial_{\rho_{i_j}}$ (implicitly this is $d\rho_I(Z_{vert})$); 
the function $R_{I,vert} : \Omega \to \R$
is by definition the exponential of the time it takes to flow by $-Z_{vert}$ to $H$. Observe
that $-Z_{vert}$ is naturally defined (not just on $\Omega$ as specified but) on all of
$\R^I$;\footnote{In contrast, the corresponding vector field $-Z_{vert}$ on $U_I \cap X$ does  not, of course, extend across any $D_i = \{\rho_i= 0\}$ for $i \in I$.} 
it is a radial vector field centered at  $\rho_I 
= \vec{\kappa}:= (\kappa_{i_1}, \ldots,
\kappa_{i_{|I|}})$ (a point outside $\rho_I(U_I)$), whose flow evidently exists for all time and converges at time $+\infty$ to the point $\vec{\kappa}$; its flowlines 
are (reparametrizations of) straight lines from $\vec{\kappa}$.  Similarly, $H$ extends to the global hypersurface $\hat{F}_I^{-1}(\delta)$ where $\hat{F}_I = \sum_{i \in I} q(2\rho_i): \R^I \to \R$. 
Observe that $\hat{F}_I$ is weakly decreasing along positive flowlines of $-Z_{vert}$
and in fact strictly decreasing at points $x$ with $\hat{F}_I(x) = \delta$, as
\[
    d\hat{F}_I(-Z_{vert}) = \sum_{i \in I} 2 (\kappa_i - \rho_i) q'(\rho_i),
\]
each $q'(\rho_i)\leq 0 $ can only be strictly negative when $(\kappa_i - \rho_i) >0$, and there is at least one $\rho_i$ with $q'(\rho_i) < 0$ at all points of $\hat{F}_I^{-1}(\delta)$. By continuity and monotonicity, the flowline (i.e., straight line) from any $x$ with $\hat{F}_I(x) > \delta$ to $\vec{\kappa}$ (noting $\hat{F}_I(\vec{\kappa}) = 0$) must cross $\hat{F}_I^{-1}(\delta)$ at a unique point.
 In particular, we see that 
the function $R_{I,vert}$ extends smoothly across (at least) all such points
with $\hat{F}_I \geq \delta$.  Note that if some $\rho_i = 0$, $\hat{F}_I \geq 1$, so
it follows that $R_{I,vert}$ extends smoothly across the coordinate hyperplanes
as desired. 

\end{proof}

\subsection{Logarithmic cohomology}\label{subsec:logcoh}

We now define the {\em logarithmic cohomology group} of $(M, \mathbf{D}= D_1 \cup \cdots \cup D_k)$
 \begin{equation}
    \QH^*(M,\mathbf{D} )
\end{equation}
which will play a key role in this paper. Recall the notation from the previous
section $D_I$, $\DIo$, $S_I$, $\SIo$ for the stratified components of
$\mathbf{D}$, their open parts, and their unit torus normal bundles.  To start, let
$C^{BM}_{*}(\SIo,\mathbf{k})$ denote the chain-complex of smooth Borel-Moore
chains of $\SIo$ with $\mathbf{k}$-coefficients. The manifolds $\SIo$ have canonical
orientations (induced by the natural ordering of the set $I$) and we will pass freely between cohomology 
classes $\alpha \in H^*(\SIo)$ and the corresponding Poincare dual Borel-Moore homology classes. We
will use multi-index notation, meaning we fix variables $t_1, \ldots, t_k$, and
for a vector $\vec{\v} = (v_1, \ldots, v_k) \in \Z_{\geq 0}^k$, denote
\[
t^{\vec{\v}} := t_1^{v_1} \cdots t_k^{v_k}.
\]

Consider the cochain complex $C^*_{log}(M,\D)$ generated by elements of the form $\alpha_c t^{\vec{\v}}$, where
$\alpha_c$ is a (co)-chain in $C^{BM}_{2n-|I|-*}(\mathring{S}_I,\mathbf{k})$ for some subset $I
\subset \{1, \ldots, k\}$, and $\vec{\v} = (v_1, \ldots, v_k)$ is a vector of
non-negative integer multiplicities {\em strictly supported} on $I$, meaning
$v_i > 0$ if and only if $i \in I$. \\

The differential on  $C^*_{log}(M,\D)$ is induced from the differential on $C^{BM}_{*}(\SIo,\mathbf{k})$.

\begin{defn} 
We will denote by 
\[
\vec{\v}_I
\]
 the vector $(v_1, \ldots, v_k)$ strictly supported on $I$ with entries,
\[
v_i := \begin{cases} 1 & i \in I \\ 0 & \textrm{otherwise}. \end{cases}
\]
In the case that $I$ consists of a single element $\lbrace i \rbrace$, we will use the notation $\vec{\v}_i$. We refer to the vectors $\vec{\v}_I$ as {\em primitive vectors}. 
\end{defn}

We can give an efficient description of $C_{log}^*(M,\D)$ as follows:
\begin{equation}
    C_{log}^*(M,\D):= 
    \bigoplus_{I \subset \{1, \ldots, k\}}  t^{\vec{\v}_I} C_{2n-|I|-*}^{BM}(\SIo,\mathbf{k})[t_i\ |\ i \in I]
\end{equation}
where $S_\emptyset = X$, and $S_I = \emptyset$ if the intersection $\cap_{i \in
I} D_i$ is empty.  We set 
\begin{align} 
    \QH^*(M,D):= H^*(C_{log}^*(M,\D)) 
\end{align} 
The cohomology $\QH^*(M,D)$ is generated by classes of the form $\alpha
t^{\vec{\v}}$. The {\em logarithmic cohomology} of $(M, \mathbf{D})$ of {\em
slope $< \lambda$}, denoted $\QH^*(M,\D)_{\lfloor \lambda \rfloor}$, is the sub
$\K$-module generated by those elements of the form $\alpha t^{\vec{\v}}$ for
some subset $I \subset \{1, \ldots, k\}$, such that 
\[
    \sum_iv_i\kappa_i< \lambda 
\]
(recall the definition of $\kappa_i$ in \eqref{eq:kappai}).

If $\lambda_1 \leq \lambda_2$, we get an inclusion: 
\[
i_{\lambda_1, \lambda_2}: \QH^*(M,\D)_{\lfloor \lambda_1 \rfloor} \to \QH^*(M,\D)_{\lfloor \lambda_2 \rfloor}
\]

The volume form $\Omega_{M, \mathbf{D}}$ from \eqref{eq:volform} gives rise to a cohomological $\mathbb{Z}$-grading on $\QH^*(M,\mathbf{D})$ via the following rule: 
\begin{equation} 
     \label{eq:loggrading}
     |\alpha t^{\vec{\v}}| =|\alpha|+2\sum_i (1-a_i) v_i.
\end{equation}

\begin{rem} 
For a general pair $(M,\D)$, an elaboration of standard Morse-Bott arguments allow one to produce a spectral sequence  
\begin{equation}\label{eq:logSS}
    \QH^*(M,\D) \Rightarrow  SH^*(X) 
\end{equation}
from logarithmic cohomology converging to symplectic cohomology \cite{DPSGII,
McLeanInPreparation}.  It is likely that the differential on the higher pages
can be described in terms of certain relative or log Gromov-Witten invariants.
Often (for instance in the examples of topological pairs discussed in \S \ref{subsec:topological}),
one can use elementary considerations
involving homotopy classes of generators or indices
to conclude that \eqref{eq:logSS} degenerates. However, this method does not
enable one to compute the ring structure on $SH^*(X)$, whereas the Log PSS
morphism does \cite{DPSGII}.
\end{rem}

\begin{defn} 
    Given a class $A \in H_2(M,\mathbb{Z})$, we will denote by $A \cdot \mathbf{D}$ the multivector $[A\cdot D_i] \in \mathbb{Z}^k$. \end{defn}

We finish this section by defining a BV type operator on logarithmic cohomology.
Notice that for any stratum $D_I$, there is a natural $T^{|I|}$ action on
$\mathring{S}_I$.  Any $\vec{\v}$ which is strictly supported on $I$ gives rise
to a natural homomorphism $\phi_{\vec{\v}}: S^1 \to T^{|I|}$ defined by
$\phi_{\vec{\v}}(e^{i\theta})=e^{i\vec{\v}\theta}$. The homomorphism
$\phi_{\vec{\v}}$ in turn induces an action 
\begin{equation}\label{gammaV}
    \Gamma_{\vec{\v}}: S^1 \times \mathring{S}_I \to \mathring{S}_I 
\end{equation}
such that $\Gamma_{\vec{\v}}$ is proper. So
we may define the pushforward on BM homology $$\Gamma_{\vec{\v},*}:
H^{BM}_{*}(S^1 \times \mathring{S}_I) \to H^{BM}_{*}(\mathring{S}_I) $$
\begin{defn} Denote the natural fundamental class on $S^1$ by
    $\epsilon^{\vee}$. Given a class $\alpha t^{\vec{\v}} \in \QH^*(M,\D)$, 
define 
\begin{equation}\label{logBVoperator}
    \Delta(\alpha t^{\vec{\v}})=\Gamma_{\vec{\v},*}(\epsilon^{\vee}\otimes\alpha)t^{\vec{\v}}
\end{equation}
\end{defn}

\subsection{Topological pairs and tautologically admissible classes}
\label{subsec:topological}

We say that a class $A \in H_2(M)_{\omega}$ is spherical if it is in the image of $\pi_2(M) \to H_2(M,\mathbb{Z}).$ Let 
\begin{equation}H_2(M)_{\omega}\end{equation}
denote the classes $A \in H_2(M,\mathbb{Z})$ such that $\omega(A)>0$.

\begin{defn}\label{def:tautologicallyadmissible} 
We say that a class $\alpha t^{\vec{\v}} \in \QH^*(M,\D)$ is {\em
tautologically admissible} if for every spherical class $A \in H_2(M)_\omega$,
\begin{align} 
     A \cdot D_i > v_i
\end{align}
 for some $i \in \lbrace 1, \cdots ,k \rbrace.$
\end{defn}

\begin{defn} \label{def:topologicalpair}
    A pair $(M,\D)$ is \emph{topological} if every class $\alpha t^{\vec{\v}}$
    is tautologically admissible in the sense of Definition
    \ref{def:tautologicallyadmissible}.  
\end{defn}

There are numerous examples of topological pairs. We list some examples: 
\begin{enumerate}
        \item If $\pi_2(M)=0$, any pair $(M,\mathbf{D})$ will be topological. 
        \item Whenever each smooth component $D_i$ of $\mathbf{D}$ corresponds to
            powers of the same line bundle and the number of components of
            $\mathbf{D}$, $k$, satisfies $k \geq \dim_{\C} M+1$, then $(M,
            \mathbf{D})$ is topological. (Such so-called \emph{K{\"a}hler
            pairs} were also considered in \cite{Sheridan} to simplify holomorphic
            curve theory). For example $(\P^n, \mathbf{D} = \{\geq n+1\textrm{
            generic planes}\})$ is such an example.

    \end{enumerate}

    \begin{rem}\label{rem:1ptGW}
        Definition \ref{def:topologicalpair} can be relaxed to a
        (substantially) broader class of $(M, \D)$ for whom suitable 
        1-point relative Gromov-Witten moduli spaces(in $M$ rel $\D$) are empty for certain complex structures preserving $\D$ (see Section 5 of \cite{DPSGII}). 
       The key compactness result (Lemma \ref{lem:compactness}) about log PSS moduli spaces continues to hold for this more
        general notion of topological pair; we 
        have opted for technical simplicity to use Definition
        \ref{def:topologicalpair} to avoid discussion of bubbling in the relative setting.
    \end{rem}

\subsection{Admissible classes}\label{subsec:admissible}
In this subsection, we define a subspace of admissible classes in logarithmic
cohomology described in the introduction. 
\vskip 5 pt
We will sometimes (but not always) impose the following
assumption (on specific subsets $I$ which will be clear from context): 

\textbf{Assumption} (A1): All spherical classes $A \in H_2(M)_\omega$ satisfying 
 \begin{equation} 
        \label{eq: intersection} A \cdot \mathbf{D}=\vec{\v}_I 
    \end{equation}
are indecomposable. \vskip 5 pt

\begin{defn} \label{defn:admissi}
    We say that a class $\alpha t^{\vec{\v}} \in \QH^*(M,\D)$ is {\em admissible} if either
\begin{itemize}
\item $\alpha t^{\vec{\v}} \in \QH^*(M,\D)$ is tautologically admissible; or
\item  $\vec{\v}=\vec{\v}_I$ is primitive, Assumption (A1) holds and for any spherical class $A \in H_2(M)_{\omega}$ such that $A \cdot \D \neq \vec{\v}_I$, there exists an $i  \in \lbrace 1, \cdots ,k \rbrace$ such that $ A \cdot D_i > v_i$.
\end{itemize}
\end{defn}
We will similarly refer to (tautologically) admissible cocycles
$\alpha_ct^{\vec{\v}} \in C^*_{log}(M,\D)$ which are cocycles which represent
(tautologically) admissible classes. We define 
\begin{equation}
    \QH^*(M,\D)^{ad} \subset \QH^*(M,\D)
\end{equation} 
to be the $\K$-module generated by admissible classes. 

\begin{rem} \label{rem:randomtautremark}
For a topological pair $(M, \D)$, of course $\QH^*(M,\D)^{ad} = \QH^*(M,\D)$. Also, in a slight abuse of terminology, we will use the terminology ``primitive admissible" to refer exclusively to classes satisfying the second bullet point of Definition \ref{defn:admissi} (even though the vectors $\vec{\v}$ associated to tautologically admissible classes can of course also be primitive). 
\end{rem}

\subsection{A Gromov-Witten invariant}\label{subsec:gwinvarint}
In this section, we define an element we call the {\em obstruction class} 
\begin{equation}
    GW_{\vec{\v}}(\alpha) \in H^*(X) 
\end{equation}
associated to any admissible class $\alpha t^{\vec{\v}}$. The element $GW_{\vec{\v}}(\alpha)$ can
only be non-zero when $\alpha t^{\vec{\v}}$ is not tautologically admissible, 
which by hypotheses can only happen when $\vec{\v} = \vec{\v}_I$ is primitive. 
\begin{rem}
    The setup we choose here is by no means the most general possible, but is the
    simplest setup which avoids technical issues like multiply-covered curves and
    gluing curves with non-constant components in the divisor $\mathbf{D}$
    while still allowing for non-trivial applications. In particular, we expect
    there should be a more general notion of admissible class for non-primitive
    vectors for which one can define obstruction classes.
\end{rem}
In what follows, we frequently pass between cohomology $H^*(X)$ and the
Poincar\'{e} dual {\em Borel-Moore homology} $H_{2n-*}^{BM}(X)$. We begin by defining a class of almost-complex structures adapted to $(M, \D)$:
\begin{defn} \label{def:JMD}
    Let 
    \begin{equation}
        \mathcal{J}(M,\D)
    \end{equation} 
    be the Banach manifold of compatible $C^{\infty}$ almost-complex
    structures which preserve $\D$ and which are $C^\epsilon$ over (or rather,
    exponentials of $C^{\epsilon}$ infinitesimal deformations of) some fixed
    almost-complex structure, where $C^{\epsilon}$ is Floer's $C^{\epsilon}$
    norm \cite{FloerUnregularized} on sections of a vector bundle for some $\epsilon = (\epsilon_i)_{i \in \mathbb{N}}$.  Similarly,
    we let $\mathcal{J}(D_I, \cup_{j\notin
    I} D_j)$ be the space of compatible complex structures on $D_I$ which
    preserve all divisors $D_j$ for $j \notin I$ 
    (as usual, taking $I= \emptyset$ gives back $\mathcal{J}(M,\D)$).  
\end{defn}
All complex structures in this paper will be of class $\mathcal{J}(M,\D)$.

\begin{defn} 
    Given a Riemann surface $C$ and a marked point $p \in C$, we define the projectivized tangent space 
    \[S_pC:=(T_{p}C \setminus \lbrace 0\rbrace)/\mathbb{R}^+.\]
  \end{defn}

We let $(C, z_0,z_1,\cdots,z_n)$ denote the Riemann sphere with $n+1$ fixed
marked points. Each marked point $z_i$ will be enhanced with a choice of
element $\vec{r}_{z_i} \in S_{z_{i}}C$. For this paper, the most important
cases will be $n=1$, when we may take we may take $z_0=\lbrace 0 \rbrace$ and
$z_1=\lbrace \infty \rbrace$ and $n=2$, when we may take $z_0=\lbrace 0
\rbrace, z_1=\lbrace 1 \rbrace, z_2=\lbrace \infty \rbrace$. In the case $n=1$
or $n=2$, we may assume the elements $\vec{r}_{z_i}$ correspond to the positive
real direction. These additional trivializations will be used to place jet
conditions on holomorphic maps as we now explain. 

\begin{defn} \label{def:regularmap}
    We say that a $J$-holomorphic map $u:C \ra M$ is {\em $D_I$-regular} if $u(C) \not\subseteq D_i$ for $i \in I$. 
\end{defn}

Let $u: C \ra M$ be a $D_I$-regular $J$-holomorphic map such that for some $p \in C$, $u(p) \in \mathring{D}_I.$ Fix
a local complex coordinate $z$ around $p$ and an element $i \in I$. In a small neighborhood $W$ about $u(p)$, choose a system of smooth complex coordinates $y_1=u_1+\sqrt{-1}v_1, \cdots, y_n=u_1 + \sqrt{-1} v_n$ such that $u(p)=\lbrace y_1=0, \cdots y_n=0 \rbrace$, $W \cap D_i= \lbrace y_1=0 \rbrace$, and $J(\partial_{u_{1}})=\partial_{v_{1}}$ along $W \cap D_i$. Inside of $W$, write $u(z)=(f_i(z),\hat{u}(z))$, where $f_i(z)$ denotes the projection to the first factor (and $\hat{u}(z)$ denotes the projection to the remaining factors). By \cite{Ionel:2003kx}*{Lemma 3.4} (see also \cite{TehraniZing}*{Equation 6.1 and Remark 6.1}) $f_i(z)$ (which along $W \cap D_i$ is the locally defined normal component of $u$ to $D_i$) has an expansion
\begin{equation} \label{eq: firstorder}
    f_i(z)= a_i z^{\vi} + O(|z|^{\vi+1})
\end{equation}
where $\vi$, the {\it multiplicity} is a positive integer and the leading
coefficient $a_i \in \C^*$. As explained in \cite{Ionel:2004uq}*{\S 6}, \cite{CieliebakMohnke}*{\S 6} the
coefficient $a_i$ is the $\vi$-jet of the component of $u$ normal to $D_i$ at
$p$ modulo higher order terms and so intrinsically
\[
    a_i \in (T_p^*C)^{\vi} \otimes (N_i)_{u(p)}.
\]
In our case, $C = \mathbb{C}P^1$ and fixing a non-zero unit vector (or real
positive ray) in $T_p (C)$, we can think of $a_i \in
(N_i)_{u(p)} \backslash \{0\}$; define $[a_i] \in (N_i)_{u(p)} /\mathbb{R}^+$
to be the real projectivization. 
Then, the direct sum $[\oplus_{i \in I} a_i] \in S_{I,u(p)}$, lives in a fiber
of the torus bundle $S_I$ to $D_I$. Thus, if $u$ satisfies an incidence
condition as in Definition \ref{def:regularmap} with multiplicities $(v_i)_{i \in I}$,
we can define the
{\em enhanced evaluation} of $u$ at $z_0$ (following \cite{Ionel:2011fk}, \cite{CieliebakMohnke}),
taking values in $S_I$ as:
\begin{equation}\label{eq:enhancedevaluation}
    \operatorname{Ev}^{\vec{\v}}_{z_0}(u) := (u(z_0), [\oplus_{i \in I} a_i])
\end{equation}

\begin{defn}\label{defn: modsphere} 
    For any $A \in H_2(M,\mathbb{Z})$ satisfying \eqref{eq: intersection}, 
    set
    $\widetilde{\mathcal{M} }_{0,2}(M,\mathbf{D},\vec{\v}_I, A)$ to be the
    moduli space of maps $ u: (C,z_0,z_1) \to (M,\D)$ with 
    \begin{align}
        u_*([\mathbb{C}P^1])=A \\ u^{-1}(D_I)= z_0 
    \end{align}
 \end{defn}

 \begin{defn} 
     From the moduli spaces in Definition \ref{defn: modsphere} form: 
     \begin{align} 
         \widetilde{\mathcal{M}}_{0,2}(M,\mathbf{D},\vec{\v}_I) :=\bigsqcup_{A, A \cdot \D=\vec{\v}_I} \widetilde{\mathcal{M}}_{0,2}(M,\mathbf{D},\vec{\v}_I, A) \\ 
         \mathcal{M}_{0,2}(M,\mathbf{D},\vec{\v}_I):=\widetilde{\mathcal{M}}_{0,2}(M,\mathbf{D},\vec{\v}_I)/\mathbb{R} 
     \end{align} 
     where the quotient in the latter equation is by $\mathbb{R}$-translations. 
  \end{defn}

We have an evaluation map at $z_1 = \infty$ which sends $u$ to $u(\infty)$:
\[ 
    ev_{\infty}: \mathcal{M}_{0,2}(M,\mathbf{D},\vec{\v}_I) \to  M  
\] 
Let 
\begin{equation} \label{GWcyclebeforefiberproduct}
    \mathcal{M}_{0,2}(M,\mathbf{D},\vec{\v}_I)^o
\end{equation}
denote the preimage $ ev_{\infty}^{-1}(X)$. 
We will frequently use the following topological assumption 
to simplify compactness and transversality arguments.

\begin{lem} \label{lem: randomcompacti}
    Assume that Assumption (A1) holds. Then the moduli space
    $\mathcal{M}_{0,2}(M,\mathbf{D},\vec{\v}_I)^o$ is a smooth oriented
    manifold for generic $J \in \mathcal{J}(M,D)$. Moreover the map 
    \[
        ev_{\infty}:\mathcal{M}_{0,2}(M,\mathbf{D},\vec{\v}_I)^o \to X
    \] 
    is proper.  
\end{lem}
\begin{proof} 
    Consider a sequence of curves with $ev_\infty \subset K$ for some compact
    set $K \subset X$. There is no sphere bubbling in non-trivial classes
    because the classes $A$ are indecomposable. As the evaluation $ev_\infty$
    lies in $X$, it follows that the non-trivial component of the Gromov limit
    $u$ is $D_I$ regular and thus carries both distinguished marked points.
    Thus no bubbling can occur and it therefore follows that the map
    $ev_{\infty}$ is proper. The fact that the moduli space is smooth for
    generic $J$ follows from Lemma 6.7 of \cite{CieliebakMohnke}(we will review related arguments in \S \ref{section:transversality}). The orientation comes from
    the usual orientation on the moduli space
    $\mathcal{M}_{0,2}(M,\mathbf{D},\vec{\v}_I)$ arising from the fact that the
    Fredholm operator $D_u$ is homotopic to the operator $\bar{\partial}$
    together with the fact that the $\mathbb{R}$ translations preserve this
    orientation. 
\end{proof}

If $\Evo(u)$ is contained in $S_I \setminus \mathring{S}_I$, then because $u$
has intersection number zero with all of the divisors $D_j$ for $j \notin I$
then $u$ must be completely contained in $\D$ and hence $ev_{\infty}(u) \in \D$
as well. Therefore we have an evaluation 
\begin{equation}\label{eq:Evo}
    \Evo: \mathcal{M}_{0,2}(M,\mathbf{D},\vec{\v}_I)^o \to  \SIo
\end{equation}
For any class $\alpha$ in $H^*(\mathring{S}_I)$, we may define a BM homology class 
in $H^{BM}_*(X)$ via 
\begin{align} \label{eq: GWinvariantsnormal} 
    GW_{\vec{\v}_I}(\alpha)=[ev_{\infty,*}(\operatorname{Ev}_0^{\vec{\v},*}(\alpha))] \in H^{BM}_{*}(X). 
\end{align}

\begin{defn}\label{def:obstructionclass}
    For any primitive admissible class $\alpha t^{\vec{\v}_I} \in \QH^*(M,D)$,
    we refer to the class defined in Equation \eqref{eq: GWinvariantsnormal} as
    the {\em obstruction class} associated to $\alpha t^{\vec{\v}}$. We extend this
    to non-primitive admissible classes by setting $GW_{\vec{\v}}(\alpha)=0$.
\end{defn}

\begin{lem} 
    \label{lem:degeneracy} For any stratum $D_I$, let $\alpha=1 \in
    H^0(\mathring{S}_I)$ and suppose that $\alpha t^{\vec{\v}}$ is admissible.
    Then $GW_{\vec{\v}_I}(\alpha)=0$.  
\end{lem}

\begin{proof} In this case there are no constraints along $\mathring{S}_I$ so
    by definition our invariant is simply the pushforward \[
        ev_{\infty,*}[\mathcal{M}_{0,2}(M,\mathbf{D},\vec{\v}_I)^o].\] But, since the
        evaluation map $ev_{\infty}$ factors through the quotient
        $\mathcal{M}_{0,2}(M,\mathbf{D},\vec{\v}_I)^o/S^1$ it follows that 
        \[
            ev_{\infty,*}[\mathcal{M}_{0,2}(M,\mathbf{D},\vec{\v}_I)^o]=0.
        \]
\end{proof}

\def\lra{\longrightarrow}
\def\J{\widetilde{\mathcal{J}}}

\section{The Log PSS morphism}
\subsection{Symplectic cohomology from Hamiltonians on the compactification}
We continue with the setup and notation described in \S \ref{subsec:nicesymplectic}, with 
$X = M \backslash \mathbf{D}$ (with its nice convex symplectic structure in the sense of Definition \ref{defn:nice}) equipped with the (holomorphic on $X$) restriction of a fixed meromorphic volume form $\Omega_{M, \mathbf{D}}$ 
as in \eqref{eq:volform}.
In this section, we recast the construction of the symplectic cohomology $SH^*(X)$ in terms of the geometry and Floer theory of $(M, \mathbf{D})$, a point of view that will prove technically useful in constructing the Log PSS map.

Recall from \S \ref{subsec:nicesymplectic} the construction of the function $R_M: M \setminus X^o \to \R$: $R_M$ is the smooth extension to $M$ (guaranteed to exist by Lemma \ref{lem:Rextends}) of the Liouville coordinate on $X$ induced by (the exponential of the time it takes to flow by $Z$ from) the hypersurface $\partial \bar{X}$ constructed above Definition \ref{def:liouvilledomain}.
We set $R_{D}= \operatorname{min}_{\D} R_M$, and now consider Hamiltonians $h_{M}^{\lambda}: M \to \mathbb{R}^{\geq 0}$ (for $\lambda > 0$) such that 
\begin{itemize} 
    \item $h_{M}^{\lambda}$ vanishes on $\overline{{X}}$.

    \item On $M \setminus X^{o}$, $h_{M}^{\lambda} = h_{M}^{\lambda}(R_M)$ is a function of $R_M$. Moreoever, 
        for some $R_{H} \in (1, R_D)$ which is much closer to 1 than $R_D$ (so $\log R_H \ll \log R_D$) 
        we have that  
        $$ h_{M}^{\lambda}=\lambda (R_{M}-1) \quad  \forall R_M \geq R_{H} $$ 
    \item  On $M \setminus X^{o}$, $(h^{\lambda}_{M})'(R_M) \geq 0$ and $(h^{\lambda}_M)''(R_M) \geq 0$.
\end{itemize}

Choose a $\mu$ such that 
\begin{align} 
    R_H < R_{D}-\mu.
\end{align} 
Let $V_{0} \subset UD$ be the open subset containing $\D$ defined by $V_{0}:=(R_M)^{-1}(R_{D}-\mu,\infty)$. Let $V \supset V_0$ denote the slightly larger open set $V:=(R_M)^{-1}(R_H,\infty)$. 

Note that since (by Lemma \ref{lem:Rdependsonrho}) $R_M$ is a function of the $\rho_1, \ldots, \rho_k$ on $M \backslash \mathring{\bar{X}}$ (which contains $\D$) , Lemma \ref{nicehamiltonianvectorfields} implies that the Hamiltonian flow of $h_M^{\lambda}$ 
is tangent to the level sets of each $\rho_i$ in this region and hence (as the flow is zero outside this region) everywhere $\rho_i$ is defined; in particular the Hamiltonian flow preserves the divisor $\D$.
It follows that time-1
orbits of this Hamiltonian are either completely contained in $\D$ or
completely contained in $X$ (and in fact $X\setminus V$, as by construction
there are no closed orbits in $V \cap X$).  We will refer to the orbits
contained in $\D$ as {\em divisorial orbits} and denote them by
$\mathcal{X}(\D; h_M^{\lambda})$ and all other orbits by  $\mathcal{X}(X;
h_M^{\lambda})$.  We can make all of the orbits nondegenerate by a $C^2$ small
time-dependent perturbation $H_M^{\lambda}: M \to \mathbb{R}$ supported in
small neighborhoods of the orbit sets (in $X$ and near $\D$). We can moreoever
ensure this $C^2$ small perturbation $H_M^{\lambda}$ satisfies
the following key properties:
\begin{itemize} 
    \item the perturbation is disjoint from $V \setminus V_{0}$.
    \item the Hamiltonian flow of $H_M^{\lambda}$ continues to preserve each divisor $D_i$. 
\end{itemize}
\def\aa{\alpha}
We now give details of this construction, using a refined (to
satisfy the above properties) variant of the perturbation appearing in
\cite{McLean:2012ab}*{Proof of Lemma 6.8} (though note that {\em loc. cit.}
only perturbed the orbits appearing in $X$).
For every $I$ (including $I =
\emptyset$, recalling the convention that $D_\emptyset = M$ and
$\mathring{D}_\emptyset = X$), Lemma \ref{nicehamiltonianvectorfields} (in particular part (2)) implies that the orbits of $X_{h^{\lambda}_M}$ that lie in the open locus
$\mathring{D}_I$ are disjoint unions of manifolds-with-corners, corresponding
to constant orbits living in $\mathring{D}_I$ away from any deeper stratum as well as orbits
which wind a certain number of times around (and live in a neighborhood of)
some deeper stratum $\mathring{D}_J$ for $I \subset J$. Let us use the notation
$Y_{\aa}$ to refer to the connected component of orbits in $\mathring{D}_I$
winding $\mathbf{w}_j$ times around $\mathring{D}_j$ for $j \in J \setminus I$,
where $\aa = (I, J, \mathbf{w})$ is a tuple with $\emptyset \subseteq I
\subseteq J \subseteq \{1, \ldots, k\}$ and $\mathbf{w} \in \Z_{>0}^{J
\setminus I}$ is a tuple of positive integers 
(by convention $\Z_{>0}^0 = \{\mathbf{0}\}$ so we allow the case
$\aa = (I, I, \mathbf{0})$, corresponding to the constant orbits in
$\mathring{D}_I$). Concretely, $Y_{\alpha}$ can be explicitly described as the locus $((\mathring{D}_I \cap U_J) \cap \bigcap_{j \in J \setminus I} \{\frac{\partial h}{\partial \rho_j} = -\mathbf{w}_j\}) \setminus \bigcup_{t \notin J} \{x \in U_t | \frac{\partial h}{\partial \rho_t} \neq 0\}$. 
Each $Y_{\aa}$ is
a $T^{J \setminus I}$ torus bundle (of the form $\{\rho_j = d_{j, \alpha} | j
\in J \setminus I\}$ for some small constants $d_{j,\alpha}$) over a closed
submanifold-with-corners of $\mathring{D}_J$ which we call
$\overline{Y}_{\aa}$; if $I = J$ then $\overline{Y}_{\aa} = Y_{\aa}$.
In turn, $\overline{Y}_{\aa}$ is a compact manifold-with-corners obtained by removing from $D_J$ subsets of the form $\{ \rho_t < c_{t; \aa}\}$ for $t \notin J$ and some small constants $c_{t; \aa}$; 
in particular its boundary and corner strata are the points of the form $\{\rho_t = c_{t; \aa}\}$ for any $t \in \{1, \ldots, k\} \setminus J$.  
The Hamiltonian vector field $X_{h^{\lambda}_M}$ restricted to the fiber of $U_I$ over any point $x
\in Y_{\aa} \subset D_I$ is of the form $\sum_{i\in I}-\lambda_i
\frac{\partial}{\partial \varphi_i}$ with $\lambda_i>0$, which infinitesimally generates a
non-trivial rotation of the fibers fixing the points where $\rho_i=0$.  
It follows that these orbits are nondegenerate in the normal (transversal to $D_I$)
directions in $M$ over the open parts $\mathring{Y}_{\aa}$. The open locus $\mathring{Y}_{\aa}$ is moreover Morse-Bott as an orbit set in $D_I$ (For $I = J$ this is elementary and for the case $J \supset I$ we refer to Step 2 of the proof of Theorem 5.16 of \cite{McLean2}, which we note relies on the special form of the $q$ function chosen in \S \ref{subsec:nicesymplectic}); hence we conclude that $\mathring{Y}_{\aa}$ is Morse-Bott in $M$. See Figure \ref{orbitstrata} for a schematic 

\begin{figure}[h] 
    \caption{A depiction of the (Morse-Bott interior of) the orbit stratum $\mathring{Y}_{\aa}$ for the case $\alpha= (\{1\}, \{1,3\}, w_3)$ for some $w_3$. Namely, $\mathring{Y}_{\alpha}$ is the (interior) of the set of orbits in $D_{1}$ which wind $w_3$ times around $D_{1,3} = D_1 \cap D_3$. The closure $Y_{\aa}$ of $\mathring{Y}_{\aa}$ is a torus bundle over $\overline{Y}_{\aa} \subset D_J$, which is also depicted ($\overline{Y}_{\aa}$ is the image of $Y_{\alpha}$ under projection to the divisor $D_J = D_{1,3}$ in this case). In the pictured schematic, the total space is the divisor $D_I$ (= $D_1$ in this case), divisors are represented by codimension-1 real submanifolds, and normal $(T^1)^k$ bundles by $(S^0)^k$ bundles. \label{orbitstrata}}
    \centering
    \includegraphics[scale=2.0]{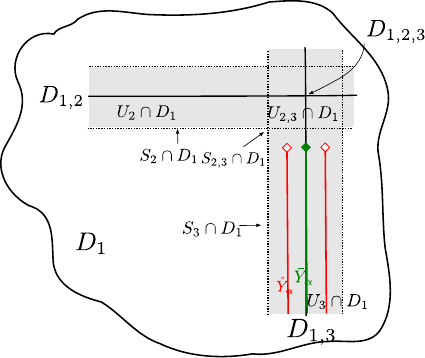}
\end{figure}

We now choose open submanifolds $\overline{Y}_{\aa}'$ of $D_J$
containing $\overline{Y}_{\aa}$ which have manifold-with-corners closures (by e.g., removing regions
where $\rho_t \leq c_{t; \aa} - \delta_{t; \aa}$ for very
small $\delta_{t; \aa}$). 
Fix disjoint isolating neighborhoods $U_{\aa}$ of $Y_{\aa}$ in (the open subset $U_J \setminus (\cup_{s \in J\setminus I} D_s \cup \cup_{t \notin J} \bar{U}_t)$ of) $M$ lying in $V_0$ if $I \neq \emptyset$ and outside $V$ if $I = \emptyset$, which are $T^J$-equivariant subfibrations of $(U_J)|_{\overline{Y}_{\aa}'}$ for some (open in $\mathring{D}_J$) $\overline{Y}_{\aa}' \supset \overline{Y}_{\aa}$, whose 
fibers with respect to $\pi_J$ a product of discs (the fibers of $\pi_j$ with $j \in I$) and
annuli (the fibers of $\pi_j$ for $j \in J \setminus I)$.  
Let $Y_{\aa}' \subset \mathring{D}_I$ denote the natural $T^{J\setminus I}$ bundle over $\overline{Y}_{\aa}'$ such that $Y_{\aa} \subset
Y_{\aa}'$. The Hamiltonian flow of $X_{h_M^{\lambda}}$ along $Y_{\aa}$ generates a local circle action which extends
to $U_J$ as a one parameter subgroup of the $T^{J\setminus I}$ symmetries. 
Let $X_{\aa}$ denote the Hamiltonian vector field of the inverse circle action and $\Delta_t$
denote the time $t$-flow of $X_{\aa}$.  Let $\widehat{h}_{\aa}$ be an outward pointing (on the boundary components of the closure)
Morse function on $Y_{\aa}'$ 
and let $h_{\aa}: S^1 \times U_{\aa} \to \mathbb{R}$ denote
the pull back of the time-dependent function $\widehat{h}_{\aa} \circ \Delta_t(x)$. Note that $h_{\aa}$ is $T^I$-equivariant (where $\aa = (I, J, \mathbf{w})$) but by construction not $T^{J \setminus I}$-equivariant. 

Fix a smaller $T^{J}$-equivariant isolating neighborhood of $Y_{\aa}$ in $M$,  $U_{\aa}' \subset U_{\aa}$
which fibers over some smaller $\overline{Y}_{\aa}'' \subset
\overline{Y}_{\aa}'$ containing $\overline{Y}_{\aa}$. 
Choose a $T^{J}$-equivariant cutoff function
$f_{\aa}$ supported in $U_{\aa}$ with $f_{\aa}=1$ on $U_{\aa}'$.  
Then for some $\delta$ sufficiently small 
set 
\begin{align}
    H_M^{\lambda} := h_M^{\lambda}+\sum_{\aa} \delta f_{\aa} h_{\aa}.
\end{align} 
Note that all of the additions to $H_M^{\lambda}$ are supported on disjoint neighborhoods which are disjoint from $V \setminus V_{0}$, implying the first desired property of this perturbation.
It follows from the $T^{I}$ equivariance of the perturbation (in every $U_I$)
that for every $x \in D_i$
\begin{align} 
    (X_{H_M^{\lambda}})_x \in T_x(D_{i}),
\end{align}  
which of course implies the second desired property of the
Hamiltonian flow of $H_M^{\lambda}$ preserving $\D$. It remains to verify non-degeneracy of $H_M^{\lambda}$.
First, a compactness argument 
shows that for 
$\delta$
sufficiently small, all of the
orbits of $X_{H_M^{\lambda}}$ lie in $U_{\aa}'$ and away from the boundary and corner strata (because in the limit as $\delta \to 0$, there are no orbits along those strata, compare \cite{McLean:2012ab}*{Proof of Lemma 6.8}). Then standard
Morse-Bott theory (compare \cite{Kwon:2016aa}*{Proposition B.4}) implies that for sufficiently small $\delta$ the set of orbits is (finite and) non-degenerate as desired.

Fix a small constant $\hbar>0$. Note that 
by taking $\delta$ sufficiently small in the above perturbation, we can (and henceforth will) assume that
\begin{equation} \label{eq:hbarinequality}
    H_M^{\lambda}>-\hbar\textrm{ everywhere}.
\end{equation}

As before we let $\mathcal{X}(\D; H_M^{\lambda})$ denote the time-1 orbits of
$H_M^{\lambda}$ contained in $\D$ (which we also call {\em divisorial orbits})
and we denote all other orbits by  $\mathcal{X}(X; H_M^{\lambda})$.
 
For what follows, recall the definition of $\mathcal{J}(M,\D)$ in Definition \ref{def:JMD}.
\begin{defn} 
    Define $\mathcal{J}(V) \subset \mathcal{J}(M, \D)$ to be the subspace of compatible almost complex structures which 
\begin{itemize} 
    \item preserve $\D$ (this is in fact automatic from saying $\mathcal{J}(V) \subset \mathcal{J}(M,\D)$), and
    \item are of contact type on the closure of $V \setminus V_{0}$.
\end{itemize} 
\end{defn} 

Define 
\begin{align} 
    \label{eq:CompactifiedFloercomplex}  CF^*(X \subset M;H_M^{\lambda}) 
            := \bigoplus_{x \in \mathcal{X}(X; H_M^{\lambda})} |\mathfrak{o}_x| 
\end{align} 

where $|\mathfrak{o}_x|$ is the $\K$-normalization of the orientation line
$\mathfrak{o}_x$ as in \eqref{eq:normalization}.

\begin{defn} 
    Let $\mathcal{J}_F(V)$ denote the space of $S^1$ dependent complex
    structures, $\mathcal{C}^\infty(S^1; \mathcal{J}(V))$.  
\end{defn}

The differential on \eqref{eq:CompactifiedFloercomplex} is defined by counting
solutions to \eqref{eq:FloerRinv} {\em in $M$} (instead of $X$) with respect to
a generic time-dependent $J_t \in \mathcal{J}_F(V)$ which additionally satisfy the topological constraint
\begin{equation}
    \label{zeromultiplicity} u \cdot \mathbf{D} = 0.
\end{equation}
In view of the following key {\em positivity of intersection
property}, the condition \eqref{zeromultiplicity} actually implies $u \cap
\mathbf{D} = \emptyset$:
\begin{lem}\label{lem:positivity}
    Let $H:=H_{s,t}$ be any Hamiltonian preserving $\mathbf{D}$, and $J:=J_{s,t}$ an
    almost complex structure (both possibly cylinder dependent) for which
    $\mathbf{D}$ is a $J$-holomorphic divisor (with normal crossings). Then, for any Floer
    trajectory $u$ of $(H,J)$ with asymptotics outside $\mathbf{D}$,  $u \cdot
    \mathbf{D} = \sum_{z \in u \cap D} (u \cdot \D)_z$, where each local
    intersection multiplicity $(u\cdot \D)_z$ is 
    $\geq 1$.
\end{lem}
The above Lemma is a consequence of positivity of intersection of
$J$-holomorphic curves $u$ with $J$-holomorphic divisors $\mathbf{D}$, along
with {\it Gromov's trick}, which allows us to import results from the analysis
of $J$-holomorphic curves to the case of Floer curves:

\begin{prop}[Gromov's trick, see e.g., \cite{McDuff:2004aa} \S 8.1] \label{prop:Gromovtrick}
    Let $(\Sigma,j)$ be a Riemann surface equipped with a 1-form $\gamma$,
    surface dependent Hamiltonian $H_{\Sigma}$ and almost complex structure
    $J_{\Sigma}$.  Then, $u:  \Sigma \ra M$ solves Floer's equation
    $(du-X_{H_{\Sigma}}\otimes \gamma)^{0,1} = 0$ if and only if $\tilde{u} = (id, u):
    \Sigma \ra \Sigma \times M$ is a $\tilde{J}$-holomorphic section with
    respect to the almost complex structure on $\Sigma \times M$
    \begin{equation}\label{section:acstructure}
        \tilde{J} = \left(\begin{array}{cc} j & 0 \\ 
            J \circ (X_{H_{\Sigma}} \otimes \gamma) - (X_{H_{\Sigma}} \otimes \gamma) \circ j & J \end{array} \right).
    \end{equation}\qed
\end{prop} 

\begin{proof}[Proof of Lemma \ref{lem:positivity}]
    If $J$ and $H$ preserve $\mathbf{D}$, then the almost complex structure
    $\tilde{J}$ appearing in Proposition \ref{prop:Gromovtrick} preserves the
    divisor $\Sigma \times \mathbf{D} \subset \Sigma \times M$.  Hence, we can
    apply positivity of intersection for ordinary $J$-holomorphic curves.
\end{proof}

Applying Lemma \ref{lem:positivity}, we see that Floer curves in $M$ which are
counted in the definition of $CF^{*}(X \subset M;H_M^{\lambda})$ actually lie
in $X$. Then \cite{Abouzaid:2010ly}*{Lemma 7.2} implies that such curves must in fact stay
away from an open neighborhood of 
$\mathbf{D}$. Hence, the compactness Theorem (Lemma \ref{lem:Floercomp})
transfers over to this setup.
In fact, more is true:
\begin{prop}  \label{prop:countincompactification}
    There is a canonical isomorphism of chain complexes 
    \begin{equation} \label{eq:compactidentif}
        CF^{*}(X \subset M;H_M^{\lambda}) \cong CF^{*}(\hatX; H^{\lambda}).
    \end{equation}
    where $\hatX$ (compare \eqref{eq:conicalstructure}) is the completion of the Liouville domain $\bar{X} \subset X$ from Definition \ref{def:liouvilledomain} and
    $H^{\lambda}$ is the family of Hamiltonians on $\hatX$ used to define $CF^*(\hatX;H^{\lambda})$ and $CF_{+}^{*}(\hatX;H^{\lambda})$ in \eqref{eq:posySH}.
\end{prop}
\begin{proof}
    The generators for each complex are identical by construction. Furthermore
    in the common subregion of $X$ and $\hatX$ where Floer curves for
    $CF^*(X\subset M, H_M^{\lambda})$ and $CF^{*}(\hatX, H^{\lambda})$
    respectively are each forced to
    stay by \cite{Abouzaid:2010ly}*{Lemma 7.2} (and also Lemma
    \ref{lem:positivity} for the former complex), $(H_M^{\lambda}, J_t)$ 
    is exactly equal to $(H^{\lambda}, J_t)$. Thus, the Floer trajectories are in bijection too.
\end{proof}

As usual, we may consider a subcomplex of $CF^*(X\subset M, H_M^{\lambda})$
generated by constant orbits in $\mathcal{X}(X; H_M^{\lambda})$, which we denote by
$CF_{-}^*(X\subset M, H_M^{\lambda})$ as well as the quotient complex
$CF_{+}^*(X\subset M, H_M^{\lambda})$. The bijection from equation
\eqref{eq:compactidentif} gives rise to an identification: 
\begin{align}
    CF_{+} ^{*}(X \subset M;H_M^{\lambda}) \cong CF_{+}^{*}(\hatX;H^{\lambda})
\end{align} 

\begin{rem} 
    To orient the reader with the several closely related Floer
    complexes that we have introduced, it may be helpful to note that the
    complex $CF^{*}(X \subset M;H_M^{\lambda})$ is neither a subcomplex nor a
    quotient complex of the Floer chain complex of $M$,
    $CF^{*}(M;H_M^{\lambda})$ because we have avoided counting Floer
    trajectories which intersect $\D$ to define our differential. In general
    $CF^{*}(M;H_M^{\lambda})$ is only defined (potentially using virtual techniques) over a
    suitable choice of Novikov ring.  
\end{rem}

One can similarly exhibit the continuation maps \eqref{eq:continuation} in terms of Floer theory in $(M,\D)$, which we now briefly sketch (as the setup adapts straightforwardly). Given a pair $\lambda_1 \leq \lambda_2$, one considers an $s$-dependent family of time-dependent Hamiltonians  $H^{s,t}_M$ and almost complex structures $J^{s,t} \in \mathcal{J}(V)$ on $M$ (for $(s,t) \in \R \times S^1$), such that each $J^{s,t}$ and $X_{H^{s,t}_M}$ preserves $\D$ and such that $(H^{s,t}_M, J^{s,t})$ coincides with the previously made choices of $(H_M^{\lambda_1}, J_t^{\lambda_1})$ (made for $CF^{*}(X \subset M;H_M^{\lambda_1})$) when $s \gg 0$ and $(H_M^{\lambda_2}, J_t^{\lambda_2})$ (made for $CF^{*}(X \subset M;H_M^{\lambda_2})$) when $s \ll 0$. We further require these choices to agree with the Hamiltonians and almost complex structures appearing in the continuation map \eqref{eq:continuation} on the sub-region common to $X$ and $\hatX$, and more generally for $H_M^{s,t}$ to be a monotone decreasing in $s$ function of $R_M$ on $M \backslash X^o$ (or rather a small perturbation thereof, where the perturbation once more avoids $V \backslash V_0$).  Counting Floer solutions in $M$ (associated to $(H^{s,t}_M, J^{s,t})$) with zero intersection number with $\D$ then defines the chain-level continuation map
\begin{equation}\label{continuationcompactification}
    \mathfrak{c}_{\lambda_1,\lambda_2}: CF^{*}(X \subset M;H_M^{\lambda_1}) \to CF^{*}(X \subset M;H_M^{\lambda_2}).
\end{equation}
An argument identical to Proposition \ref{prop:countincompactification} implies that, under the identification \eqref{eq:compactidentif}, this map \eqref{continuationcompactification} coincides with the previously defined continuation map \eqref{eq:continuation} on the chain level (similarly for the $CF_+$ variant of the continuation map). As desired this implies that the symplectic cohomology $SH^*(X)$ (and $SH^*_+(X)$) can be computed as
 \begin{align}
     SH^*(X) &\cong \lim_{\lambda} HF^{*}(X \subset M;H_M^{\lambda})\\
     SH^*_+(X)  &\cong \lim_{\lambda} HF^{*}_+(X \subset M;H_M^{\lambda}).
 \end{align}

\subsection{The classical PSS morphism}

Next, recall the classical PSS map \cite{Piunikhin:1996aa}.  
Consider the domain 
\[S = \mathbb{C}P^1 \setminus \lbrace 0 \rbrace,\] 
thought of as a punctured sphere, with a distinguished marked point $z_0 =
\infty$ and a negative cylindrical end \eqref{eq:negstrip} near $z=0$  which for
concreteness we take to be given by $$(s,t) \to e^{2\pi(s+it)} $$ The coordinates
$(s,t)$ extend to all of $S \setminus z_0$. Fix a subclosed one-form $\beta$
which restricts to $dt$ on the cylindrical end and which restricts to zero in a
neighborhood of $z_0$. To be explicit, we consider a non-negative, monotone
non-increasing cutoff function $\rho(s)$ such that
\begin{equation}\label{eq:cutoff}
    \rho(s) = \begin{cases} 0 & s \gg 0 \\ 1 & s \ll 0 \end{cases}
\end{equation}
and let \begin{equation}  
    \beta=\rho(s)dt. \label{eq:sco} 
\end{equation}

Near $z_0$, we also fix a distinguished tangent vector which points in the
positive real direction. 

\begin{defn} Let $\mathcal{J}_S(V)$ denote the subspace of (smoothly) domain-dependent complex structures on $S$, $J_S = \{J_z\}_{z \in S}$ with $J_z \in \mathcal{J}(V)$, such that 
\begin{itemize}
\item $J_S$ is independent of $z \in S$ in $V_0$ and in a neighborhood of $z_0$. 
\item Along the cylindrical end, the complex structure only depends on the \emph{time coordinate} in a neighborhood of $0$,  e.g., $(J_S)_{s,t}=J_t$ for all $s <  -K$ for some $K$.
\end{itemize}
\end{defn} 

\begin{defn}  Fix $J_S \in \mathcal{J}_S(V)$. For any
    orbit $x_0 \in \mathcal{X}(H_M^{\lambda})$, we define
    $\mathcal{M}(x_0)$ as the space of solutions to 
\[
u: S \ra M
\]
satisfying
\begin{equation} \label{eq: FloerM}
(du - X_{H_M^\lambda} \otimes \beta)^{0,1} = 0
\end{equation}
with asymptotic condition
\begin{equation} 
    \label{eq:limitsM}
\lim_{s \ra -\infty} u(\epsilon(s,t)) = x_0.
\end{equation}
\end{defn}

 For each $\alpha \in H^{BM}_{*}(X),$  fix a chain level representative
 $\alpha_c \in C^{BM}_{*}(X)$. For any such representative $\alpha_c \in
 C^{BM}_{*}(X)$, and orbit $x_0$ in $\mathcal{X}(X ; H_M^{\lambda})$, choose
 a generic surface dependent almost-complex structure $J_S$ as above which agrees with some $J_t$ used to define the Floer complex $CF^*(X\subset M, H_M^{\lambda})$ along the cylindrical end. Consider those $u \in
 \mathcal{M}_{M}(x_0)$ such that $u(S) \subset X$ and denote this moduli space
 by $\mathring{\mc{M}}_M(x_0)$. 
 Next form 
 \begin{align} 
     \mc{M}(\alpha_c,x_0) :=
     \mathring{\mc{M}}_M (x_0) \times_{ev_{z_{0}}} \alpha_c.
 \end{align} 
 For
 generic choices, this is a manifold of dimension $|x_0|-|\alpha|$ provided
 that $|x_0|-|\alpha| \leq 1$. A standard orientation analysis shows that
 whenever $|x_0| - |\alpha| = 0$, (rigid) elements $u \in \mc{M}(\alpha_c,
 x_0)$ induce isomorphisms of orientation lines
 \begin{equation}
     \mu_u: \R \stackrel{\cong}{\ra} \mathfrak{o}_{x_0}.
\end{equation}
Thus, we can define
\begin{equation} \label{eq:pssdefn}
    \operatorname{PSS}(\alpha_c)= 
    \sum_{x_0, |x_0|-|\alpha|=0} \sum_{u \in \mc{M}(\alpha_c, x_0)} \mu_u.
\end{equation} 
 It is a classical fact that this count gives rise to a well-defined map: 
 \begin{equation} 
     \label{eq:PSScl} \operatorname{PSS}: H^*(X) \to   HF^{*}(X \subset M; H_{M}^\lambda)  
 \end{equation} 

In fact, we have a factorization: 
\[
\xymatrix{
  H^*(X) \ar@{-->}[d] \ar[dr]^{PSS} \\
  HF_{-}^{*}(X \subset M; H_{M}^\lambda)  \ar[r]  & HF^{*}(X \subset M; H_{M}^\lambda)
}.
\]
To see this, note that by Stokes' theorem, the topological energy $E_{top}(u)$
to any solution $u \in \mc{M}(\alpha_c,x_0)$ (in the sense of \eqref{eq:Etop})
is simply the action $\mc{A}(x_0)$; on the other hand by the discussion in \S
\ref{subsec:algstructures} (specifically \eqref{eq:monotonicidentity}) $E_{top}(u)$ is
an upper bound for geometric energy $E_{geo}(u)$ up to a very small error (because the perturbation 1-form in this case is $K = H^{\lambda}\beta$, where $\beta$ is subclosed and $H^{\lambda}_M$ is $C^2$-close to a non-negative function). On
the other hand one can directly compute that the action of any non-constant orbit in
$\mathcal{X}(X; H^{\lambda}_M)$ is negative (and less than a fixed $-\delta$, again using the fact that $H^{\lambda}_M$ is sufficiently close to $h^{\lambda}_M$). Hence for $H_M^{\lambda}$ close to $h^{\lambda}_M$, any element  
$u \in \mc{M}(\alpha_c,x_0)$ with $x_0$ non-constant would have to have
$E_{geo}(u) < 0$, meaning $u$ cannot exist. Hence
$\mc{M}(\alpha_c,x_0)$ is only possibly non-empty for constant orbits $x_0$,
producing the desired factorization.

It will be technically convenient to distinguish the
variant map 
\begin{align}  \label{eq:PSSpseudo}
    \PSS_{\vec{\v}_{\emptyset}}:H_{2n-*}(\bar{X},\partial{\bar{X}}) \to
    HF_{-}^{*}(X \subset M; H_{M}^\lambda),
\end{align} 
which is defined by representing elements of 
$H_{2n-*}(\bar{X},\partial{\bar{X}})$ by relative pseudocycles \cite{Kahn} $P$
such that $\partial P \subset \partial{\bar{X}}$ and counting elements of 
\[
    \mc{M}(P, x_0) = \mathring{\mc{M}}_M(x_0) \times_{ev} P 
\]
for varying $x_0$ as in \eqref{eq:pssdefn}.  On homology this map agrees with
the classical PSS map \eqref{eq:PSScl} defined using singular co-chains under
the isomorphism $H^*(X) \cong H_{2n-*}(\bar{X},\partial{\bar{X}})$.
Furthermore, given a generic relative null-bordism $Z_b$ for a pseudo-manifold
$P$, we have that 
\begin{equation}
    \label{eq:PSS_cl} \operatorname{\partial}_{CF}\circ
   \PSS_{\vec{\v}_{\emptyset}}(Z_b)=\PSS_{\vec{\v}_{\emptyset}}(P) .
\end{equation}
\begin{rem}\label{rem:groundring}
    As stated in our conventions (see \S \ref{conventions}), we restrict to
    $\mathbf{k} = \Z$, $\mathbb{Q}$, or $\C$ in order to make the standard technically simplifying
    use of the theory of pseudocycles up to bordism \cite{Kahn} as a
    model for homology with $\mathbf{k}$ coefficients. To model general
    homology classes over arbitrary $\mathbf{k}$, it seems better to pass to a
    different model for homology (which still
    interacts well with Floer theory, such as Morse homology as implemented in
    the sequel article \cite{DPSGII}). On the other hand sometimes a given
    homology class over $\mathbf{k} \notin \{\Z, \mathbb{Q}, \C\}$ can be represented by a
    pseudochain (e.g, in characteristic $p$ a pseudochain whose boundary is a
    $p$-fold cover), in which case the simplified Floer-theoretic arguments
    given here work without further modification.
\end{rem}

\subsection{Adding multiplicity and jet constraints near $\infty$}

One of the main objectives of the present paper is to describe a version of the
map \eqref{eq:PSScl} which takes into account tangency conditions along the
divisor $\D$. 
As before let $\vec{\v}$ be a vector strictly
supported on a subset $I$ with $D_I$ a non-empty stratum. 
The map 
$\mathrm{PSS}_{log}^\lambda ( (-) t^{\vec{\v}})$ 
is an enhanced version of the
classical PSS morphism, involving curves passing through cycles not necessarily
in $X$ but at $\infty$ ($\mathbf{D}_i$), with various incidence and
multiplicity conditions.
The following moduli space therefore plays a central role in this
paper. 

\begin{defn} 
Fix $J \in \mathcal{J}_S(V)$. For every orbit $x_0 \in \mathcal{X}(X ;
H_M^{\lambda})$ and $\vec{\v}$ as above, define a moduli space 
\begin{equation}\label{modulispace1}
\mc{M}(\vec{\v}, x_0)
\end{equation}
as follows:
consider the space of maps
\[
u: S \ra M
\]
satisfying Floer's equation
\begin{equation}
    (du - X_{H_M^\lambda} \otimes \beta)^{0,1} = 0
\end{equation}
with asymptotics and tangency/intersection conditions
\begin{align}
    &\lim_{s \ra -\infty} u(\epsilon(s,t)) = x_0\\
    \label{eq:incidencecondition}
&u(z) \notin \D\textrm{ for $z \neq z_0$}; \\
\label{eq:incidence} &u(z_0) \textrm{ intersects $D_i$ with multiplicity exactly $\vi$ for all $i$ 
(so $u(z_0) \notin D_i$ if $i \notin I$)}.
\end{align}
\end{defn} 
In the setting of \eqref{eq:incidencecondition}, as in \S \ref{subsec:gwinvarint} (and
specifically \eqref{eq:enhancedevaluation}), the real-oriented-projectivized $\vi$
normal jets of the map $u$ (with respect to a fixed real tangent ray in
$T_{z_0} C$) give an {\em enhanced evaluation map}
\begin{equation}\label{eq:enhancedevaluationpss}
    \Evzo:  \mc{M}(\vec{\v}, x_0) \ra \mathring{S}_I.
\end{equation}
  Let $\alpha_c t^{\vec{\v}} \in C_{log}^{*}(M,\D)$ be a cocycle representing a
  class $\alpha t^{\v} \in \QH^*(M,\D)$.  
\begin{defn} \label{def: higher} The
      moduli space $\mc{M}(\vec{\v}, \alpha_c, x_0)$ is defined to be the
      moduli space 
    \begin{align} 
        \label{eq:eval} \mc{M}(\vec{\v}, x_0) \times_{\Evzo} \alpha_c,
    \end{align}
\end{defn}  \vskip 5 pt

\begin{lem} 
The virtual dimension of $\mc{M}(\vec{\v},\alpha_c, x_0)$ is given by
\begin{equation} \operatorname{vdim}(\mc{M}(\vec{\v}, \alpha_c, x_0))= |x_0|-|\alpha
t^{\vec{\v}}|, \end{equation} 
where $|\alpha t^{\vec{v}}|$ is the grading on log cohomology induced by the
holomorphic volume form $\Omega_{M, \mathbf{D}}$ in \eqref{eq:loggrading}
(and similarly $|x_0|$ is the grading on orbits induced by $\Omega_{M,
\mathbf{D}}$).
\end{lem} 
\begin{proof} 
    This is standard: one notes that at a point $u \in
    \mc{M}(\vec{\v},\alpha_c, x_0)$, the tangent space to $\mc{M}(\vec{v},
    \alpha_c, x_0)$ is the kernel of the linearized operator $\tilde{D}_u'$
    corresponding to Floer's equation intersected with the linearization of the
    constraints \eqref{eq:incidence}, \eqref{eq:eval}.  Hence, the dimension of
    the tangent space is the index of $\tilde{D}_u'$ (which is
    $|x_0| + 2 c_1(u) = |x_0| + \sum_i 2 a_i v_i$), minus the codimension of
    the constraints imposed ($|\alpha| + \sum_i 2 v_i$). 
\end{proof} 
For  a generic choice of almost complex structure and when
$\operatorname{vdim}(\mc{M}(\vec{\v}, \alpha_c, x_0)) \leq 1$, the moduli space
is a manifold of the expected dimension, by a standard Sard-Smale argument
applied to Lemmas \ref{lem:transversality} and \ref{lem:submersion} (which are
deferred to the next subsection). 
We next show that the PSS moduli space has a
suitable compactification provided that the value of the Hamiltonian
$H_M^{\lambda}$ is sufficiently big along the symplectization region. We start
with some preparatory lemmas:

\begin{lem} \label{lem:energyintersection}
    Given a solution $u \in \mc{M}(\vec{\v},\alpha_c, x_0)$, its topological energy (as defined in \eqref{eq:Etop}) satisfies
    \begin{align} 
        E_{top}(u)= \sum_iv_i \kappa_i - \int_{x_0(t)}x^*\theta - \int_S d(u^*H^{\lambda}_M\beta).
    \end{align}  
\end{lem}

\begin{proof} 
Recall from \eqref{eq:Etop} (with $K=H^{\lambda}_M\beta$) that $$ E_{top}(u)= \int_Su^*\omega - \int_S d(u^*H^{\lambda}_M\beta)$$ 
 The intersection of $u$ with $\mathbf{D}$ is an isolated point, so we can removed a small ball $B(z_0)$ around that point and as the size of the ball goes to zero, we have that

\begin{align} 
    \int_Su^*\omega= -\int_{S^1}x_0^*\theta- \int_{\partial B(z_0)}(-\kappa_i d\varphi_i) \\
 =\sum_iv_i \kappa_i-\int_{S^1}x_0^*\theta. \end{align}  Therefore $$ E_{top}(u)= \sum_iv_i \kappa_i - \int_{x_0(t)}x^*\theta - \int_S d(u^*H^{\lambda}_M \beta)$$ as claimed. 

  \end{proof}

To state our key compactness result, note that given a primitive admissible
class $\v_I$ and a generic complex structure, we defined in
\eqref{GWcyclebeforefiberproduct} a Gromov-Witten type moduli space
$\mathcal{M}_{0,2}(M,\D,\vec{\v}_{I})^o$ 
along with in \eqref{eq:Evo} an
evaluation map $\operatorname{Ev}_0$ from it to $\mathring{S}_I$.  Restricting the fiber product 
\begin{align}
    \mathcal{M}_{0,2}(M,\D,\vec{\v}_{I})^o \times_{\operatorname{Ev}_0} \alpha_c 
\end{align} 
to $\bar{X} \subset X$ defines a pseudocycle (rel boundary) $\GWvc$ representing
$\GWv$ (which was defined in \eqref{eq: GWinvariantsnormal} using $\mathcal{M}_{0,2}(M, \D, \vec{v}_I)^o$). We suppress the choice of complex structure from our notation, though
they are of course necessary for chain level constructions. 

\begin{lem} 
\label{lem:compactness} 
Suppose that $\sum_i v_i\kappa_i + \hbar -
\lambda < 0$, where $\hbar>0$ is the previously chosen small constant satisfying \eqref{eq:hbarinequality}. Let $\alpha_c t^{\vec{\v}}$ be an admissible cocycle and $x_0$ be a Hamiltonian
orbit in $\mathcal{X}(X; H_M^{\lambda})$. 
Then, for generic $J_S \in
\mathcal{J}_S(V)$, (i) if $|x_0|-|\alpha t^{\vec{\v}}|=0$, the moduli space ${\mc{M}}(\vec{\v},\alpha_c, x_0)$ is a compact 0-manifold, and (ii) if $|x_0|-|\alpha t^{\vec{\v}}|=1$, the moduli space ${\mc{M}}(\vec{\v},\alpha_c, x_0)$
    admits a compactification (in the sense of Gromov-Floer convergence)
    $\overline{\mc{M}}(\vec{\v},\alpha_c, x_0)$ to a compact 1-manifold-with-boundary such that:
\begin{enumerate} 
    \item If $\vec{\v}$ is not primitive, 
        \begin{equation}\label{bd1}
            \partial\overline{\mc{M}}(\vec{\v}, \alpha_c, x_0)=\partial\overline{\mc{M}}(\vec{\v}, \alpha_c, x_0)_F:= \bigsqcup_{x',|x_0|-|x'|=1} \mc{M}(\vec{\v},\alpha_c, x') \times \mc{M}(x_0,x')
        \end{equation}
\item If $\vec{\v}=\vec{\v}_I$, 
    \begin{equation}\label{bd2}
        \partial\overline{\mc{M}}(\vec{\v}, \alpha_c, x_0)=\partial\overline{\mc{M}}(\vec{\v}, \alpha_c, x_0)_F \bigsqcup \partial\overline{\mc{M}}(\vec{\v}, \alpha_c, x_0)_S 
    \end{equation}
    where 
    \begin{equation}
        \partial\overline{\mc{M}}(\vec{\v}, \alpha_c, x_0)_S:=\GWvc \times_{ev} \mathcal{M}(\vec{\v}_{\emptyset},x_0).
    \end{equation}
    (This latter space is empty if $\alpha_c t^{\vec{\v}}$ is tautologically admissible).

\end{enumerate}
  \end{lem}

\begin{proof} 
    The space $M$ is compact and so standard Gromov-Floer compactness results
    imply that the only possible accumulation points of a sequence of elements in
    ${\mc{M}}(\vec{\v},\alpha_c, x_0)$ are a PSS solution followed by a
    sequence of broken Floer cylinders with trees of sphere bubbles glued on at
    any point (where intermediate Floer cylinder breaking a priori occurs along
    any orbit of $H^{\lambda}_M$). The Lemma follows from (applying
        transversality --- discussed above Lemma \ref{lem:energyintersection} and
     more in \S \ref{section:transversality} --- and standard gluing
        results to) the below sequence of assertions regarding a given
        limiting stable curve:
\begin{enumerate}
    \item[(i)] There are no cylinder breakings along orbits of $H_M^{\lambda}$ contained
            in the divisor $\D$. 

        \item[(ii)] There are no sphere bubbles unless $\vec{\v}$ is primitive (i.e.,
            $\vec{\v} = \v_I$ for some $I$), in which case sphere bubbling
            occurs in codimension 1 and contributes boundary strata in (2) above.

        \item[(iii)] Any cylinder component has 0 intersection number with $\D$.
            (Hence, by positivity of intersection Lemma \ref{lem:positivity}, 
            is completely contained in $X = M \backslash \D$).  
    \end{enumerate}
    Beginning with assertion (i), we proceed by contradiction and  assume
    that $y$ is the last
    orbit of $\mathcal{X}(\D; H_M^{\lambda})$ which
    arises as an
    asymptote of such a degeneration. More precisely, suppose $y$ is in
   $\mathcal{X}(\D; H_M^{\lambda})$, and that the limiting stable curve $u_{\infty}$ contains
    $(u_1,u_2)$ with $u_1$ a broken PSS solution asymptotic to $y$ at its output and
    and $u_2$ is a broken Floer trajectory from $y$ to $x_0$ 
    (there may be sphere bubbles as
    well attached to $u_1$ and/or $u_2$, but for now we ignore them).
    Abbreviating $H:= H_M^{\lambda}$, we have by \eqref{eq:hbarinequality} 
    (as $(-d\beta) \geq 0$ and $\int (- d \beta) = 1$) that  
    \begin{align} \label{eq:energyu1}
    \Etop(u_1) = \Egeo (u_1)- \int H d\beta 
           \geq  - \hbar 
    \end{align}  
    We now show that $u_2$ cannot exist by energy considerations. 
    Namely, let $R$ be the Liouville coordinate and consider the slice $\{R=R_H\}$.
    Along this slice, 
    \begin{equation} 
        H=\lambda (R-1) 
    \end{equation} 
    and hence (as $\theta(X_H) = \omega(Z, X_H) = dH(Z) = dH(R\partial R) = RH'(R) = \lambda R$) on this slice
    \begin{equation}
        \label{thetaX}
    \theta(X_H) = H + \lambda.
    \end{equation}
    Consider the portion $\bar{S} = u_2^{-1}(R^{-1}([R_H, \infty)))$ of the
    domain of $u_2$ mapping to the region above this slice. Let $\underline{S}$ denote the remainder of the domain of $u_2$, $\underline{S}:=\operatorname{dom}(u_2)\setminus \bar{S}$ and let $S_\infty$ denote the domain of $u_\infty$. 
    The geometric energy of $u_2$ restricted to $\bar{S}$ can be estimated by:
    \begin{align*} 
         E_{geo}((u_2)|_{\bar{S}})  &\leq E_{top}((u_2)|_{\bar{S}}) \ \ \ \ \ \ \ \ \ \ \ \ \ \ \ \ \ \ \ \ \ \ \ \ \ \ \ \ \ \ \ \ \ \ \ \ \ \ \ \ \ \ \ \ \ \ \ \quad \quad \quad \textrm{(by \eqref{eq:monotonicidentityonend})}\\
        &\leq E_{top}(u_{\infty}) - E_{top}(u_1)- E_{top}((u_2)|_{\underline{S}})\\
        & \leq E_{top}(u_\infty) + \hbar- E_{top}((u_2)|_{\underline{S}}) \ \ \ \ \ \ \ \ \ \ \ \ \ \ \ \ \ \ \ \ \ \ \ \ \ \ \ \ \ \ \ \ \ \ \ \ \textrm{(by \eqref{eq:energyu1})}\\
        &=\sum_iv_i \kappa_i - \int_{x_0(t)}x^*\theta - \int_{S_\infty}d(H\beta) + \hbar- E_{top}((u_2)|_{\underline{S}}) \ \ \ \textrm{(by Lemma \ref{lem:energyintersection})}\\
\end{align*}
In the equations above, the definition of any integral
    (such as topological energy) over a broken curve is by definition the
    sum of the integrals over the components. The inequality in the 2nd equation comes from the fact that the topological energy over sphere bubbles is strictly positive.   In applying Lemma
    \ref{lem:energyintersection}, we use the fact that $E_{top}$ is a
    topological quantity and hence is preserved under limits in the Gromov topology. Let $\partial \bar{S} = u_2^{-1}(R^{-1}(R_H))$ (with its induced orientation). Combining the fact that $E_{top}((u_2)|_{\underline{S}})= - \int_{x_0(t)}x^*\theta-\int_{\partial \bar{S}} u^* \theta-\int_{\underline{S}}d(H\beta)$ with the equality $- \int_{S_\infty}d(H\beta)+\int_{\underline{S}}d(H\beta) =- \int_{\partial \bar{S}} u^*(H) \beta$ (by Stokes'
theorem), gives us that:
\begin{align*}
           &\sum_iv_i \kappa_i - \int_{x_0(t)}x^*\theta - \int_{S_\infty}d(H\beta) + \hbar- E_{top}((u_2)|_{\underline{S}}) = \sum_iv_i \kappa_i+ \hbar  + \int_{\partial \bar{S}} u^* \theta-u^*(H) \beta \\
        &= \sum_iv_i \kappa_i+ \hbar  + \int_{\partial \bar{S}}u^* \theta-u^*(\theta(X_H))\beta  + \int_{\partial \bar{S}}\lambda \beta \ \ \ \ \ \ \ \ \ \ \ \ \ \ \ \ \ \ \ \ \ \ \ \ \ \ \ \textrm{(by \eqref{thetaX})}
    \end{align*} 
 
Now because $\beta = dt$ on the domain of $u_2$, Stokes' theorem implies
    that
    \begin{equation} 
        \int_{\partial \bar{S}}\lambda \beta= -\int_{y}\lambda \beta 
        = -\lambda.  
    \end{equation} 
    Therefore, by our assumption that $\sum_i v_i \kappa_i + \hbar - \lambda < 0$,
    \begin{equation}\label{eq:keyinequality} 
        \sum_iv_i \kappa_i+ \hbar  + \int_{\partial \bar{S}}
        u^*\theta-u^*(\theta(X_H))\beta + \int_{\partial \bar{S}}\lambda
            \beta \leq \int_{\partial \bar{S}}u^*\theta-u^*(\theta(X_H))\beta.
        \end{equation} 
    The rest proceeds as in \cite{Abouzaid:2010ly}*{Lemma 7.2}, but we recall
    the details for completeness. Setting $X_K = X_H \otimes dt$ 
    we have that
    the right hand side of \eqref{eq:keyinequality} 
    \begin{align} 
        & = \int_{\partial \bar{S}} \theta \circ (du - X_K) \\ 
        &= -\int_{\partial \bar{S}} \theta \circ J \circ (du - X_H\otimes dt) \circ j \\ 
        &= - \int_{\partial \bar{S}} dR \circ du \circ j,   
    \end{align} 
    as $dR(X_H) = 0$.
     Finally, letting $\hat{n}$ denote the outward normal along
     $\partial\bar{S}$, observe that $\partial \bar{S}$ is  oriented by the
     vector $j\hat{n}$. Now we calculate that
     $-dR(du)j(j\hat{n})=-dR(du)(-\hat{n}) \leq 0$, which implies that the
     final integral, hence $\Egeo((u_2)|_{\bar{S}})$ is non-positive. It follows that $du
     = X_K$ everywhere and in particular, $R$ must be constant on
     $u_2(\bar{S})$ which is impossible.  Hence, such a breaking cannot occur. 

     Since cylinder components have asymptotics all contained in $X = M \backslash \D$, it
     follows that the entire broken curve and each of its components define
     classes in $H_2(M, M \backslash \D)$, hence have well-defined
     topological intersection number with each of the components of $\D$, which
     is additive over the components of the broken curve. The total topological
     intersection number of the broken curve with a given $D_i$ is equal to
     $v_i$, and for any components of the stable curve not completely contained
     in $D_i$, the intersection number with $D_i$ must be positive, by Lemma
     \ref{lem:positivity}.

    Turning to assertion (ii), note that $M \backslash \D$ is exact so every
    non-constant sphere bubble in our stable curve must intersect $\mathbf{D}$.
    Since $\alpha_c t^{\vec{\v}}$ is admissible, we see that:
    \begin{itemize}
        \item If $\vec{\v}$ is not primitive then there are no sphere bubbles at all.
            Indeed, denoting by $A$ the sum of the homology classes of all
            non-constant sphere bubbles in the given stable curve, admissibility
            implies $A \cdot D_i > v_i$ for some $i$. But additivity of
            intersection numbers implies that $A \cdot D_i + (\mathrm{remaining\ curve})
            \cdot D_i = v_i$, a contradiction as the latter intersection number
            is non-negative.

        \item If $\vec{\v} = \vec{\v}_I$ for some $I$, the same arguments imply
            that a sphere bubble can only possibly appear in an indecomposable
            homology class $A$ with $A \cdot \mathbf{D} = \vec{\mathbf{v}}_I$.
            If this happens, there must be a unique sphere bubble $u_1$ with $[u_1]
            \cdot \mathbf{D} = \vec{\mathbf{v}}_I$. The
            remaining curve $u_2$, which does not have any components contained
            in $\D$, must have zero intersection number with $\D$ and hence (by
            positivity of intersection) 
            be completely contained in $X = M \backslash \D$ (in particular,
            $u_1$ must contain the point $z_0$). Furthermore $u_1$, as it meets
            $u_2$ at $u_1(\infty) \in X$,
            must lie in $\mathcal{M}_{0,2}(M, \D, \vec{\v}_I)^o$ (in particular is
            not completely contained in $\D$ either).

    \end{itemize}

    Finally, let's turn to (iii). We have shown that
    all components of our stable curve are not completely contained in $\D$,
    hence intersect $\D$ positively. Also, we have shown that the total
    intersection number of any components which are {\em not} Floer cylinders
    (e.g., the unique element of $\mathcal{M}(\vec{\v}, \alpha_c, x')$ along
    with any sphere bubbles) with $\D$ is $\vec{\v}$. It follows that cylinder
    components each have 0 intersection number with $\D$ as desired.

    With assertions (i)-(iii) established, standard gluing analysis then shows
    that the moduli spaces \eqref{bd1}-\eqref{bd2} are the only ones arising in
    codimension 1. 
    To justify the appearance of the sphere bubble in codimension 1 (rather
    than 2), note that we are studying the moduli space of maps from a plane
    with fixed real tangent ray at $z_0$; hence, any sphere bubble inherits a
    tangent ray at 0, and 
    there is a one-dimensional subspace worth of gluings for which the
    induced tangent ray to the glued solution at $z_0$ is positive real.

\end{proof}
Suppose for the remainder of the subsection that $\sum_i v_i\kappa_i + \hbar -
\lambda < 0$. For any admissible cocycle $\alpha_c t^{\vec{\v}}$, we define $\PSSlog^{\lambda} (\alpha_c t^{\vec{\v}}) \in CF^{*}(X \subset M;H_M^{\lambda})$ (again, the complex structure $J_t$ needed to define $CF^{*}(X \subset M;H_M^{\lambda})$ should agree with $J_S$ along the cylindrical end) by the formula  
\begin{equation}
 \PSSlog^{\lambda} (\alpha_c t^{\vec{\v}}) = \sum_{x_0, \operatorname{vdim}(\mc{M}(\vec{\v}, \alpha_c, x_0))=0} \sum_{u \in\mc{M}(\vec{\v}, \alpha_c, x_0)} \mu_u 
\end{equation}
where once more, for a rigid element $u\in \mc{M}(\vec{\v}, \alpha_c, x_0)$,
$\mu_u: \R \ra \mathfrak{o}_{x_0}$ is the isomorphism induced on orientation
lines (and by abuse of notation, their $\K$-normalizations) by the gluing
theory. 

Let $\PSSlog^{\lambda,+} (\alpha_c t^{\vec{\v}})$ be the image:
\begin{align} 
    \PSSlog^{\lambda,+} (\alpha_c t^{\vec{\v}})= \overline{\PSSlog^{\lambda} (\alpha_c t^{\vec{\v}})} \in CF_{+}^*(X \subset M;H_M^{\lambda}) 
\end{align}
It follows from Lemma \ref{lem:compactness} that $\PSSlog^{\lambda,+} (\alpha_c
t^{\vec{\v}})$ defines a cocycle: 
\begin{align} 
    \partial_{CF_{+}}(\PSSlog^{\lambda,+} (\alpha_c t^{\vec{\v}}))=0. 
\end{align} 
If $\alpha_c t^{\vec{\v}}$ is tautologically admissible, then
Lemma \ref{lem:compactness} implies that $\PSSlog^{\lambda} (\alpha_c
t^{\vec{\v}})$ defines a cocycle as well:
\begin{align} 
    \partial_{CF}(\PSSlog^{\lambda} (\alpha_c t^{\vec{\v}}))=0. 
\end{align} 

\subsection{Transversality} \label{section:transversality}
Fix a multiplicity vector $\vec{\v} = (v_1, \ldots, v_k)$ corresponding to our
collection of divisors $D_1, \ldots, D_k$.  The purpose of this section is to
prove that we can achieve transversality for 
$\operatorname{PSS}_{log}^{\lambda}((-)t^{\vec{\v}})$ 
within the class of complex structures $\mathcal{J}_S(V).$ We  view
$\mathcal{J}_S(V)$ as a subspace of surface dependent complex structures in
$\mathcal{J}(M,\D)$. Concretely, fixing a pair $(H_M, J_t)$ used to define $CF^{*}(X \subset M;H_M^{\lambda})$ (i.e., achieving transversality), we will show transversality can be achieved within the subspace $\J \subset \mathcal{J}_S(V)$ of $S$-dependent almost-complex structures which agree with this particular $J_t$ on the cylindrical end for $s \ll 0$. 
\vskip 10 pt

Let $m$ be an integer larger 
than $|\vec{\v}|= \sup_i \vi+1$, and for some (real) $p>1$ 
define
\begin{equation}\label{eq:banachspace}
    \mathcal{B}^{m,p}_{S,x}
\end{equation}
to be the Banach manifold of maps $u: S \ra M$ which are locally in $W^{m,p}$ and
which converge to some orbit $x$ in the $W^{m, p}$ sense with respect to the
fixed cylindrical end (see e.g., \cite[Proof of Thm. 5.1]{Floer:1995fk}, \cite[Def. 7.1.3]{FOOO2}, \cite[Def.
4.1]{AbouzaidBordism}, \cite[(C.1.8)]{pardonVFC} for precise articulations of
this condition).  When $S$ and $x$ are implicit we will just use the
notation $\mathcal{B}^{m,p}$. Over \eqref{eq:banachspace}
there is a Banach bundle 
\begin{equation}
    \mathcal{E}:= \mathcal{E}^{m-1,p}_{S,x}
\end{equation} 
whose fiber over $f: S \ra M$ consists of the space of maps
\begin{equation}
    \mathcal{E}_f := W^{m-1,p}(S, \Omega^{0,1}_S \otimes f^*TM),
\end{equation} 
where on the right hand side, the Sobolev space of sections 
$S \ra \Omega^{0,1}_S \otimes f^* TM$ is defined using the cylindrical ends on
$S$ (to fix a measure on $S$), 
an almost complex structure on $S$, and a metric and connection on $M$ (though it is independent of the particular choice thereof)
---see for instance \cite{Schwarzthesis} or the references cited above in the definition of $\mathcal{B}^{m,p}$. Again, we leave $S$ and $x$ implicit when they are fixed, and simply use the notation $\mathcal{E}$. 
For any $H$ with time-1 orbit $x$ and any $J$, the Cauchy-Riemann operator
\begin{equation}
    \bar{\partial}_{H,J}: = (d(-) - X_H \otimes \gamma)^{0,1}
\end{equation}
induces a section $\bar{\partial}_H: \mathcal{B}^{m,p} \ra \mathcal{E}^{m-1,p}$
which is Fredholm whenever $x$ is non-degenerate \cite{Floer:1988aa}.

There is an associated {\em universal moduli space} to our problem, the space
of solutions to Floer's equation over $S$ for varying $J$, which have $v_i$
multiplicity intersection with $D_i$ at $z_0$: 
\begin{equation}  \label{eq: universalmodulispace} \mc{M}_{univ}(\vec{\v}, x_0)=: \lbrace (f,J) \in
    \mathcal{B}^{m,p} \times \J, \bar{\partial}_{H,J} f=0,
    d^{\vi-1}(f)_{z_0} \in T_{f(z_0)}D_i  \textrm{ for }1 \leq i \leq k \rbrace 
\end{equation}
(the notation $d^{\vi-1}( f)$ means the {\em $\vi-1$-jet of $f$}, in the sense
described in \cite{CieliebakMohnke}*{\S 6}).   
Let $T_J(\widetilde{J})$
denote the space of infinitesimal deformations of
our (domain-dependent) $J$ within the class $\widetilde{J}$, which necessarily vanish in a neighborhood of 0 on the cylindrical end of $S$.
We recall that our almost-complex structures are all $C^{\infty}$, and moreoever $C^{\epsilon}$ deformations of a fixed almost complex structure in the sense of Definition \ref{def:JMD}. 
The tangent space to $\mc{M}_{univ}(\vec{\v}, x_0)$ at a point $(f,J)$, using a
local chart around $f(z_0)$ in $X$ in which $D_i = \{y_i = 0\} \subset \C^n$ (where $y_1, \ldots, y_n$ are the standard coordinates), so
that points $f(z)$ for $z$ near $z_0$ are thought of as living in $\C^n$:
\begin{equation}
    T_{(f,J)} \mc{M}_{univ}(\vec{\v}, x_0) = \{ (\psi, Y) \in T_f \mathcal{B}^{m,p} \times T_J \J | d^{\vi-1} \psi_{z_0} \in \{y_i = 0\} \textrm{ for }1 \leq i \leq k\}
\end{equation}

As before there is an enhanced evaluation map $\Evzo : \mc{M}_{univ}(\vec{\v}, x_0)  \to \mathring{S}_I$ and with respect to it, we define for any cocycle $\beta \in C^*(\mathring{S}_I) := C_{2n-|I|-*}^{BM}(\mathring{S}_I)$, the constrained moduli space
\begin{equation}
    \mc{M}_{univ}(\beta, \vec{\v}, x_0) := \mc{M}_{univ}(\vec{\v}, x_0) \times _{\Evzo} \beta.
\end{equation}

The following two key Lemmas are the main results of this section:
\begin{lem} 
    \label{lem:transversality}
    For $ m-2/p > |\vec{\v}|$, the space  $\mc{M}_{univ}(\vec{\v}, x_0)$ is a smooth Banach manifold. 
\end{lem}
\begin{lem} \label{lem:transversalityjet}
    For $ m-2/p > |\vec{\v}|$, and $\beta \in C^*(\mathring{S}_I)$, where $I =
    \operatorname{supp}(\vec{\v})$, the space  $\mc{M}_{univ}(\beta, \vec{\v}, x_0)$ is a smooth
    Banach manifold.  
\end{lem}
The proof of Lemmas \ref{lem:transversality}-\ref{lem:transversalityjet} follow closely
\cite{CieliebakMohnke}*{Lemma 6.5-6.6}. Fixing an $\ell \geq 0$ with $\ell < m
- 2/p$, we define a subspace of the tangent spaces to $\mathcal{B}^{m,p}$ of
tangent vectors with vanishing $\ell$ jet: 
\begin{equation} 
    B_0^{m,p}=: \lbrace \psi \in T_f\mathcal{B}^{m,p}, d^{\ell} \psi_{z_0}=0 \rbrace
\end{equation}
Similarly, we define a subspace of (the vertical tangent space to)
$\mathcal{E}_f^{m-1, p}$ of sections with vanishing $(\ell - 1)$ jet (a condition
which is by convention vacuous if $\ell=-1$):
\begin{equation}
    E_0^{m-1,p}=: \lbrace \eta \in \mathcal{E}_f^{m-1,p}, d^{\ell-1} \eta_{z_0}=0 \rbrace
\end{equation}

The first assertion is that 
\begin{lem}\label{lem:vanishingjetsurjective}
    The linearized operator 
\begin{equation}
    F_0: B_0^{m,p} \oplus T_J(\J) \to  E_0^{m-1,p}
\end{equation}
which sends 
\begin{equation}
    (\psi,Y) \to D_f \psi + \frac{1}{2} Y_{z, f(z)} \circ (df_z - X_{f(z)} \otimes \gamma_z) \circ j_z
\end{equation}
is surjective. 
\end{lem}

Assuming for a moment the proof of Lemma \ref{lem:vanishingjetsurjective}, we may now give the proof of  Lemma \ref{lem:transversality}. 

\begin{proof}[Proof of Lemma \ref{lem:transversality}]
    Without loss of generality by reordering the divisors we can assume that the support of $\vec{\v}$ is $\lbrace
  1,\cdots |I| \rbrace$ and that the vectors $v_i$ are ordered so that $v_i
  \leq v_{i+1}$. Let $\vec{\v}^{(i)}$ denote the vector $(v_1,v_2,\ldots, v_{i-1}, v_i, v_i, \ldots, v_i,0,0 \cdots,0)$ 
  where the final $v_i$ occurs in position
  $|I|$. We will prove by induction on $i$ that the moduli spaces
  $\mc{M}_{univ}(\vec{\v}^{(i)}, x_0)$ are smooth Banach manifolds, where
  as in \eqref{eq: universalmodulispace} 
  \begin{equation}
      \label{eq: universalmodulispace2} \mc{M}_{univ}(\vec{\v}^{(i)}, x_0)=: \lbrace
      (f,J) \in \mathcal{B}^{m,p} \times \J, \bar{\partial}_{H,J}
      f=0, d^{\v^{(i)}_{j}-1}(f)_{z_0} \in T_{f(z_0)}D_j \textrm{ for }1 \leq j \leq k  \rbrace 
  \end{equation}

 The base case $i = 0$ is completely standard, so assume it is true for $i$. If
 $v_i=v_{i+1}$ then $\vec{\v}^{(i)} = \vec{\v}^{(i+1)}$ completing the
 induction, so let us assume $v_{i+1}>v_i$.  Let
 $\vec{v}^{(i,j)}$ denote the vectors $(v_1,v_2,\cdots v_i+j,\cdots, v_i+j,0,0
 \cdots,0)$ for $ 0 \leq j \leq  v_{i+1}-v_i$ and define
 $\mc{M}_{univ}(\vec{\v}^{(i,j)}, x_0)$ in the same way as above. The inductive step
 is itself proven by induction, this time on $j$. For every point $ (f,J) \in
 \mc{M}_{univ}(\vec{\v}^{(i,j)}, x_0)$ pick a chart near $p= f(z_0)$ which we
 identify with (an open subset of) $\C^n$ so that $p$ corresponds to the
 origin. Under this identification, we assume that each $D_i$ is identified
 with the hyperplane given by $y_i=0$, that the chart is small enough so that
 (on the preimage of this chart under $f$) the inhomogeneous term in \eqref{eq: FloerM} is zero
 and that
 after a linear change of coordinates $J(0)=i$, where $i$ is the standard
 almost complex structure on $\mathbb{C}^n$. It suffices to show that for every
 such $(f,J)$ the $v_{i}+j+1$- linearized jet map normal to $
 \mathbb{C}^n/\mathbb{C}^{n-|I|+i}$ is surjective. This is proven exactly as in
 part b) of Lemma 6.5 of \cite{CieliebakMohnke}.  
 \end{proof}

Continuing to assume the proof of Lemma \ref{lem:vanishingjetsurjective}, the proof of Lemma \ref{lem:transversalityjet} requires one further Lemma:
\begin{lem} 
    \label{lem:submersion} Assume that in the notation of Lemma \ref{lem:transversality}, $m-2/p > |\vec{\v}|+1 $. Then the evaluation map 
    \begin{equation} 
        \Evzo : \mc{M}_{univ}(\vec{\v}, x_0)  \to \mathring{S}_I 
    \end{equation} is a submersion. 
\end{lem}
\begin{proof} 
    As before, let $(f,J)$ be a point in $\mc{M}_{univ}(\vec{\v}, x_0)$ and
    $ev_{z_0}: \mc{M}_{univ}(\vec{\v}, x_0) \to \mathring{D}_I$ the ordinary
    evaluation map (i.e., $ev_{z_0}(f,J) = f(z_0)$).  There is an exact
    sequence 
    \begin{align}\label{exactsequence} 
    0 \to T_{\Evzo(f,J)} T^{|I|} \to T_{\Evzo(f,J)}
    \mathring{S}_I \to T_{ev_{z_0}(f,J)}\mathring{D}_I \to 0 
\end{align}

We first show that $d(ev_{z_{0}})$ surjects onto
$T_{ev_{z_{0}}}\mathring{D}_I$. To prove this, we let $\bar{U}$ be a
sufficiently small open neighborhood of $ev_{z_0}(f,J) = f(z_0)$ in $\mathring{D}_I$, over which the normal bundle of each $D_i$ is trivial for $i \in I$. We may
consider an open neighborhood of $f(z_0)$ in $V \subset M$ of the form $U=\bar{U} \times \mathbf{D}_{\epsilon'}^{|I|} \subset V$ (as usual we abuse notation and use $\subset$ for the trivializing diffoemorphism of $U_I|_{\bar{U}}$ followed by the tubular neighborhood map), where $\epsilon' \ll \epsilon/2$ (recall $\epsilon$ is the size of the tubular neighborhoods specified in Definition \ref{defn:nice}).
 We assume
that $f^{-1}(U)=U_{z_{0}}$ is a connected neighborhood of $z_0$ on which the
inhomogeneous term in \eqref{eq: FloerM} vanishes. Let $v$ be a tangent vector
to $\mathring{D}_I$ at $f(z_0)$ and consider a Hamiltonian $\bar{H}: D_I \to \mathbb{R}$
supported in $\bar{U}$ and for which 
\begin{align} 
    X_{\bar{H}}(f(z_0))=v.
\end{align} 
We next consider a cutoff function $\eta(x): \mathbb{R} \to
\mathbb{R}^{\geq 0}$ which is $1$ for $x \leq \epsilon'$ and 0 for $x \geq
\epsilon/2$ (note by hypotheses $\epsilon' < \epsilon/2$). Denote by 
\begin{align} 
    H=(\prod_{j \in I} \eta(r_j))\pi_I^*(\bar{H})
\end{align} 
Let $\mathcal{L}_{X_H} J$ denote the Lie derivative of $J$ in the
direction of $X_H$. The flow of the Hamiltonian vector field preserves all
divisors $D_i$. Then by \cite{McDuff:2004aa}*{Exercise 3.1.3, discussion on page 60} together with
the invariance of intersection number with the divisors $D_i$ under
homotopies, the tuple $(-X_H, \mathcal{L}_{X_H} J)$ defines a tangent vector to the moduli space
$\mc{M}_{univ}(\vec{\v}, x_0)$ at the point $(f,J).$  Applying the differential of $ev_{z_0}$ at the point $(f,J)$ gives:
\begin{equation} 
    d(ev_{z_0})_{(f,J)}(-X_H, \mathcal{L}_{X_H} J)=v.  
\end{equation}

It therefore suffices to prove surjectivity of the linearized evaluation map to
the fibers. But this is accomplished in the same way by taking for any $i \in I$
the Hamiltonians $H_i=\frac{1}{2} (\prod_{j \in I} \eta(r_j)) r_i^2$. Let $\phi_{i,t}$
denote Hamiltonian flow of this Hamiltonian. Then 
\begin{align}
    d(\Evzo)(X_{H_{i}}, \mathcal{L}_{X_{H_{i}}} J)=\frac{d}{dt}|_{t=0}(\Evzo
    (\phi_{i,t}^{-1}(f), \phi_{i,t}^*(J) )=\partial_{\theta_{i}} 
\end{align}
where $\partial_{\theta_i}$ is the generator of the circle action in the $i$-th
fiber and hence this proves surjectivity onto the fibers.  
\end{proof} 
\begin{proof}[Proof of Lemma \ref{lem:transversalityjet}]
    This is an immediate consequence of Lemmas \ref{lem:transversality} and \ref{lem:submersion}. 
\end{proof}

\begin{proof}[Proof of Lemma \ref{lem:vanishingjetsurjective}]
    This proof is a direct adaptation of \cite{CieliebakMohnke}*{Lemma 6.6}; as in
    {\em loc. cit.}, the presence of the $\ell$ jet condition forces one to
    work directly with distributions instead of the more typical argument of
    reducing to the $W^{1,p}$ case where one can use the adjoint equation (as
    in \cite{McDuff:2004aa}*{Prop. 3.2.1} or \cite{Floer:1995fk}). 
    As in \cite{CieliebakMohnke}*{Lemma 6.6}, standard functional analytic properties of $\mathcal{E}_f^{m-1, p}$ (see {\em loc.
    cit.}) reduce the proof of this Lemma to verifying that 
    $(\operatorname{im} F_0)^{\perp} \subset (E_0^{m-1,p})^\perp$ in the dual space $(\mathcal{E}_f^{m-1,p})^*$ (where $\perp$ denotes annihilator). 
    Suppose that $\Lambda \in (\mathcal{E}_f^{m-1, p})^*$ is contained in
    $(\operatorname{im} F_0)^{\perp}$, i.e., it vanishes on
    $\operatorname{im}(F_0)$.  Then elliptic regularity for distributions
    implies that the restriction of $\Lambda$ to $S^{\ast}:=S \setminus z_0$ is
    represented by some smooth section $\eta: S^{\ast}  \to E$. e.g.
\[ \Lambda(\phi)=\langle \phi, \eta\rangle_{L^2},\]
with $D_f^* \eta = 0$ and $\langle Y \circ (df - X \otimes \gamma) \circ j ,
\eta \rangle = 0$ for all $Y \in T_J \J$. To complete the proof (or rather reduce it to arguments in {\em loc. cit.}) it suffices to show that $\eta$
vanishes on $S^{\ast}$; here we shall deviate slightly from the argument of {\em loc.
cit.} (which closely emulate the standard transversality argument for simple $J$-holomorphic curves
c.f., \cite{McDuff:2004aa}*{Prop. 3.2.1}), turning instead to the corresponding
standard argument for solutions to Floer-type equations (compare \cite{Floer:1995fk}).

To prove $\eta = 0$, first let $R(f)$ denote the set of points $z$ which are {\em regular}, meaning that 
\begin{equation} 
    du_z - X_{u(z)} \otimes \gamma_z \neq 0.  
\end{equation}
A standard argument (see \cite{Floer:1995fk}*{Thm.
4.3} or \cite{Abouzaid:2010ly}*{\S 8.2} for a version adapted to general Floer
curves) establishes that regular points are open and dense in any open set $U
\subset S$ where $d\gamma|_U = 0$. For any regular point $z \in R(f)$ with $\eta(z) \neq 0$ and $f(z) \not \in V$, one may use the
usual argument of \cite{Floer:1995fk} to produce a $Y$ in
$T_J(\J)$ (in particular $Y = 0$ sufficiently near $0$)  
such that 
$ \langle Y (f) \circ df(z') \circ j, \eta(z') \rangle \geq 0$ 
and $>0$ at $z$, a contradiction to $\eta(z) \neq 0$ (note we require $f(z) \notin V$
because our almost complex structures are assumed to be of a restricted form
within $V$; of course such $z$ always exist because the asymptotics of $f$
lie outside $V$, compare \cite{Abouzaid:2010ly}*{Lemma 8.7}). 
This implies that 
$\eta = 0$ on a neighborhood in $S^{\ast}$ of $0$; hence $\eta = 0$ on all of $S^{\ast}$ by unique continuation. 
\end{proof}

\subsection{(In)dependence of various choices} 

We now investigate to what extent the $\PSSlog^{\lambda,+}$ map depends on various choices. 

\begin{lem} 
    \label{lem:continuations4} For any admissible $\alpha_ct^{\vec{\v}}$ and $\lambda_1 \leq \lambda_2$ we have an equality  \begin{align} 
    \PSSlog^{\lambda_{2},+}(\alpha_ct^{\vec{\v}})= \mathfrak{c}_{\lambda_1,\lambda_2} \circ \PSSlog^{\lambda_{1},+} (\alpha_c t^{\vec{\v}}) 
\end{align} on the level of cohomology.  \end{lem}  \begin{proof} The proof is by a degeneration of domain (as usual equipped with suitable Floer data) argument. We consider a non-negative, monotone non-increasing, function
$f(s)$ such that 
\begin{itemize}\item $f(s)=\lambda_1, s\gg 0$ \item $f(s)=
            \lambda_2, s\ll 0$ 
    \end{itemize} 
The 1-form
$$\beta=f(s)\rho(s)dt,$$ 
where $\rho$ is the cutoff function described in \eqref{eq:cutoff} and in the lines above it,
is subclosed. Now let $q \in \mathbb{R}$ be a
parameter and consider the 1-parameter family of subclosed 1-forms:
$$\beta^q=\rho(s-q)f(s)dt.$$ 
We will count solutions to the perturbed pseudo
holomorphic curve equation with input $\alpha_c$ and output $x_0$. Denote
the resulting parametrized moduli space (for varying $q \in \R$) by
\begin{equation}\label{pssacceleratedmodulispace}
    \mc{P}(\vec{\v},\alpha_c, x_0) = \coprod_{q \in \R} \mc{P}_q(\vec{\v},\alpha_c, x_0)
\end{equation}  
When $|x_0|-|\alpha
t^{\vec{\v}}|=0$, $\operatorname{vdim}(\mc{P}(\vec{\v},\alpha_c, x_0))=1.$ 
The proof of Lemma \ref{lem:compactness} extends verbatim to exclude a priori a
number of components from the Gromov-Floer compactification of the moduli space
\eqref{pssacceleratedmodulispace}, such as Floer breaking along orbits in $\D$, and most sphere bubbles. The end result is that 
the moduli space
\eqref{pssacceleratedmodulispace} has a compactification
whose boundary includes three types of components:  
\begin{itemize} 
\item the limit as $q \to \infty$
\item $q \to -\infty$
\item broken curves at finitely many $q$ values 
\end{itemize} 

The first limit corresponds to the composition $\mathfrak{c}_{\lambda_1,\lambda_2}
\circ \PSSlog^{\lambda_{1},+}$.  The second limit corresponds to taking
$\PSSlog^{\lambda_{2},+}$. The third type of boundary
component consists of two different strata: 
\begin{align} 
    &\coprod_{y ; \ |y| - |\alpha t^{\vec{v}}| = -1}  \mc{P}(\vec{\v},\alpha_c, y) \times
        \mathcal{M}(x_0, y) \\
    &\label{eq:plusstratum}  \mc{P}(\vec{\v}_\emptyset,\underline{GW_{\vec{\v}}(\alpha_c)}, x_0) 
\end{align}
(where by our hypothesis of admissibility this latter moduli space \eqref{eq:plusstratum} is empty unless
$\vec{\v}$ is primitive, i.e., equal to $\vec{\v}_I$ for some $I$).

Consider the assignment $\operatorname{T}$ of degree $-1$ given by counting
solutions to \eqref{pssacceleratedmodulispace} in the usual fashion, namely
\begin{equation} 
    \operatorname{T}(\alpha_c t^{\vec{\v}}) := \sum_{y, |y|-|\alpha t^{\vec{\v}}|=-1} \sum_{u \in  \mc{P}_q(\vec{\v},\alpha_c, y)} \mu_u
\end{equation} 
where $\mu_u: \R \ra \mathfrak{o}_y$ is the induced isomorphism on orientation
lines given by a gluing analysis.  
Then, by studying the boundary strata of one-dimensional components of \eqref{pssacceleratedmodulispace} and noting that the stratum \eqref{eq:plusstratum} contributes zero in $CF_{+}^*$, we obtain that $\operatorname{T}$ defines a chain homotopy between $\PSSlog^{\lambda_{2},+}$ and $\mathfrak{c}_{\lambda_1,\lambda_2} \circ \PSSlog^{\lambda_{1},+}$, e.g. we have that 
\begin{align} 
    \PSSlog^{\lambda_{2},+}(\alpha_c t^{\vec{\v}}) - \mathfrak{c}_{\lambda_1,\lambda_2} \circ \PSSlog^{\lambda_{1},+}(\alpha_c t^{\vec{\v}}) = \partial_{CF_{+}} \circ T (\alpha_c t^{\vec{\v}}) 
\end{align} 
\end{proof}

It follows from the Lemma that for any such chain $\alpha_c t^{\vec{\v}}$, we may obtain an element 
\begin{align} 
    \operatorname{PSS}_{log}^{+}(\alpha_c t^{\vec{\v}}):= \mathfrak{c}_{\lambda,\infty} \circ \PSSlog^{\lambda,+}(\alpha_ct^{\vec{\v}}) \in SH_{+}^*(X) 
\end{align} 
which is independent of $\lambda$. 

\begin{lem} 
    \label{lem: welldef}     
    For any admissible class $\alpha t^{\vec{\v}}$, and given two cycles
    $\alpha_c$ and $\alpha_c'$ representing $\alpha$, we have that
    \begin{align} 
        [\operatorname{PSS}_{log}^{\lambda,+} (\alpha_c t^{\vec{\v}})]= [\operatorname{PSS}_{log}^{\lambda,+}(\alpha_c' t^{\vec{\v}})] \in HF_{+}^{*}(X\subset M, H_{M}^\lambda) 
    \end{align}   
\end{lem}
\begin{proof} 
     Let $\alpha_0$ be a chain such that
    $\partial(\alpha_0)=\alpha_c-\alpha_c'$. Then after choosing complex
    structures generically, consider the moduli space of elements of 
    \eqref{modulispace1} whose enhanced evaluation lies along $\alpha_0$, which we denote by
    $\mc{M}(\vec{\v}, \alpha_0, x_0)$. When this moduli space is 1-dimensional,
    it admits a compactification with boundary (ignoring strata which
    contribute zero in the quotient complex $CF_{+}^*$) $\mc{M}(\vec{\v},
    \alpha_c, x_0)$, $\mc{M}(\vec{\v}, \alpha'_c, x_0)$ and   
    \begin{equation}
        \bigsqcup_{y, |y|-|\alpha_0 t^{\vec{\v}}|=0}
        \mc{P}_q(\vec{\v},\alpha_0, y)\times \mc{M}(x_0,y)
    \end{equation} 
    Thus $\PSSlog^{\lambda,+}(\alpha_c t^{\vec{\v}})$ and $\PSSlog^{\lambda,+}(\alpha_c' t^{\vec{\v}})$ are
    cohomologous as claimed.  
\end{proof}

As a consequence of Lemma \ref{lem:
welldef}, the PSS construction gives rise to a well-defined map on cohomology 
\begin{equation} \label{eq: maponclasses} 
    \PSSlog^{\lambda,+}:
\QH^*(M,\D)_{\lfloor \lambda' \rfloor}^{ad} \to HF_{+}^{*}(X\subset M ; H_M^{\lambda}) 
\end{equation}

for $\lambda'+\hbar<\lambda.$ Furthermore, by Lemma \ref{lem:continuations4}, we have a commutative diagram: 

\begin{equation} \label{eq: commutativediagram}
    \xymatrix{
        \QH^*(M,\D)_{\lfloor \lambda_1' \rfloor}^{ad} \ar[r]^{i_{\lambda_1',\lambda_2'}} \ar[d]^{PSS^{\lambda_{1},+}} & \QH^*(M,\D)_{\lfloor \lambda_2' \rfloor}^{ad}  \ar[d]^{PSS^{\lambda_{2},+}}\\ 
        HF_{+}^{*}(X \subset M; H_M^{\lambda_{1}}) \ar[r]^{\mathfrak{c}_{\lambda_1,\lambda_2}}  &  HF_{+}^{*}(X \subset M; H_M^{\lambda_{2}})
    }
\end{equation}

As a consequence, we obtain a well-defined map 
 \begin{align} 
     \PSSlog^{+}: \QH^*(M,\D)^{ad} \to SH_{+}^*(X) 
 \end{align}
A similar further argument shows that this map is indeed canonical, i.e.,
independent of choices made (of complex structures and Hamiltonians in the class
specified).

\subsection{Lifting to symplectic cohomology} \label{subsec:lifting}

We will need one more choice to define classes in symplectic cohomology:

\begin{defn} 
    Fix $\lambda$. Given an admissible cocycle $\alpha_c t^{\vec{\v}}$ with
    $\lambda \geq \sum \kappa_i v_i+\hbar$ corresponding to a primitive vector, we
    say that a pseudomanifold with boundary $Z_b$ is a {\em bounding cochain}
    for $\alpha_ct^{\vec{\v}}$ if
    \begin{align} 
        \label{eq:chaineq} \operatorname{\partial}_{CF} \circ
        \PSSlog^{\lambda} (\alpha_ct^{\vec{\v}})-\operatorname{\partial}_{CF}\circ
        \PSS^{\lambda}_{\vec{\v}_{\emptyset}}(Z_b)=0 
    \end{align} 
    We formally extend this definition to non-primitive $\vec{\v}$ by only allowing $Z_b=0$.
    We denote such a pair by $(\alpha_c t^{\vec{\v}}, Z_b)$.  
\end{defn} 
We suppress $\lambda$ in the notation that follows.
\begin{lem} 
    If $\alpha_c t^{\vec{\v}}$ is an admissible cocycle with vanishing
    obstruction class \eqref{eq: GWinvariantsnormal}, then there exists $Z_b$
    such that \eqref{eq:chaineq} holds.  
\end{lem}
\begin{proof} 
    Assume first that $\vec{\v}$ is not primitive, or more generally
    that $\alpha t^{\vec{\v}}$ is tautologically admissible. Given $\alpha_c
    t^{\vec{\v}}$, consider a given $x_0$ such that
    $\operatorname{vdim}(\mc{M}(\vec{\v}, \alpha_c, x_0))=1$. The boundary of
    the compactification $\overline{\mc{M}}(\vec{\v}, \alpha_c, x_0)$ has one
    component 
    \begin{align} \bigsqcup_{|x_0|=|x'|+1}
        \mc{M}(\vec{\v},\alpha_c, x') \times \mc{M}(x_0,x')
    \end{align} 
    which represents the coefficient of $|\mathfrak{o}_{x_0}|$ in the composition
    $\partial_{CF} \circ \operatorname{PSS}(\alpha_c t^{\vec{\v}})$; in particular
    $\partial_{CF} \circ \operatorname{PSS}(\alpha_c t^{\vec{\v}}) = 0$, and
    $\alpha_c t^{\vec{\v}}$ satisfies \eqref{eq:chaineq} with $Z_b = 0$ as
    desired.  In the primitive non-tautologically admissible case, choose
    generic nullbordism $Z_b$ for $\GWvc$. There is one extra component to the
    boundary of $\overline{\mc{M}}(\vec{\v}, \alpha_c, x_0)$ in this case,
    given by: 
    \begin{align} 
        \GWvc \times_{ev} \mathcal{M}(\vec{\v}_{\emptyset},x_0).
    \end{align} 
    In this case, we have that 
    \begin{align} 
        \partial_{CF}(\PSSlog (\alpha_c t^{\vec{\v}}))=\PSS_{\vec{\v}_{\emptyset}}(\GWvc) 
    \end{align} 
    In view of \eqref{eq:PSS_cl}, this can be canceled out by
    $\operatorname{\partial}_{CF} \circ \operatorname{PSS}_{\vec{\v}_{\emptyset}}(Z_b)$ and thus
    \eqref{eq:chaineq} holds.  
\end{proof}

For any pair $(\alpha_ct^{\vec{\v}}, Z_b)$, define $\PSSlog^{\lambda} (\alpha_c t^{\vec{\v}}, Z_b) \in HF^{*}(X \subset M;H_M^{\lambda})$ via the formula 
\begin{equation} \PSSlog^{\lambda} (\alpha_c t^{\vec{\v}}, Z_b)= [\PSSlog^{\lambda} (\alpha_c t^{\vec{\v}})-  \operatorname{PSS}_{\vec{\v}_{\emptyset}}(Z_b)] \end{equation} 

For the rest of this subsection, we exclusively consider admissible cocycles
$\alpha_ct^{\vec{\v}}$ for which we may take $Z_b=0$ for generic $J \in
\mathcal{J}_S(V)$. This applies for all tautologically admissible classes, but
also in other situations. For example: 

\begin{defn} 
    After choosing a triangulation $T$ of $\mathring{S}_I$, we can represent
    the fundamental class $\alpha$ on $\mathring{S}_I$ by a cycle given by
    taking the sum of all $2n-|I|$ dimensional simplices, $\alpha_c$. We refer
    to this as the fundamental cycle associated to $T$. 
\end{defn}

\begin{lem} \label{lem:trivialbounding}
    Fix a triangulation $T$ of $\mathring{S}_I$. If $\alpha=1 \in
    H^0(\mathring{S}_I)$, and $\alpha_c$ is the fundamental cycle associated to
    $T$, then equation \eqref{eq:chaineq} holds with $Z_b=0$.  
\end{lem} 
\begin{proof} 
    As in the proof of Lemma \ref{lem:degeneracy},  the evaluation
    $ev_{\infty,*}$ map factors through a lower dimensional chain
    $[\overline{\mathcal{M}}_{0,2}(M,\mathbf{D},\vec{\v}_I)^o/S^1]$. For
    generic $J$, the PSS moduli space misses this cycle completely, implying
    that
    \begin{align} 
        \PSS_{\vec{\v}_{\emptyset}}(\GWvc)=0 
    \end{align}
    at the chain level.  
\end{proof}

In view of Lemma \ref{lem:trivialbounding}, we will always set $Z_b=0$ and will
drop it from the notation when working with classes of the form
$\alpha_ct^{\v}$, with $\alpha_c$ a fundamental cycle.  We now state analogues
of  Lemmas \ref{lem:continuations4}, \ref{lem: welldef}. 

\begin{lem} 
    \label{lem:continuationsPSSH} 
    Fix an admissible cocycle $\alpha_ct^{\vec{\v}}$ for which we may take
    $Z_b=0$ for generic $J \in \mathcal{J}_S(V)$. For any $\lambda_1 \leq \lambda_2$,
    we have an equality  
    \begin{equation} 
        \PSSlog^{\lambda_{2}}(\alpha_ct^{\vec{\v}})= \mathfrak{c}_{\lambda_1,\lambda_2} \circ \PSSlog^{\lambda_{1}} (\alpha_c t^{\vec{\v}}) 
    \end{equation}
\end{lem}
\begin{proof} This is proven exactly as in Lemma \ref{lem:continuations4}. \end{proof}

It follows from the lemma that for any such cocycle $\alpha_c t^{\vec{\v}}$, we may obtain an element 
\begin{align} 
    \label{canonicalclass}
    \operatorname{PSS}_{log}(\alpha_c t^{\vec{\v}}):= \mathfrak{c}_{\lambda,\infty} \circ \PSSlog^{\lambda}(\alpha_ct^{\vec{\v}}) \in SH^*(X) 
\end{align} 
which is independent of $\lambda$. 

\begin{rem} 
    \label{rem: BormanSheridan} 
Let $(M, \D)$ be any pair where all $\kappa_i$ are equal to some $\kappa > 0$. 
For every component $D_i$, consider a fundamental chain $\alpha_i$ on the
corresponding circle
bundle $\mathring{S}_i$. We leave it as an exercise in this case to verify that $\alpha_i t^{\vec{\v}_i}$ is admissible for each $i$ in the sense of Definition \ref{defn:admissi} (second bullet point): for instance to verify Assumption (A1), suppose that $A \in H_2(M)_{\omega}$ satisfied $A = A_1 + A_2$ with $A_1, A_2 \in H_2(M)_{\omega}$. Then $A_1 \cdot \D + A_2 \cdot \D = \vec{\v}_i$ so $\sum_j (A_1 \cdot D_j + A_2 \cdot D_j) = 1$; on the the other hand since $[\omega] = \kappa (\sum_j [D_j])$, we have $\omega(A_1) = \kappa \sum_j (A_1 \cdot D_j) > 0$ and similarly $\omega(A_2) = \kappa \sum_j (A_2 \cdot D_j) > 0$, which is impossible.
Now Lemma \ref{lem:trivialbounding} further implies that $\alpha_i t^{\vec{\v}_i}$ is an admissible cocycle with $Z_b = 0$, hence we obtain via \eqref{canonicalclass} canonical
classes $\theta_i=\PSSlog (\alpha_it^{\vec{v}_i}) \in SH^*(X)$.  It is the classes
$\theta_{BS}=\sum_i \theta_i$  which are constructed in the work of \cite{bormansheridan}. 
\end{rem}

\begin{lem} 
    \label{lem: welldef2}     For any tautologically admissible class $\alpha t^{\vec{\v}}$, and given two cycles $\alpha_c$ and $\alpha_c'$ representing $\alpha$, we have that \begin{align} 
        [\operatorname{PSS}_{log}^{\lambda} (\alpha_c t^{\vec{\v}})]= [\operatorname{PSS}_{log}^{\lambda}(\alpha_c' t^{\vec{\v}})] \in HF^{*}(X\subset M, H_{M}^\lambda) 
    \end{align}   
\end{lem}
\begin{proof} This is proven exactly as in Lemma \ref{lem: welldef}.  \end{proof}

For topological pairs $(M,\mathbf{D})$, Lemma \ref{lem: welldef2} implies that there is a well-defined map: 
\begin{equation} \label{eq: maponclasses2} 
    \PSSlog^{\lambda}:
\QH^*(M,\D)_{\lfloor \lambda' \rfloor} \to HF^{*}(X\subset M ; H_M^{\lambda}) 
\end{equation}
for $\lambda'+\hbar<\lambda.$  The natural analogue of the diagram \eqref{eq: commutativediagram} (without ``+") also holds, giving rise to a map:  
\begin{align} \label{eq:logpsstopological}
    \operatorname{PSS}_{log}: \QH^*(M,\mathbf{D}) \to SH^*(X) 
\end{align}

\subsection{Interplay with the BV operator} 
In this subsection, we 
show that the
log PSS map intertwines the BV operator $\Delta$ on symplectic cohomology with the map $\Delta$ on $\QH^*(M,\D)$ defined in \eqref{logBVoperator}. Rather than prove a maximally general statement, we content ourself with proving such compatibility for log cohomology classes which have primitive multiplicity and are representable by (pushforwards of) fundamental chains (along proper maps), or more generally pseudocycles (see Remark \ref{pssBVpseudo}); variations on these arguments would handle more general log cohomology classes.

Let
$S = \mathbb{C}P^1 \setminus \lbrace 0 \rbrace$, thought of as a punctured
sphere, with a distinguished marked point $z_0 = \lbrace \infty \rbrace $ and a
cylindrical end around $\lbrace 0 \rbrace $ as before.  At $z_0$ we consider an
$S^1$-family of domains by adding a unit-tangent vector (or equivalently, a real tangent ray) at $z_0$,
$\vec{r}_{z_{0}} \in S_{z_0} S$ which we allow to take on arbitrary values
(previously a single such vector was fixed and used in our definition of the
enhanced evaluation map, see discussion above
\eqref{eq:enhancedevaluationpss}).
We may, in complete analogy with log PSS, define the
moduli spaces
$\overline{\mc{M}}_{S^1}(\vec{\v}, \alpha_c, x_0)$ of maps from this varying
family of domains with intersection multiplicity, enhanced evaluation
constraint (with respect to $\vec{r}_{z_0}$), and Floer asympotic conditions
specified by $\vec{\v}$, $\alpha_c$ and $x_0$ respectively, satisfying Floer's
equation (for some generic $J$).  We can (also as before) define an operation
$\operatorname{PSS}_{log,S^1}^{\lambda}$ by the formula 
\begin{equation}
    \operatorname{PSS}_{log,S^1}^{\lambda}(\alpha_c t^{\vec{\v}})=
    \sum_{x_0, |x_0| - |\alpha t^{\vec{v}}|=-1,} \sum_{u \in
    \overline{\mc{M}}_{S^1}(\vec{\v}, \alpha_c, x_0)} \mu_u \in CF^*(X \subset M, H^{\lambda}_M)
\end{equation}
where $\mu_u: \R \ra \mathfrak{o}_{x_0}$ is the isomorphism of orientation
lines (and their $\K$-normalizations) induced by rigid elements $u$.  (an
analogue of Lemma \ref{lem:compactness} implies that the above count is
well-defined for generic $J$). Furthermore, 
\begin{lem}
    If $[\alpha_c t^{\vec{\v}}] \in \QH^*(M,\D)$ is an admissible class, then the element $\operatorname{PSS}_{log,S^1}^{\lambda}(\alpha_c t^{\vec{\v}})$ is a cocycle for the Floer differential, i.e., $\partial_{CF}(\operatorname{PSS}_{log,S^1}^{\lambda}(\alpha_c t^{\vec{\v}})) = 0$.
\end{lem}
\begin{proof}
The same analysis performed in Lemma \ref{lem:compactness}  tells us that, in the case $|x_0| - |\alpha t^{\vec{v}}|=0$, the Gromov-Floer compactification
$\overline{\mc{M}}_{S^1}(\vec{\v}, \alpha_c, x_0)$ is a compact 1-manifold-with-boundary
\begin{equation}\label{bd2s1}
    \partial\overline{\mc{M}}_{S^1}(\vec{\v}, \alpha_c, x_0)=
    \partial\overline{\mc{M}}_{S^1}(\vec{\v}, \alpha_c, x_0)_F:= \bigsqcup_{x',|x_0|-|x'|=1} \mc{M}_{S^1}(\vec{\v},\alpha_c, x') \times \mc{M}(x_0,x');
\end{equation}
in particular while as before there could be a sphere bubble appearing the limit of a
sequence of maps if $\vec{\v} = \vec{\v}_I$,  such a bubble now appears in the
(more usual) codimension 2 and hence does not appear in \eqref{bd2s1} (note that the last paragraph in the proof of Lemma
\ref{lem:compactness}, which explained why sphere bubbling occurred in codimension 1 in that case, no longer applies as we allow the
tangent ray at $z_0$ to vary in $\overline{\mc{M}}_{S^1}(\vec{\v}, \alpha_c, x_0)$). The desired statement then follows as usual by counting elements of $\partial\overline{\mc{M}}_{S^1}(\vec{\v}, \alpha_c, x_0)$.
\end{proof}
Note that $\alpha_c t^{\vec{\v}}$ is not necessarily \emph{tautologically} admissible 
in the above Lemma. For simplicity, we now consider a special class of such
$\alpha_c t^{\vec{\v}}$, those with $\alpha_c$ represented by a submanifold and with primitive
multiplicity $\vec{\v}:=\vec{\v}_I$.  Let $\vec{\v}_I$ be a primitive vector and  
\begin{align}
    \Gamma_{\vec{\v}_{I}}: S^1 \times \mathring{S}_I \to \mathring{S}_I 
 \end{align} 
the corresponding
$S^1$ action (see \eqref{gammaV}). Let $P$ be an oriented (possibly non-compact) manifold and $p: P
\to \mathring{S}_I$ be a proper map.
Choose a triangulation on $P$ and let
$\alpha_c$ denote the pushforward of the fundamental chain along the map $p$,
with $\alpha = [\alpha_c]$ the corresponding cohomology class. We then have a map 
\begin{align} \label{eq: actingpush} 
    \Gamma_p :S^1 \times P \to \mathring{S}_I 
\end{align} 
which is defined as the composition of maps
$\Gamma_p:=\Gamma_{\vec{\v}} \circ (\operatorname{id} \times p)$. Let $\sigma_c$
denote the pushforward of the fundamental chain (defined for example as the
Eilenberg-Zilber product of the fundamental chain on $S^1$ and the fundamental
chain on $P$) under $\Gamma_p$ and let $\sigma:=[\sigma_c]$ the corresponding
cohomology class. 

 \begin{lem} \label{lem: BVandPSS} 
     Let $P$, $\alpha_c$ be as above and assume that $\alpha t^{\vec{\v}_I}$ 
     is an admissible class in $\QH^*(M,\D)$. Then, for any bounding cochain $Z_b$ for $\alpha_c t^{\vec{\v}_I}$ as in \S \ref{subsec:lifting}, there is a cohomological equality
     \begin{align} \label{eq: BVPSSeq} [\Delta \circ \PSSlog^{\lambda}(\alpha_c t^{\vec{\v}_I}, Z_b)]=[\operatorname{PSS}^{\lambda}_{log,S^1}(\alpha_c t^{\vec{\v}_I})]. \end{align} 
 \end{lem} 
 \begin{proof} 
     The proof involves studying a version of the log PSS moduli space for a
     family of domains which, along with their Floer data, are parametrized by
     $(r,q) \in S^1 \times [0,\infty]$. To simplify the notation, we will denote $H^{\lambda}_M$ by $H_t$ (as usual $J_t$ will be an almost complex structure needed to define the Floer complex). We fix an extension of $(H_t, J_t)$ to a family of surface dependent Hamiltonians/almost complex structures $(H_{s,t, r}, J_{s,t,r})$ used to define the BV operator as in \eqref{eq:FloerS1}. Finally, over $S$, we fix the subclosed one form $\beta$ as in \eqref{eq:sco} as well as $J_S \in \mathcal{J}_S(V)$. 

     For each point $\nu = (r,q) \in S^1 \times
     [0,\infty)$, we associate the domain $S = \C P^1 \setminus \lbrace 0 \rbrace$,
     thought of as a punctured sphere with negative end along with\footnote{Strictly speaking, to make the discussion compatible with the general framework outlined in Section 2, we should make the negative cylindrical end $\epsilon_r$ around $z=0$ depend on $r \in S^1$. However, this would make the notation more cumbersome.}    \begin{itemize}
         \item a distinguished marked point $z_0$ at $\infty$  along with a
     distinguished tangent direction at $z_0$ pointing in the positive real
     direction.
   \item A surface dependent almost complex structure $J_{\nu}$ which agrees with some surface independent $J_0$ in a neighborhood of $z_0$ and agrees with $J_{t-r}$ along the cylindrical end. 

    \item A perturbation one form $K_\nu$ which vanishes in a neighborhood of $z_0$
        and which along the cylindrical end, is equal, in some neighborhood (which will depend on $q$) of $-\infty$, to:
        \begin{align} 
            K_\nu= X_{H(t-r,x)}\otimes dt.
        \end{align}  
\end{itemize}
Abbreviating $X_r=X_{H_{s,t,r}}$ and $J_r=J_{s,t,r}$, we assume that the $\nu$-dependent choice of Floer data satisfies the following consistency conditions
\begin{itemize}
    \item for $q\gg 0$, the datum $(K_\nu, J_{\nu})$ is obtained by taking finite connected sum (by an amount depending on $q$ and limiting as $q \to \infty$ to a nodal connect sum) of the Floer data $(X_{H_{t}}\otimes \beta, J_S)$ and $(X_{r}\otimes dt, J_{r})$ along the cylindrical ends. That is,  for $q \gg 0$, the Floer data coincides at a point $(s,t)$ in (cylindrical coordinates) with $(X_{s+N + q, t, r} \otimes dt, J_{s + N + q, t, r})$ for some sufficiently large fixed $N$ (chosen so that $s \geq N$ lies in the region where $(H_{s,t,r}, J_{s,t,r})$ coincides with $(H_t, J_t)$). 
    \item Over $q=0$, the datum on the cylindral end agrees with the $r$-clockwise rotation of the data $(H_{M}^{\lambda}, J_S)$, i.e., $(K_\nu, J_\nu) = (X_{H_M^{\lambda}(t-r)} \otimes \beta, J_S(s,t-r))$. 
\end{itemize}
For any $\vec{\v}$ and $x_0$, let $\mc{P}^{S^{1}}(\vec{\v}, x_0)$ denote the moduli space of maps $u: S \to M$ which solve the equation 
        \begin{equation}
(du - X_{K_{\nu}})^{0,1} = 0
\end{equation}
and which satisfy \eqref{eq:incidence} as well as the asymptotic condition
\begin{equation}
\lim_{s \ra -\infty} u(\epsilon(s,t)) = x^{(r)}_0.
\end{equation}
  (For the present argument, we will only need to consider $\vec{\v}= \vec{\v}_I$ or $\vec{\v}=\vec{\v}_{\emptyset}.$) Form the moduli spaces  
  \begin{align} \label{eq:lalalalala} 
  \mc{P}^{S^{1}}(\vec{\v}_I, \alpha_c, x_0):= \mc{P}^{S^{1}}(\vec{\v}_I, x_0) \times_{\Evzo} \alpha_c. 
  \end{align} 
  \begin{align} 
      \label{eq:plusstratumIII}  \mc{P}^{S^{1}}(\vec{\v}_\emptyset, Z_b, x_0):= \mc{P}^{S^{1}}(\vec{\v}_\emptyset, x_0) \times_{ev} Z_b. 
\end{align}
and suppose $|\alpha_c t^{\vec{\v}_I}|-|x_0|=1$, so that both moduli spaces above (recalling that $Z_b$ is a bounding cochain for $\alpha_c t^{\vec{\v}_I}$) are 1-dimensional for generic choices of Floer data satisfying the above conditions (by a standard variation of the arguments in \S \ref{section:transversality}). The key step in the proof is to analyze the boundary strata of their Gromov compactifications; we'll summarize the result of appealing to Gromov compactness as well as the relevant variant of Lemma \ref{lem:compactness}. First, in the compactification of \eqref{eq:lalalalala}, 4 types of boundary strata occur: 

\emph{(Type I:)}  In the limit as $q\to \infty$, when solutions converge to elements of the fiber product 
 \begin{align} 
     \bigsqcup_{|x_1|-|x_0|=1} \mc{M}(\vec{\v}_I, \alpha_c, x_1) \times \mathcal{M}_{S^1}(x_0,x_1) 
 \end{align} 
 where  $\mathcal{M}_{S^1}(x_0,x_1)$ denotes the space of solutions to  \eqref{eq:FloerS1}. 
 The signed count of elements in this fiber product defines the composition $\Delta \circ \PSSlog^{\lambda}(\alpha_c t^{\vec{\v}_I})$.

 \emph{(Type II:)} The moduli space restricted to $q=0$ appears as part of the boundary, and can be identified with $\mc{M}_{S^1}(\vec{\v}_I, \alpha_c, x_0)$. To see this, note
that in the definition of $\mc{M}_{S^1}(\vec{\v}_I, \alpha_c, x_0)$, if we want to fix the
vector $\vec{r}_{z_{0}}$ to point in the positive real direction, we can rotate
the surface $S$ by $(s,t) \to (s,t-r)$. This is equivalent to solving Floer's
equation for the rotated inhomogeneous term $(X_{H_M^{\lambda}(t-r)} \otimes \beta)$ with rotated almost complex structure $J_S(s,t-r)$ such that the solutions are asymptotic to the output orbit
$x^{(r)}_0$. This boundary accounts for (on the chain level) the right-hand side of \eqref{eq: BVPSSeq}.

\emph{(Type III:)} This is the stratum where sphere bubbles form as in the proof of Lemma \ref{lem:compactness}. It is given by: 
\begin{align} 
\label{eq:plusstratumII}  
\mc{P}^{S^{1}}(\vec{\v}_\emptyset,\underline{GW_{\vec{\v}_I}(\alpha_c)}, x_0):= \mc{P}^{S^{1}}(\vec{\v}_\emptyset, x_0) \times_{ev}\underline{GW_{\vec{\v}_I}(\alpha_c)}. 
\end{align} 

 \emph{(Type IV:)} Finally, we have ordinary cylindrical breaking at intermediate values of $(r,q)$ as well, which gives rise to:
 \begin{align} 
     \bigsqcup_{|x_0|-|x_1|=1} \bigcup_{r \in S^1} \mc{P}^{S^{1}}_r (\vec{\v}_I, \alpha_c, x_1) \times \mathcal{M}_{r}(x_0,x_1) \cong  \bigsqcup_{|x_0|-|x_1|=1} \mc{P}^{S^1}(\vec{\v}_I, \alpha_c, x_1 \times \mathcal{M}(x_0, x_1) 
 \end{align} 
 where $\mathcal{M}_{r}(x_0,x_1)$ is the moduli space of twisted Floer trajectories introduced just before \eqref{eq:FloerSBV} and $\mc{P}^{S^{1}}_r (\vec{\v}_I, \alpha_c, x_1)$ denotes the moduli space over some fixed $r \in S^1$.  Since $\mathcal{M}_r(x_0, x_1)$ can be canonically identified with $\mathcal{M}(x_0, x_1)$, we obtain the second identification, and conclude that this contribution is exact.

 Next, the 1-dimensional moduli spaces \eqref{eq:plusstratumIII} again have (for generic choices of data)  a compactification involving strata over $q=\infty$ which give strata of the same type as \emph{(Type I)} above as well as strata with cylindrical breaking as in \emph{(Type IV)} above.\footnote{Unlike the above case, the boundary over $q=0$ is generically empty because $\mathcal{M}(\vec{\v}_\emptyset, Z_b, x_0)$ is empty; this is one explanation for why there is no bounding cochain on the right-hand side of \eqref{eq: BVPSSeq} even though there is one on the left.} There is also a third boundary stratum isomorphic to \eqref{eq:plusstratumII} (though it is important to note that it arises from curves whose evaluation at $z_0$ lies on the boundary of $Z_b$, which by definition is a nullbordism for $\underline{GW_{\vec{\v}_I}(\alpha_c)}$, rather than from sphere bubbling).

With this analysis in place, we complete the proof of \eqref{eq: BVPSSeq}. As in the proof of Lemma \ref{lem:continuations4}, we may use the count of \emph{rigid} elements of $\mc{P}^{S^{1}}(\vec{\v}_I, \alpha_c, y)$ and  $\mc{P}^{S^{1}}(\vec{\v}_\emptyset, Z_b, y)$ to define an operator $T^{S^{1}}$ of degree -2 via the formulae:\begin{align} 
    \operatorname{T}^{S^{1}}(\alpha_c t^{\vec{\v}}) := \sum_{y, |y|-|\alpha_c t^{\vec{\v}}|=-2} \sum_{u \in \mc{P}^{S^{1}}(\vec{\v},\alpha_c, y)} \mu_u \\
 \operatorname{T}^{S^{1}}(Z_b) := \sum_{y, |y|-|Z_b|=-2} \sum_{u \in  \mc{P}^{S^{1}}(\vec{\v}_\emptyset,Z_b, y)} \mu_u  
\end{align}  The above analysis of the boundary of the compactified 1-dimensional moduli spaces then implies that:
  \begin{align} \label{eq: strongerBVPSS} \Delta \circ \PSSlog^{\lambda}(\alpha_c t^{\vec{\v}_I})- \Delta \circ \PSS_{\vec{\v_\emptyset}}(Z_b)-\operatorname{PSS}^{\lambda}_{log,S^1}(\alpha_c t^{\vec{\v}_I})= \partial_{CF} \circ (T^{S^1}(\alpha_c t^{\vec{\v}_I})-T^{S^1}(Z_b)) \end{align}

In deducing this equation, we have used the fact that the contributions of the two boundary strata isomorphic to \eqref{eq:plusstratumII} cancel each other out. Equation \eqref{eq: strongerBVPSS} descends to \eqref{eq: BVPSSeq} on the level of cohomology, completing the proof.
\end{proof} 

\begin{lem} 
\label{lem:S1eq} Let $p: P \to \SIo$ be as in Lemma \ref{lem: BVandPSS}. We have an equality 
\begin{align} \label{eq:S1eq} \operatorname{PSS}_{log,S^1}^{\lambda}(\alpha_{c} t^{\vec{\v}_I})= \operatorname{PSS}_{log}^{\lambda}(\sigma_{c} t^{\vec{\v}_I}) 
\end{align} 
\end{lem}
\begin{proof} 
  By definition, the moduli space $\mc{M}(\vec{\v}_I, \sigma_c, x_0)$ is a subspace of $\mc{M}(\vec{\v}_I, x_0) \times (S^1 \times P)$. Denoting points in this product by $(u, (\theta,x))$, then $\mc{M}(\vec{\v}_I, \sigma_c, x_0)$ is the subset of maps $u \in \mc{M}(\vec{\v}_I, x_0)$ where $\Evzo = \theta \cdot p(x)$. Meanwhile, $\mc{M}_{S^1}(\vec{\v}_I, \alpha_c, x_0)$ is similarly a subspace of $(\mc{M}(\vec{\v}_I, x_0) \times S_{z_0}\mathbb{C}) \times P$.  We identify  $S_{z_0}\mathbb{C} \cong S^1$ by viewing $\vec{r}_{z_0}$ as being the rotation of the positive real ray by some $\theta^{-1}$ for $\theta \in S^1$. If the enhanced evaluation with respect to $\vec{r}_{z_0}$ is $p(x)$, then the enhanced evaluation with respect to the positive real ray is $\theta \cdot p(x)$ (and vice-versa). Thus, the map  $((u,\theta),x) \mapsto (u, (\theta,x))$ defines a canonical bijection of moduli spaces of curves used to define the
left and right-hand sides of equation \eqref{eq:S1eq}. Moreover, it is
easy to check that the conditions for these moduli spaces to be cut out
transversely are equivalent and that the bijection preserves (relative)
orientations.  
\end{proof}
Altogether, Lemmas \ref{lem: BVandPSS} and \ref{lem:S1eq} along with the definition \eqref{logBVoperator} of the BV operator on log cohomology immediately imply
\begin{cor}\label{corpssBV}
    Let $\alpha t^{\vec{\v}_I}$ be a primitive admissible class with $\alpha = [\alpha_c]$ representable by the fundamental chain $p_*[P]$ associated to a proper (not necessarily compact) map $P \to \mathring{S}_I$, 
    and let $Z_b$ be a bounding cochain for $\alpha_c t^{\vec{\v}_I}$. Then, on the level of cohomology, there is an equality:
    \begin{equation}
        [\PSSlog^{\lambda}(\Delta (\alpha_c t^{\vec{\v}_I}))] =  [\Delta \circ \PSSlog^{\lambda}(\alpha_c t^{\vec{\v}_I}, Z_b)].
    \end{equation}
    \qed
\end{cor}

\begin{rem} \label{pssBVpseudo}
    Corollary \ref{corpssBV} (and in particular Lemma \ref{lem: BVandPSS} and
    \ref{lem:S1eq}) immediately generalize (with the same proof) to the case
    that $\alpha$ is representable by a pseudocycle rather than the fundamental
    chain of a smooth map. When our divisor $\D = D$ consists of a single
    smooth component, $SD$ is compact  and we can represent
    every cohomology class by a psuedocycle, allowing us to apply this variant of
    Corollary \ref{corpssBV} to all primitive log cohomology classes. We will make use 
of this in the proof of Theorem \ref{thm: dilationcrit}.
\end{rem}

\def\lra{\longrightarrow}

\section{Quasi-dilations in string topology} \label{sec:stringtopology}
For simplicity, we assume
\begin{equation}\label{eq:manifoldhypotheses}
    \textrm{all manifolds $Q$ appearing in this section are connected, closed, oriented and Spin.}\footnote{As most of our analysis will be in dimension three, we recall that any closed, oriented three manifold is Spin.}
\end{equation}
The main goal of this section is to explain how the existence of quasi-dilations in the string topology of such $Q$ imply topological constraints on the manifold, particularly in dimension three. 

\subsection{Review of string topology} \label{subsection:stringtopback} 
Throughout this subsection let $n=\operatorname{dim}(Q)$ and $\K$ will be an arbitrary commutative ring. For any such $Q$, Chas and Sullivan \cite{Chas:aa, CohenJones} constructed a BV algebra structure on the homology of the free loop space of $Q$, $\mathbb{H}^*(\mathcal{L}Q,\K) = H_{n-*}(\mathcal{L}Q,\mathbf{k})$, called the \emph{loop homology algebra} (over $\mathbf{k}$). The basic manifestations of this structure most relevant to us are:
\begin{itemize}
    \item $\mathbb{H}^*(\mathcal{L} Q,\K)$ has a {\em unit element} $1 \in \mathbb{H}^0(\mathcal{L} Q) \cong H_{n}(\mathcal{L}Q,\mathbf{k})$ corresponding to a choice of fundamental class $[Q]$ via the inclusion of constant loops $Q \hookrightarrow \mathcal{L} Q$; and
    \item $\mathbb{H}^*(\mathcal{L} Q, \K)$ has a BV-operator, which is especially easy to describe. Denote the canonical rotation action on the loop space by 
 \begin{align} \label{eq:looprot} \Gamma: S^1 \times \mathcal{L}Q \to \mathcal{L}Q\end{align}  
 The BV-operator is given by 
 \begin{align}  
     \Delta(\alpha)=\Gamma_*(\epsilon \otimes \alpha), 
 \end{align}  
 where $\epsilon \in H_1(S^1)$ is a fixed choice of fundamental class.
 \end{itemize}
 There is an isomorphism of groups \cite{Viterbo:1996kx, Salamon:2006ys, AbSch1, Abouzaid:2015ad}
\begin{equation}\label{eq:abschmap}
\mathcal{AS}: SH^*(T^*Q,\mathbf{k}) \to \mathbb{H}^*(\mathcal{L}Q,\K) 
\end{equation}
which intertwines algebra structures \cite{AbSch2}.  Although it is not
explicitly mentioned, the particular isomorphism \eqref{eq:abschmap} given in
\cite{AbSch1} can be straightforwardly shown to intertwine BV operators as
well, hence is an isomorphism of BV-algebras (alternatively, see the more
recent \cite{Abouzaid:2015ad}). Fixing the base-point $1 \in S^1$, evaluating a loop at $1$ produces an
evaluation morphism 
\[ \operatorname{ev}: \mathcal{L} Q \to Q 
\] 
Let $\mathbb{H}^*(Q)$
denote the intersection ring of $Q$ (meaning the homology $H_{n-*}(Q, \K)$ equipped
with the intersection product). The induced map 
\[
    \operatorname{ev}_*: \mathbb{H}^*(\mathcal{L}Q, \K) \to \mathbb{H}^*(Q, \K)
\] 
is a unital ring homomorphism.

Fix a BV-algebra $(A, \Delta)$ over $\K$. Let $\alpha \in A^0$ and $\Psi \in
A^1$ be two elements satisfying  
\begin{equation} \label{eq:dilation} \Delta(\Psi)=\alpha \end{equation}
We say that $(\Psi,\alpha)$ is 
\begin{itemize}
 \item a \emph{quasi-dilation} if $\alpha$ is a unit with respect to the algebra structure in $A$;
    \item a \emph{dilation} if $\alpha=1$. 
\end{itemize}

Whenever the string topology BV-algebra $\mathbb{H}^*(\mathcal{L}Q,\K)$ of a manifold $Q$ carries a pair $(\Psi,\alpha)$ which define a (quasi-) dilation, we will say that $Q$ admits a (quasi-) dilation over $\K$. 

\begin{rem} The notion of
    quasi-dilation is due to Seidel \cite{catdyn}, who
    introduced it in the equivalent form 
    \begin{equation} \label{eq:dilationS}
        \Delta(\alpha\Psi_S)= \alpha.
    \end{equation} 
    Equation \eqref{eq:dilationS}, which is
    related to the dilation condition \eqref{eq:dilation} by
    replacing $\alpha \Psi_S$ with $\Psi$
    (which is an invertible operation), has a natural geometric
    interpretation. Let $Z$ be a smooth algebraic variety equipped with an
    invertible function $\alpha$, a holomorphic volume form $\omega$, and a
    vector-field $\Psi_S$. The condition that the vector field $\Psi_S$
    dilates $\alpha \omega$ may be written as
    \[\mathcal{L}_{\Psi_S}(\alpha\omega)=\alpha \omega. \] Under a standard 
    dictionary between commutative and non-commutative geometry (see e.g., Lecture 19 of \cite{catdyn}), 
    this translates to the equation \eqref{eq:dilationS}.  
\end{rem}
We will also consider one further definition, which is specific to the string topology setup, but useful for formulating Lemmas \ref{lem: dominate}, \ref{lem: charS3}  and in the proof of Corollary \ref{cor: Vitdil}.
\begin{defn} For any (closed, oriented, spin) manifold $Q$, let $\alpha \in \mathbb{H}^0(\mathcal{L} Q,\K)$, $\Psi \in \mathbb{H}^1(\mathcal{L} Q, \K)$ be two elements satisfying  \eqref{eq:dilation}. We say that $(\Psi, \alpha)$ is a \emph{pseudo-dilation} if $\operatorname{ev}_*(\alpha)$ lies in $\K^{\times}$ inside $\mathbb{H}^0(Q,\K)=\K.$ Similarly to the (quasi-) dilation case, whenever such $(\Psi, \alpha)$ exist, we will say that $Q$ admits a pseudo-dilation (over $\K$). \end{defn}
Evidently if $(\Psi, \alpha)$ is a quasi-dilation, then it is a pseudo-dilation. \vskip 5 pt

It will be helpful at several points later on to discuss the case where $Q$ is a $K(G,1)$ space in slightly more detail. Let $C(G)$ denote the set of conjugacy classes of $G$ and for any element $g \in G$, let $C_g$ denote the centralizer of $g$. Following e.g. Section 10.1 of \cite{Chas:aa}, we have that, up to homotopy, the loop space decomposes as a disjoint union of connected components \begin{align} \label{eq:loopdecomp} \mathcal{L}Q \cong \bigsqcup_{[g] \in C(G)} K(C_g,1)  \end{align}
where $C_g$ denotes the centralizer of a representative $g$ of the given conjugacy class $[g]$. Notice that each of these components $K(C_g,1)$  is a covering space of the original $Q$ and thus we immediately see that $H_*(\mathcal{L}Q)$ vanishes for $*>n$ or equivalently that  $\mathbb{H}^*(\mathcal{L}Q)$ is concentrated in non-negative degree.  The decomposition \eqref{eq:loopdecomp} also easily yields the following well-known description of $\mathbb{H}^0(\mathcal{L}Q)$ (which is additively $H_n(\mathcal{L}Q))$:

\begin{lem} \label{lem:centre} There is a ring isomorphism $\mathbb{H}^0(\mathcal{L}
Q) \cong \mathcal{Z}_{\mathbf{k}}(G)$, the center of the group algebra
$\mathbf{k}[G]$.  \end{lem}
\begin{proof}[Proof Sketch] 
  $\mathcal{Z}_{\mathbf{k}}(G)$ can be presented more explicitly as follows:
    \label{lem: groupring} for any group $G$, and any commutative ring
    $\mathbf{k}$, the center $\mathcal{Z}_{\mathbf{k}}(G)$ has a free $\mathbf{k}$-basis $\lbrace \sum_{g \in C} g \rbrace$ indexed by finite conjugacy
    classes $C$ of $G$. Similarly, a given connected component of the loop space has homological dimension $n$ precisely when, using the same notation as in \eqref{eq:loopdecomp}, $C_g$ has finite index in $G$ (so that $K(C_g,1)$ is a \emph{finite} cover of $K(G,1)$). Hence, $\mathbb{H}^0(\mathcal{L}Q)$ also has a canonical basis indexed by finite conjugacy classes. It is not difficult to check that this isomorphism preserves ring structures on both sides. \end{proof}

We also have the following useful observation of Seidel-Solomon (Example 6.2 of \cite{Seidel:2010uq}):

\begin{lem} \label{lem: nodiK} Suppose $Q$ is a $K(G,1)$ space. Then $Q$ does not admit a dilation over any $\K.$ \end{lem}
\begin{proof} The BV operator respects the decomposition of the loop space into connected components. So, for the purpose ruling out dilations it suffices to analyze the component of contractible loops which, as can be seen from \eqref{eq:loopdecomp}, retracts onto the space of \emph{constant} loops. The BV-operator therefore acts trivially on homology classes arising from this component and there are no dilations. \end{proof} 

Let us determine which (as usual closed, oriented) surfaces admit quasi-dilations. It follows immediately from the main calculation in \cite{Menichi} that $\mathbb{H}^*(\mathcal{L}S^2,\mathbb{Z})$ admits a quasi-dilation over $\mathbb{Z}$ which becomes a dilation after tensoring with any field $\K$ of characteristic not equal to two. The remaining surfaces are aspherical and hence can be analyzed using \eqref{eq:loopdecomp} (together with the related Lemmas \ref{lem:centre} and \ref{lem: nodiK}). When the genus $g=1$, $\mathbb{H}^*(\mathcal{L}T^2,\mathbb{Z})$ admits a quasi-dilation as can be seen by explicit computation using \eqref{eq:loopdecomp}; $\mathcal{Z}_{\mathbf{k}}(G)= \mathbf{k}[G]$, which is isomorphic to a Laurent polynomial ring in two variables and one may take $\alpha$ to be one of the generators. On the other hand if $g \geq 2$, it is well-known that the center of the group algebra $\mathbf{k}[\pi_1(\Sigma_g)]$ is trivial (see e.g. page 564 of \cite{MR2379052}) so $\mathbb{H}^0(\mathcal{L}
\Sigma_g)=\K$. It follows that if $g\geq 2$, there are no quasi-dilations in $\mathbb{H}^*(\mathcal{L} \Sigma_g, \K)$, because there are no dilations. 

Our description of 2-manifolds admitting quasi-dilations uses the classification of surfaces. To obtain a similar description of 3-manifolds admitting quasi-dilations, we need to import some tools from 3-manifold topology. The next subsection recalls the necessary background and \S \ref{subsection:3dil} carries out the classification. 

\subsection{Some 3-manifold topology} 
We begin with the following definition: 

\begin{defn} We say that a closed, oriented manifold $Q$ is \emph{dominated} by a closed oriented manifold $\hat{Q}$ if there is a non-zero degree map $f: \hat{Q} \to Q$. We say that $Q$ is \emph{1-dominated} if the map $f$ has degree 1. \end{defn}

In this subsection, we will study the following question:

\begin{ques} Which 3-manifolds are dominated (respectively 1-dominated) by products $S^1 \times \hat{B}$ where $\hat{B}$ is an oriented Riemann surface? \end{ques} 

Let us first recall the notion of a Seifert manifold, which will arise at several points in our discussion. Let $Q_f$ denote the mapping torus of the rotation $f: D^2 \cong D^2$ by $2\pi p/q$ with $p/q \in \mathbb{Q}$ and $p,q$ relatively prime. $Q_f$ is homeomorphic to $S^1 \times D^2$ and is decomposed into disjoint circles: a special fiber which is the image in the quotient space $[0,1] \times \lbrace 0 \rbrace$ and the regular fibers which are the images of the union of $q$ segments of the form $[0,1] \times \lbrace x_i \rbrace$ where the $x_i \in D^2$ form an orbit under rotation. A Seifert fibering of a manifold $Q$ is a decomposition of $Q$ into disjoint circles (``fibers") such that each fiber has a neighborhood which is fiber preserving diffeomorphic to a neighborhood of a fiber in one of the above standard fiberings of $S^1 \times D^2$ (for some choice of $p/q$).

 A Seifert manifold $Q$ is a 3-manifold which admits a Seifert fibration (see e.g. \cite{Neunotes} for a nice introduction to these manifolds). The orbit space given by collapsing each of the fibers to a point can be given the structure of a two dimensional orbifold $B_Q$. Seifert manifolds are easily classified by their (un-normalized) Seifert invariants $(g, (\alpha_1,\beta_1), \cdots (\alpha_m,\beta_m))$, where $g$ is the genus\footnote{For nonorientable surfaces, one regards the genus as being minus the number of $\mathbb{R}P^2$ connect summands needed to construct $B_Q$.}  of the surface (underlying) $B_Q$ and the $(\alpha_i,\beta_i)$ are relatively prime integers. As we will only need these invariants in the proof of Lemma \ref{thm:Rong}, we don't recall their definition here but refer to page 14 of \cite{Neunotes}.

 Let $h \in \pi_1(Q)$ denote the class of a regular fiber of the Seifert fibering. Then it follows from the standard presentation of the fundamental group of a Seifert manifold (page 34 of \cite{Neunotes}) that the subgroup $\langle h\rangle$ generated by $h$ is a normal subgroup. The following celebrated result shows that this property characterizes Seifert fibered 3-manifolds:

\begin{thm} \cite{MR1189862, MR1296353} \label{thm:SFC} Let $Q$ be a closed oriented 3-manifold such that $\pi_1(Q)$ admits a non-trivial cyclic normal subgroup. Then $Q$ is Seifert fibered. If the center of $\pi_1(Q)$ is non-trivial, then $Q$ is Seifert fibered over an oriented base. \end{thm} 

\begin{rem} Building on this result, the reference \cite{MR2379052} shows something slightly stronger--- that a closed, oriented 3-manifold $Q$ with $\mathcal{Z}_{\mathbf{k}}(\pi_1(Q)) \neq \K$ is Seifert-fibered. Parallel to the discussion for 2-manifolds, Lemmas  \ref{lem:centre} and \ref{lem: nodiK} then immediately show that any aspherical  3-manifold which admits a quasi-dilation is Seifert-fibered. Our approach in \S \ref{subsection:3dil} handles all cases simultaneously but also relies on Theorem \ref{thm:SFC} (via  Theorem \ref{thm: productdomtop}). \end{rem}

The starting point for our analysis will be the following Theorem which relies on Theorem \ref{thm:SFC} as well as many other important developments in 3-manifold topology (notably \cite{Gromov, MR895623, Perelman1, Perelman2}). 

\begin{thm} 
    \label{thm: productdomtop} (Theorem 1 and Theorem 3 of \cite{MR3054316}) Let $\hat{B}$ be an oriented Riemann surface and let $Q$ be a closed, oriented 3-manifold dominated by $S^1 \times \hat{B}$. Then, either:
\begin{itemize}
\item  $Q$ is finitely covered by $S^1 \times B$ for $B$ closed and oriented of genus $\geq 1$ (and in particular aspherical), or
    
\item $Q$ admits a metric of positive scalar curvature. Equivalently (see the two paragraphs preceding Remark 1 on page 24 of \cite{MR3054316}), $Q$ is finitely covered by a connected sum $\#_n S^1 \times S^2$ (by convention, the case $n=0$ corresponds to $S^3$). 
\end{itemize}
 \end{thm}  

In order to classify 3-manifolds admiting quasi-dilations over $\mathbb{Z}$ (see Lemma \ref{lem: charS3}), we will need the following refinement of Theorem \ref{thm: productdomtop} which concerns 1-domination by $S^1 \times \hat{B}$. 

\begin{lem} \label{thm:Rong} Let $Q$ be a closed, oriented aspherical manifold which is 1-dominated by $S^1 \times \hat{B}$. Then $Q \cong S^1 \times B$ for $B$ closed and oriented of genus $\geq 1$. \end{lem}
\begin{proof} Because the map $f: S^1 \times \hat{B} \to Q$ has degree 1, it is surjective on $\pi_1$. Let $h \in \pi_1(S^1)$ denote a generator of $\pi_1(S^1).$ As the map does not factor through $\hat{B}$, we must have that $f_*(h) \neq 0 \in \pi_1(Q)$ and is central by surjectivity. By Theorem \ref{thm:SFC}, this implies that $Q$ is Seifert fibered. 

We may therefore apply the following simplified version of a result of Rong (Theorem 3.2 of \cite{Rong2} building on Corollary 3.3. of \cite{Rong}) which states that if $f: Q' \to Q$ is a map of degree 1 between Seifert manifolds and $Q'$ has Seifert invariants $(g',(\alpha_1',\beta_1'), \cdots, (\alpha_l',\beta_l'))$ then $Q$ has Seifert invariants $(g, (\alpha_1,\beta_1), \cdots (\alpha_m,\beta_m))$ with $l \geq m$ (and $g' \geq g$ but this plays no role). In our present (very simple) case, the Seifert invariants of $S^1 \times \hat{B}$ are empty and so $Q$ also admits a Seifert fibration with empty Seifert invariants. This immediately implies that $Q \cong S^1 \times B.$   \end{proof} 

Finally, in the proofs of Lemmas \ref{lem: charS3} and \ref{lem:dilatp}, we will also require the following two Lemmas concerning spherical space forms (i.e. a quotient of $S^3$ by a finite subgroup of $SO(4)$ acting freely on $S^3$ by rotations) and their fundamental groups.  

\begin{lem} \label{lem:covertricks} Let $Q$ be a spherical space form. For any prime $p$ which divides $|\pi_1(Q)|$ and any map $f: S^1 \times B \to Q$, $p$ divides $\operatorname{deg}(f)$. \end{lem}

\begin{proof} We first note the following fact:  Let $\pi:C \to S^1 \times B$ be a finite covering where $B$ is an oriented Riemann surface. Then $C \cong S^1 \times B_C$ for some Riemann surface $B_C$. To see this, observe that the manifold $C$ inherits a canonical Seifert fibration structure $C \to B_C$ by Lemma 8.1 of \cite{Neunotes} and there is an induced covering map $B_C \to B$ (see Lemma 4.4. of \cite{MR3874954} for a detailed proof of this fact). It follows that $B_C$ is an ordinary Riemann surface(without orbifold structure) and the Euler number of the circle bundle $C \to B_C$ is zero. Therefore $C \cong S^1 \times B_C.$

With this established, let $p$ be a prime as in the statement of the Lemma. We first note that it suffices to assume that the map $f$ in the statement of the Lemma is surjective on $\pi_1$. To see this, suppose that it is not surjective on $\pi_1$, then the map $f$ factors through a covering space $\tilde{Q} \to Q$ 
\[ \begin{tikzcd}
S^1 \times B \arrow{r}{\tilde{f}}  \arrow{rd}{f} 
  & \tilde{Q} \arrow{d}{h} \\
    & Q
\end{tikzcd} \]
so that $\tilde{f}$ is $\pi_1$-surjective. Then, if $p$ no longer divides $\pi_1(\tilde{Q})$, this means that the degree of $h$ must be divisible by $p$ which implies that the degree of $f$ is divisible by $p$ and we are done. Otherwise, replacing $Q$ by $\tilde{Q}$ gives the desired reduction. 
 
 Having reduced to this case, we assume for the rest of the proof that $f$ is surjective on $\pi_1$. The next step is to reduce to the case where $Q$ is a lens space. To do this, note that by Cauchy's theorem, we may find a cyclic subgroup $\langle g\rangle$ of order $p$ in $\pi_1(Q)$ and we let $\tilde{Q}$ be the covering space of $Q$ corresponding to this subgroup. Using the fact noted at the beginning of this proof that any finite covering space of $S^1 \times B$ is of the same form, we may, by passing to covering spaces in both the source and the target of the map, obtain a Cartesian square:
\[
\begin{tikzcd}
S^1 \times \tilde{B} \arrow{r}{\tilde{f}} \arrow[swap]{d}{h'} & \tilde{Q} \arrow{d}{h} \\
S^1 \times B  \arrow{r}{f} & Q
\end{tikzcd}
\]
with $\tilde{Q}$ a lens space and $\operatorname{deg}(\tilde{f})= \operatorname{deg}(f).$ Again replacing $Q$ by $\tilde{Q}$, it suffices to consider the case where $Q$ is a lens space.

Finally we treat the case where $Q$ is a lens space.  Let
$\mathcal{B}$ denote the Bockstein homomorphism associated to the coefficient sequence 
$0 \ra \Z/p \ra \Z/p^2 \ra \Z/p \ra 0$. The Bockstein long exact sequence shows that $H^1(Q,\mathbb{Z}/ p \mathbb{Z}) \to H^2(Q,\mathbb{Z}/ p \mathbb{Z})$ is an isomorphism. This (together with the cup product structure on the cohomology of lens spaces) shows that there is a
generator $\alpha \in H^1(Q,\mathbb{Z}/ p \mathbb{Z})$ such that $\alpha \cup
\mathcal{B}(\alpha)$ is a generator of $H^3(Q,\mathbb{Z}/p\mathbb{Z})$.  Since the Bockstein
homomorphisms are zero for $S^1\times B$, the Lemma follows from the
naturality of the cup-products and Bockstein homomorphisms.
\end{proof}

\begin{lem}\label{lem:swanlemma} For any finite group $\pi$ which acts freely on $S^3$, \begin{align} \label{eq:classifyingcalc} H_3(B\pi,\mathbb{Z}) \cong \mathbb{Z}/|\pi| \end{align} where $B\pi$ denotes the classifying space of $\pi$ and $|\pi|$ denotes its cardinality. \end{lem}
\begin{proof} By Lemma 3.1 and the discussion after the first Definition on page 268 of \cite{MR0124895} (see Section 6.1 and Example 6.3 of \cite{MR2655176} for a nice exposition) the Tate cohomology groups $\hat{H}^*(\pi,\mathbb{Z})$ are 4-periodic for any finite group which acts freely on $S^3$. In particular, we have \begin{align} \label{eq: periodicity} \hat{H}^{-4}(\pi,\mathbb{Z}) \cong \hat{H}^0(\pi,\mathbb{Z}) \cong \mathbb{Z}/|\pi| \end{align} where the second isomorphism in \eqref{eq: periodicity} is Equation (6.2) of \cite{MR2655176}. But by definition of Tate cohomology, $\hat{H}^{-4}(\pi,\mathbb{Z}):= H_3(B\pi,\mathbb{Z})$ and so \eqref{eq: periodicity} immediately implies \eqref{eq:classifyingcalc}. \end{proof} 

\subsection{Classifying 3-manifolds which admit quasi-dilations} \label{subsection:3dil}

We now turn to describing how the results of the previous section constrain the topology of 3-manifolds admitting quasi-dilations.  

\begin{lem} \label{lem: dominate} 
Suppose that a closed oriented 3-manifold $Q$ admits a pseudo-dilation over $\mathbb{Q}$.  Then $Q$ is dominated by $S^1 \times \hat{B}$ for some closed, orientable Riemann surface $\hat{B}$. 
\end{lem}

\begin{proof} 

    By definition the element $\Psi$  is a class $\mathbb{H}^1(\mathcal{L} Q) = H_2(\mathcal{L}Q)$.  Without loss of generality assume that $\Psi$ is concentrated in the homology of a single connected component of the loop space $H_2(\mathcal{L}_{c}Q)$. By \cite{MR61823} (see also \cite{McDuff:2004aa}*{Ex. 6.5.4} for a simple argument in degree 2), we may represent a non-zero multiple of this class (i.e. $N \Psi$ for some non-zero $N \in \mathbb{Z}$) by a map from an orientable, connected 2-manifold  
 \[ \Sigma: \hat{B} \to \mathcal{L}_c Q \]  
 It follows that the map \[\mathbf{ev} \circ \Gamma \circ (\operatorname{id}
 \times \Sigma): S^1 \times \hat{B} \to Q\] has non-zero degree, in particular
 that $Q$ is dominated by this manifold.  
 \end{proof}

It follows from Theorem \ref{thm: productdomtop} that we have the following result:

\begin{thm} 
    \label{thm: productdom}  Let $Q$ be a closed, oriented 3-manifold which admits a psuedo-dilation over $\mathbb{Q}$. Then, either:
\begin{itemize}
\item  $Q$ is finitely covered by $S^1 \times B$ for $B$ closed and oriented of genus $\geq 1$ (and in particular aspherical), or
    
\item $Q$ admits a metric of positive scalar curvature. Equivalently, $Q$ is finitely covered by a connected sum $\#_n S^1 \times S^2$ (by convention, the case $n=0$ corresponds to $S^3$). 
\end{itemize}
 \end{thm}  

One can improve this result when the pseudo-dilation $(\Psi, \alpha)$ has an
integral lift:

\begin{lem} 
    \label{lem: charS3} 
    Suppose $Q$ is a closed, oriented 3-manifold. \begin{enumerate} \item If $Q$ admits a metric of positive scalar curvature and a psuedo-dilation over $\mathbb{Z}$ then $Q \cong \#_n S^1 \times S^2$, $n\geq 0$. 
\item If $Q$ is aspherical and admits a quasi-dilation over $\mathbb{Z}$, then $Q \cong S^1 \times B$ where B is a Riemann surface of genus $\geq 1.$
  \end{enumerate}
\end{lem} 
\begin{proof}

To prove (1) observe that if $Q$ admits a pseudo-dilation over $\mathbb{Z}$, the fundamental class can be written as the sum of pushforwards of the fundamental class along $f_j: S^1 \times B_j \to Q$ (the $j$ range over different components of $\mathcal{L}Q$). In particular, we must have that the integers $\operatorname{deg}(f_j)$ are relatively prime. On the other hand, as was shown in \cite{Gromov, MR895623 , Perelman1, Perelman2} and reviewed in
(\cite{MR3054316}*{\S 2}), $Q$ admits a prime decomposition
with primes $S^1\times S^2$ or spherical space forms $Q_i$. For every
spherical space form which appears, we obtain a degree 1 map $Q \to Q_i$. This
implies by Lemma \ref{lem:covertricks} that for each map $f_j: S^1 \times B_j \to Q$, and every $p$ which
divides $|\pi_1(Q_i)|$, $p$ divides $\operatorname{deg}(f_j)$. Therefore, every
$Q_i$ must be $S^3$ and the Lemma is proved.

To prove (2) note that $Q$ is aspherical and hence the fundamental group of $Q$ is torsion free.\footnote{To see this, suppose there is finite cyclic subgroup $C \subset \pi_1(Q)$ which induces a covering space $BC \to Q$. As $BC$ has infinite homological dimension, this is impossible.} Corollary 2.3 of \cite{MR1726305} implies that all central units of the integral group ring $\mathbb{Z}[\pi_1(Q)]$ are of the form $\pm g$ for $g$ in the center of $\pi_1(Q).$ In particular $\alpha= \pm g \in \mathcal{Z}(\mathbb{Z}[\pi_1(Q)]) \cong \mathbb{H}_0(\mathcal{L}Q)$ and it therefore follows that all of the $B_j$ from the proof of (1) above map to a single component of the loop space and that $Q$ admits a degree 1 map from $Y=S^1 \times \hat{B}.$ It follows from Lemma \ref{thm:Rong} that $Q \cong S^1 \times B$ for some $B$ of genus $\geq 1$.
 \end{proof}

\begin{rem}
     Lemma \ref{lem: charS3} is sharp in the sense that examples of exact
        Lagrangian embeddings in Appendix \ref{section:AppendixA} show that there exist quasi-dilations in
        the loop homology algebra over $\mathbb{Z}$ of $\#_n S^1 \times S^2$
        for all $n$ (the K{\" u}nneth formula produces an
        integral quasi-dilation when $Q \cong S^1 \times B$).  
\end{rem} 

Finally, we examine the implications of the existence of \emph{dilations} over $\K=\mathbb{Z}$ or a field of characteristic $p>0$ leading to a curious 
characterization of the 3-sphere in terms of dilations:

\begin{lem} \label{lem:dilatp} Suppose that $Q$ is a closed, oriented 3-manifold. \begin{enumerate} \item For any $Q$  such that $\mathbb{H}^*(\mathcal{L} Q,\mathbb{Z})$ admits a dilation, then $Q \cong S^3$.  \item Suppose that  $\mathbb{H}^*(\mathcal{L} Q,\K)$ admits a dilation where $\K$ is a field of characteristic $p>0$. Then $Q$ admits a prime decomposition with primes $S^1\times S^2$ or spherical space forms $Q_i$ such that $p$ does not divide $\pi_1(Q_i)$.   \end{enumerate} \end{lem}   

\begin{proof} For (1), by Lemma \ref{lem: charS3} we have $Q$ is an aspherical $S^1 \times B$ or $Q=Q_n :=\#_n S^1 \times S^2$
for $n \geq 1$.  It is impossible that $Q$ is an aspherical $S^1 \times B$ by Lemma \ref{lem: nodiK} and so we only need to treat the second case.
 We have a degree 1 map $f: Q_n \to
S^1 \times S^2$, which induces a map $f_L: \mathcal{L}Q_n \to \mathcal{L}(S^1\times S^2)$ on loop spaces. The map $f_L$ has the following two properties \begin{equation} \label{eq:deltaprop}
f_{L,*}([Q_n])= [S^1 \times S^2]
   \quad\text{and}\quad 
\Delta(f_{L,*}(\Psi))= f_{L,*}(\Delta(\Psi))
\end{equation}
where $[Q_n]$ and $[S^1 \times S^2]$ denote fundamental cycles of constant loops and $\Psi$ is an arbitrary class in $\mathbb{H}^*(\mathcal{L}Q_n)$.  If $Q_n$ had an integral dilation $(\Psi,[Q_n])$, then the two equations from \eqref{eq:deltaprop} show that \begin{align} \Delta(f_{L,*}(\Psi))= f_{L,*}(\Delta(\Psi))= [S^1 \times S^2] \end{align}
and hence that $S^1 \times S^2$ would have an integral dilation, which is known not to be true from the calculation of \cite{Menichi} mentioned above together with an elementary argument using the K{\" u}nneth formula. 

To check (2), note that for any prime summand we get a degree 1 map $f: Q \to Q_i$ again inducing a map $f_L$ on loop spaces satisfying the properties \eqref{eq:deltaprop}. As in the proof of (1), this means that if $Q$ had a dilation $(\Psi,[Q])$ over $\K$, then \begin{align} \label{eq:dilQi} \Delta(f_{L,*}(\Psi))= f_{L,*}(\Delta(\Psi))= [Q_i] \end{align} Equation \eqref{eq:dilQi} together with Lemma \ref{lem: nodiK} imply that none of the $Q_i$ can be aspherical. 

So, again in view of Equation \eqref{eq:dilQi}, the claim (2) reduces to showing that for any $Q=S^3/\pi$, $Q$ does not admit a dilation over a field of characteristic $p$ which divides $|\pi|$. As the classifying space $B\pi$ can be obtained from $S^3/\pi$ by attaching the higher cells, the classifying map induces a surjection: 
\begin{align} \label{eq:piattaching}p_*: H_3(S^3/\pi) \to H_3(B\pi). 
\end{align}  
which by Equation \eqref{eq:classifyingcalc} is an isomorphism in characteristic $p$ (it is here that we use that $p$ divides $|\pi|$). Suppose we have $\Psi \in H_2(\mathcal{L}Q,\mathbf{k})$ for $\mathbf{k}$ a field of characteristic $p$ with $\Delta(\Psi)=[Q]$. Then by transfering, we obtain a class $p_*(\Psi)  \in H_2(\mathcal{L}_{[id]}B\pi,\mathbf{k})$ with $\Delta(p_*(\Psi)) \neq 0$. This is a contradiction, since for any aspherical space, the space of contractible loops retracts onto the constant loops and hence $\Delta=0.$
\end{proof}

These classification results have the following application to Lagrangian
embeddings: 
\begin{cor} 
    \label{cor: Vitdil} Suppose that $X$ is a 6-dimensional Liouville
    domain such that $SH^*(X,\mathbb{Q})$ admits a quasi-dilation. Suppose that
    $Q \hookrightarrow X$ is an exact Lagrangian embedding. Then $Q$ is one of
    the two types of manifolds listed in Theorem \ref{thm: productdom}. If the
    quasi-dilation lifts to $SH^*(X,\mathbb{Z})$, then $Q \cong \#_n S^1 \times S^2$ for some 
    $n \geq 0$ or $Q \cong S^1 \times B$ with $g(B) \geq 1.$ \end{cor} 
\begin{proof}   Assume $SH^*(X,\mathbb{Q})$ admits a quasi-dilation $(\Psi,\alpha)$ and consider the Viterbo restriction (recall \eqref{eq:Viterbo}) of 
these elements $(j^!(\Psi),j^!(\alpha))$ to a Weinstein neighborhood $D^*Q$. Because we make no assumption on the Maslov class of our Lagrangians (see \cite{Ritter} for other applications of Viterbo functoriality in the absence of Maslov class assumptions), 
$\mathcal{AS} \circ j^!$ no longer necessarily preserves $\mathbb{Z}$-gradings (only $\mathbb{Z}/2\mathbb{Z}$-gradings) and so $\mathcal{AS} \circ j^!(\alpha)$ and $\mathcal{AS} \circ j^!(\Psi)$ may not be concentrated in a single degree.  
    Nevertheless, because $\mathcal{AS}$ and $j^!$ are both BV-algebra maps and $\operatorname{ev}_*$ is an algebra map, $\operatorname{ev}_*(\mathcal{AS} \circ j^!(\alpha))$ must still be a unit in $\mathbb{H}^*(Q, \K)$. It follows that the degree 0 piece of $\operatorname{ev}_*(\mathcal{AS} \circ j^!(\alpha))$ is a unit in $\mathbb{H}^0(Q, \K)$ and consequently that $Q$ admits a rational pseudo-dilation by taking the
    degree 1 component of $\mathcal{AS} \circ j^!(\Psi)$ and the degree 0 component of
    $\mathcal{AS} \circ j^!(\alpha)$. The result then
    follows from Theorem \ref{thm: productdom}.

 When $Q$ admits a metric of positive scalar curvature, the second statement is
    proven in exactly the same way after invoking Part (1) of Lemma \ref{lem: charS3}. When $Q$ is aspherical note that  we have that $\mathbb{H}^*(\mathcal{L}Q)$ is concentrated in non-negative degree which implies that the degree zero piece of $\mathcal{AS} \circ j^!(\alpha)$ is also invertible. We therefore conclude using Part (2) of Lemma \ref{lem: charS3}.
\end{proof}

    \begin{rem}   Corollary \ref{cor: Vitdil} has implications in symplectic topology outside of the main cases we consider in this paper. For example, let $X$ be a 3-dimensional $A_n$ Milnor fibre which we recall admits a Lefschetz fibration with general fiber $T^*S^2.$ Using Lefschetz fibrations techniques (as in  \cite{Seidel:2010uq}*{\S 7}), $SH^*(X,\mathbb{Z})$ can be seen to admit a quasi-dilation which becomes a dilation after tensoring with any field $\K$ with characteristic not equal to 2(because as noted at the end of \S \ref{subsection:stringtopback}, $\mathbb{H}^*(\mathcal{L}S^2,\mathbb{Z})$ has a quasi-dilation with this property). This implies that if $Q$ is a rational homology sphere which admits an exact Lagrangian embedding, $j: Q \hookrightarrow X$, then $Q \cong S^3$. \end{rem}

\begin{rem}\label{fukayacontrast} 
    In pioneering work, Fukaya \cite{FukayaLoopSpace1, FukayaLoopSpace2} produces
      for (any) Lagrangian embedding $Q \subset
    \C^3$, a Maurer-Cartan element of the (chain-level $L_{\infty}$ structure on the) 
    (equivariant) chains on the free loop space $C_{n-*}(\mc{L} Q)$ whose
    associated deformation is trivial. In the spirit of the present section, it may be interesting to study the following 
homological (rather than chain-level) version of Fukaya's condition: Let  $\alpha \in
 \mathbb{H}^0(\mathcal{L}Q,\mathbf{k})$ and $\Psi \in
 \mathbb{H}^1(\mathcal{L}Q,\mathbf{k})$ be two loop space classes. Then $(\Psi,
 \alpha)$ define a \emph{coordinate function} if $$ [\alpha,\Psi]=1. $$ We conjecture that if $Q$ is a closed, 
orientable  3-manifold with a coordinate function $(\Psi,\alpha)$, then $Q$ 
 is diffeomorphic to $S^1 \times B$, where $B$ is an orientable Riemann surface. While discussing Fukaya's work, it also seems 
important to note that in the present paper, we are able to obtain much sharper results when our quasi-dilations are defined integrally (see Lemma \ref{lem: charS3}) --- something which is not possible (or relevant) in Fukaya's setup. 
\end{rem}

\def\lra{\longrightarrow}

\section{Dilations and Applications} 
\subsection{Dilations}

Let $M$ be a Fano variety of dimension $n \geq 3$ and $D$ a very ample smooth
divisor such that $H^2(M,\mathbf{k})= \K \langle \PD(D)\rangle$, where $\PD$ denotes the Poincar\'e dual, and
\[K_M^{-1}=\mathcal{O}(mD)\] for $m>n/2$.  Let $j_D: D \to M$ denote the
inclusion.

\begin{lem} \label{lem: H2zero} As usual, let $X= M\setminus D.$ Then, under the preceding assumptions, $$H^2(X,\K)=0.$$ \end{lem} 
\begin{proof} Consider the long exact sequence (all homologies are taken with $\K$-coefficients): $$...\to H_{2n-2}(D) \cong H_{2n-2}(M) \to H_{2n-2}(M,D) \to H_{2n-3}(D) \to... $$   By the adjunction formula, $D$ is a Fano manifold. Hence, we have that $D$ is simply connected (see e.g. Section 5.1. of \cite{MR2011745}) and in particular $H_{2n-3}(D,\K) \cong H^1(D,\K)=0.$ We therefore have that $H_{2n-2}(M,D)=0.$ Using excision and Poincar\'e duality, this implies that $H^2(X,\K)=0$ as claimed. \end{proof}

 For any three cohomology classes $\beta_i \in H^*(M)$, let
$GW_M(\beta_1, \beta_2)$ and $GW_M(\beta_1,\beta_2,\beta_3)$ denote the
2-point and 3-point Gromov-Witten invariants on $M$ \cite{McDuff:2004aa}. Throughout this subsection,
since we are in the case of a smooth divisor, we will use the more suggestive notation $S_D$ for $S_1$, the unit normal bundle around $D$. Similarly, for a class/chain in log cohomology, we will use the notation $\alpha t$ for the class $\alpha t^{\vec{\v}_1}$ with primitive multiplicity vector on the single smooth component $D_1 = D$.
 
\begin{thm} \label{thm: dilationcrit} Let $\alpha_0$ be a class in $H^{2m-4}(M)$ and set $\alpha_i=\PD(D)^i\cup \alpha_0$ for $i \in \lbrace 1,2 \rbrace$. Suppose that  
\begin{align} \label{eq: dilation1}
j_D^*(\alpha_2)=0 &\in H^*(D)\\
\label{eq: dilation2} GW_{M}(\PD(D),\alpha_1,pt) &\in \mathbf{k}^{\times}.
\end{align}
Then $SH^*(X,\mathbf{k})$ admits a dilation.
\end{thm} 
The proof of Theorem \ref{thm: dilationcrit} requires two lemmas. To state
the first, we introduce some notation. Choose two generic smooth very ample divisors $D_N,D_N' \subset D$
representing the normal bundle $ND$ (these exist by Bertini's theorem). Recall that given such a $D_N$, one may
form the real oriented blow up $Bl_{D_N}(D)$ of $D$ along $D_N$, which is a submanifold
with boundary inside of $S_D$ (see e.g \cite{Sabbah}*{\S 8.2}). Locally, given a
holomorphic function $f:U \to \mathbb{C}$, the real  oriented blow-up along $f^{-1}(0)$
is constructed as closure of the graph of \[\frac{f}{||f||}: U \setminus
\lbrace f^{-1}(0) \rbrace \to \mathbb{S}^1.\] inside of $U \times S^1.$ For a general smooth divisor, it
is constructed by patching together these local models  and has boundary
$\pi^{-1}(D_N) \subset S_D$. 

Pick a generic pseudocycle $P$ representing
$j_D^*(\alpha_0)$ in $H^*(D)$ (the condition \eqref{eq: dilation2} implies that $j_D^*(\alpha_0)$
is a non-trivial class by the adjunction formula).  Recall that by definition of a pseudocycle, the limiting set of $P$ is covered by a smooth manifold $W_P$ of dimension at most $\operatorname{dim}(P)-2$. By perturbing $P$ (or $D_N$), we can assume that $D_N$ intersects $P$ and $W_P$ transversely. Under this transverality assumption, $L:= P \times_D Bl_{D_N}(D) \to S_D$ is a pseudocycle (with boundary) whose
boundary $\partial L$ is the pullback of $S_D$ along  $k: P \times_{D} D_N \to D$. 

 We next set $K_0:= k \times_{D} Bl_{D_N'}(D)$, which after a possible further perturbation of either $P$ or $D_N'$, will be a pseudocycle with boundary the pull-back of $S_D$ along $$P \times_{D} (D_N \cap D_N') \to D. $$ Under assumption \eqref{eq: dilation1}, $P \times_{D} (D_N \cap D_N') \to D$ represents a trivial (pseudo) homology class i.e. there exists a pseudocycle $K_1$ which bounds it. We may therefore may glue $K_0$ with $\pi^{-1}(K_1)$ (and smoothly approximate) to form a pseudocycle $K$ mapping to $S_D$. We have that $H^2(X)=0$ by Lemma \ref{lem: H2zero} and so
\begin{align} \label{eq:rk1coh}
    GW_{\vec{\v}_1}([K])=0 \in H^2(X). 
\end{align} 
where $GW_{\vec{\v}_1}([K])=0$ is the obstruction class defined as in Equation \ref{eq: GWinvariantsnormal}. 
We can therefore choose a Hamiltonian $H_M^{\lambda}$ with $\lambda> m+ \hbar$ and pick a bounding relative pseudocycle $Z_b$ so $\PSSlog^{\lambda}(Kt, Z_b)$ defines a class in
$SH^*(X, \mathbf{k})$.\footnote{This class \emph{a priori} depends on $\lambda$ since we did not address invariance issues when $Z_b \neq 0.$} As
before consider the fiber product 
\begin{align} ev_\infty:
    \mathcal{M}_{0,2}(M,D,\vec{\v}_1)^o \times_{\operatorname{Ev}_0}
    L \to X. 
\end{align} 
This defines a pseudocycle
$\underline{GW_{\vec{\v}_1}(L)}$ when restricted to $\bar{X}$. To see this, notice that $\partial L$ is a union of circle fibers and so there is an rotational $S^1$-symmetry on $\mathcal{M}_{0,2}(M,D,\vec{\v}_1)^o \times_{\operatorname{Ev}_0}
    \partial L.$ Hence, as in the proof of Lemma
\ref{lem:degeneracy},  the map
from the fiber-product 
\begin{align} ev_{\infty}:
    \mathcal{M}_{0,2}(M,D,\vec{\v}_1)^o \times_{\operatorname{Ev}_0}
    \partial L \to X 
\end{align} 
factors through the quotient $(\mathcal{M}_{0,2}(M,D,\vec{\v}_1)^o \times_{\operatorname{Ev}_0}
    \partial L)/S^1$. We may now state our first lemma: 

\begin{lem}\label{lem:pssBV6}
There is a (cohomological) equality 
    \begin{equation}
    \Delta (\operatorname{PSS}_{log}^{\lambda}(Kt,Z_b))= - \PSS_{\vec{\v}_{\emptyset}}(\underline{GW_{\vec{\v}_1}(L)}) 
    \end{equation}
\end{lem}
\begin{proof} 
    The argument of Corollary \ref{corpssBV} and Lemmas \ref{lem: BVandPSS}, \ref{lem:S1eq} (c.f. Remark \ref{pssBVpseudo})  shows that there is a cohomological equality \begin{align} \Delta (\PSSlog^{\lambda}(Kt,Z_b)) = \PSSlog^{\lambda}(\Gamma_Kt). \end{align}
where $\Gamma_K: S^1 \times K \to SD$ denotes the spinning (with speed 1) from \eqref{eq: actingpush}. It is useful at this stage to recall that $K$ was defined by gluing $K_0$ and $K_1$ above. Next, we make two observations: the first is that the spinning $\Gamma_{K_{1}}: S^1 \times K_1 \to SD$ factors through $K_1$. Thus, for any $x_0$ such that $|x_0|=|\Gamma_K t|$, we can choose complex structures so that for any $u \in \mathcal{M}(\vec{\v}_1,x_0)$, $\operatorname{Ev}_{z_{0}}(u)$ is disjoint from $\Gamma_{K_{1}}.$  The second observation is that the pseudocycles $\Gamma_{K_{0}}$ and $\partial L$ are identified away from $S^1 \times \partial K_0$ viewed as a submanifold of (the domain of) $\Gamma_{K_{0}}$ and the pull-back of $SD$ to $P \times_{D} (D_N \cap D_N')$ viewed as a submanifold of (the domain of) $\partial L$ (the images of these two pseudocycles coincide exactly). 

Combining these two observations shows that \begin{align}\PSSlog^{\lambda}(\Gamma_Kt)= \PSSlog^{\lambda}(\partial{L}t) \end{align} We therefore have that: 
 \begin{align} \label{eq:spinningwheels} \Delta (\PSSlog^{\lambda}(Kt,Z_b))= \PSSlog^{\lambda}(\partial{L}t).
 \end{align}  Meanwhile we have that \begin{align} \partial_{CF} \circ \PSSlog^{\lambda} (Lt)=\PSSlog^{\lambda}(\partial{L}t) + \PSS_{\vec{\v}_{\emptyset}}(\underline{GW_{\vec{\v}_1}(L)}) \end{align} This means that on the level of cohomology, $\PSSlog^{\lambda}(\partial{L}t)= -\PSS_{\vec{\v}_{\emptyset}}(\underline{GW_{\vec{\v}_1}(L)})$. Combining this with \eqref{eq:spinningwheels} we see that as desired, on the level of cohomology \begin{align}  \Delta ( \operatorname{PSS}_{log}^{\lambda} (Kt,Z_b))= -  \PSS_{\vec{\v}_{\emptyset}}(\underline{GW_{\vec{\v}_1}(L)}). \end{align}
\end{proof}

\begin{lem} \label{lem:gweq6}
    We have an equality 
    \begin{equation}
        GW_{\vec{\v}_1}(L) \cdot [pt]= GW(\PD(D),\alpha_1,pt),\end{equation}
    where $\cdot$ on the the left-hand side denotes the intersection product (recall that
    $GW_{\vec{\v}_1}(L) \in H^{BM}_{2n}(X) \cong H^0(X)$).
\end{lem} 
\begin{proof}  
    Fix a generic point $pt$ in the interior of $\bar{X}$. Let $S^1$ act on  $\mathcal{M}_{0,2}(M,D,\vec{\v}_1)$ by rotation in the domain. Observe that for generic $J \in \mathcal{J}(M,\D)$, we have an orientation preserving bijection of moduli spaces  \begin{align} \underline{GW_{\vec{\v}_1}(L)} \times_{ev_{\infty}} pt=  \mathcal{M}_{0,2}(M,D,\vec{\v}_1)/S^1\times_{ev_0} P \times_{ev_{\infty}} pt \end{align} Moreover all such curves passing through $pt$ are $D_u$ regular for generic $J$ even before taking fiber products with $L$ and $P$ respectively. It follows that $GW_{\vec{\v}_1}(L) \cdot pt = GW_M(\alpha_1, pt)$. To conclude, observe that the divisor equation implies that $GW_M(\alpha_1, pt)=GW(\PD(D),\alpha_1,pt)$. 
\end{proof}

\begin{proof}[Proof of Theorem \ref{thm: dilationcrit}]
    Set $\beta = \PSSlog^{\lambda}(Kt,Z_b)$.
    Lemmas \ref{lem:pssBV6} and \ref{lem:gweq6} imply that $\Delta \beta =  -GW(\PD(D),\alpha_1,pt) \cdot
    \PSS_{\vec{\v}_{\emptyset}} (1) = -GW(\PD(D),\alpha_1,pt) \cdot 1 $, 
    a non-zero multiple of the unit by hypothesis. Now normalize.
\end{proof}

\begin{rem} \label{rem:Seidelconj} Theorem \ref{thm: dilationcrit} is closely related to \cite{catdyn}*{Conjecture 18.6}.  In the situation of Theorem \ref{thm: dilationcrit}, Seidel proposes an explicit formula for the Hamiltonian Floer cohomology (with its BV operator) of a Hamiltonian of  a specific slope $\lambda$ and suggests applying this to produce dilations (see Examples 18.8 and 18.9 of \emph{loc. cit}). Theorem \ref{thm: dilationcrit} uses the Log PSS map to bypass the computation of Hamiltonian Floer cohomology. \end{rem}

\begin{rem} \label{rem:applicationsofdilationcrit} It would be interesting to apply Theorem \ref{thm: dilationcrit} to cases where the Gromov-Witten invariants are well-understood (such as homogeneous varieties). A sample case where this should be possible is where $M$ is a generic hyperplane generating the divisor class group in $Gr(2,2n+2)$ and $D$ is a generic hyperplane section. \end{rem} 

\subsection{Quasi-dilations}

For $J \in \mathcal{J}(M,\D)$ and $A \neq 0 \in H_2(M,\mathbb{Z})$, let
$\mathcal{M}_{0,0}(M,A,J)$ denote the moduli space of $J$-holomorphic spheres
$$ u: \mathbb{C}P^1 \to M$$ such that $u_*([\mathbb{C}P^1])=A$. Let
$\widehat{\mathcal{M}}_{0,0}(M,A,J)$ denote the moduli space 
\begin{equation} \label{eq:hatM}
   \widehat{\mathcal{M}}_{0,0}(M,A,J):=\bigsqcup_{A_i \neq 0,\sum A_i=A} \prod \mathcal{M}_{0,0}(M,A_i,J). 
 \end{equation} 

 For the remainder of this section, assume $(M,\D)$ is a pair and $J:= \lbrace 2, \cdots, |J|+1 \rbrace$ is a subset of $\lbrace 1, \cdots, k \rbrace$ such that $D_J=\cap_{j \in J} D_j$ is connected and non-empty. We assume that there exists a volume form $\Omega$ on $M$ whose divisor of zeroes is as in \eqref{eq:volform} with $a_i=1$ for $i \in \lbrace 1 \cup J \rbrace$. We will also assume that $D_1$ is a divisor which satisfies $D_1 \cap D_J = \emptyset$ as well as the following conditions (B1)-(B4) below:
\begin{itemize}
    \item[(B1)] The normal bundle to $D_1$ is trivial when restricted to $D_1\setminus \cup_{i\neq 1}D_i $
\end{itemize}

Because the torus bundle $\mathring{S}_1$ is trivial, the cohomology splits as
the cohomology $H^*(\mathring{S}_1) \cong H^*(S^1)\otimes H^*(\mathring{D}_1)$.
In particular, if we fix an isomorphism $ k[\epsilon]/\epsilon^2 \cong
H^*(S^1)$, there is a corresponding cohomology class $\beta_1:=\epsilon\otimes 1 \in
H^*(\mathring{S}_1)$. This cohomology class in turn gives rise to a  generator
$\beta_{1}t^{\vec{\v}_1} \in \QH^{1}(M,\D)$. We fix a section $\beta_{1,c}$ inducing the trivialization of $\mathring{S}_1$ (which gives a cycle representing $\beta_1$). Let $\alpha_{1}t^{\vec{\v}_1}$
denote the generator in $\QH^0(M,\D)$ corresponding to the fundamental class on
$\mathring{S}_1$ and let $\alpha_{1,c}$ be a representative fundamental cycle constructed via the Eilenberg-Zilber product of fundamental cycles (and the isomorphism $\mathring{S}_1 \cong S^1 \times \mathring{D}_1$ induced by $\beta_{1,c}$). Let $\alpha_Jt^{\vec{\v}_J}$ denote the generator in
$\QH^0(M,\D)$ corresponding to the fundamental class in $\mathring{S}_J$ and choose a fundamental cycle representative $\alpha_{J,c}$.  We require:

\begin{itemize}  
    \item[(B2)] The class $\beta_{1}t^{\vec{\v}_1}$  is admissible with vanishing obstruction class.
    \item[(B3)]  There exists a $J_0 \in \mathcal{J}(V)$ so that  \begin{align} \label{eq: compactness} \widehat{\mathcal{M}}_{0,0}(M,B_1,J_0)=\mathcal{M}_{0,0}(M,B_1,J_0)  \end{align} for any class $B_1 \in H_2(M)_{\omega}$ which satisfies $B_1 \cdot \D =\vec{\v}_J$.
\item[(B4)] \begin{enumerate} [label=(\alph*)] \item Let $B$ be a spherical class with $B\cdot \mathbf{D}=\vec{\v}_{1}+\vec{\v}_{J}$. Given a decomposition of $B=B'+B''$ with $B' \cdot D_i \geq 0$ for all $i$ and such that $\widehat{\mathcal{M}}_{0,0}(M,B',J_0)$ and $\widehat{\mathcal{M}}_{0,0}(M,B'',J_0)$ are non-empty, we have that $B' \cdot \mathbf{D}=\vec{\v}_1$ or $B' \cdot \mathbf{D}=\vec{\v}_{J}.$ \item Set $\underline{\vec{\v}}= \vec{\v}_{1}+\vec{\v}_{J}$. Then for any $B$ with $\widehat{\mathcal{M}}_{0,0}(M,B,J_0) \neq \emptyset$ either \begin{align} \label{eq: constrainingspheresIII} B \cdot \D \in \lbrace \underline{\vec{\v}},\vec{\v}_{1},\vec{\v}_{J}\rbrace \end{align} or there exists $i \in \lbrace 1,\cdots, k \rbrace$ with $B \cdot D_i > \underline{v}_i.$  \end{enumerate}
\end{itemize}

Because $\vec{\v}_1$ is admissible, Lemma \ref{lem: randomcompacti} shows that for generic $J$, the moduli space $\mathcal{M}_{0,2}(M,\mathbf{D},\vec{\v}_1)^o$ is smooth and maps properly to $X$. Note that by Gromov compactness, for an almost complex structure $J$ sufficiently close to $J_0$, conditions (B3) and (B4) continue to hold with $J_0$ replaced by $J$.  In what follows, we will assume that all complex structures are taken sufficiently close to $J_0$ so that these conditions still hold. In particular, we have the following lemma: 

\begin{lem} \label{lem:smoothmapproperly}
For generic $J$ sufficiently close to $J_0$, the moduli spaces $\mathcal{M}_{0,2}(M,\mathbf{D},\vec{\v}_J)^o$ are smooth and map properly to $X$ as well. \end{lem} 
\begin{proof}   Consider a sequence of curves with $ev_\infty \subset K$ for some compact
    set $K \subset X$. As stated above, for $J$ sufficiently close to $J_0$, condition (B3) continues to hold and hence no bubbling can occur in this sequence. The rest of the argument proceeds exactly as in the proof of Lemma \ref{lem: randomcompacti}. \end{proof} 

In the proof of Lemma \ref{lem: multiply} (which concerns product structures), we will also need to consider 3-pointed versions of these moduli spaces:

\begin{defn} 
    Let $B \in H_2(M)_{\omega}$ be a spherical class such that $B\cdot
    \mathbf{D}=\vec{\v}_{1}+\vec{\v}_{J}$ and let
    $\mathcal{M}_{0,3}(M,\mathbf{D},\vec{\v}_{1},\vec{\v}_{J},B)$ denote the
    space of $D_{1}$ and $D_{J}$-regular maps (Definition \ref{def:regularmap}) $ u: (C,z_0,z_1,z_{\infty}) \to (M,\D)$
    such that 
    \begin{align*}
        u_{\ast}([\mathbb{C}P^1])&= B\\  
        u^{-1}(D_1)&=  z_0  \\
        u^{-1}(D_j)&=  z_1 \textrm{ {\rm for} } j \in J.
    \end{align*}
We let  \begin{align} \mathcal{M}_{0,3}(M,\mathbf{D},\vec{\v}_{1},\vec{\v}_{J}):=\bigsqcup_B\mathcal{M}_{0,3}(M,\mathbf{D},\vec{\v}_{1},\vec{\v}_{J},B)  \end{align} 
\end{defn}

As the classes are primitive, all such curves $u$ are automatically somewhere
injective and thus we may achieve transversality for such maps. Consider
$ev_{\infty}^{-1}(X)=\mathcal{M}_{0,3}(M,\mathbf{D},\vec{\v}_{1},\vec{\v}_{J})^{o}$.
We have a partial compactification of this moduli space
$\overline{\mathcal{M}}_{0,3}(M,\mathbf{D},\vec{\v}_{1},\vec{\v}_{J})^{o}$
given by incorporating nodal curves with the following properties:
\begin{itemize} 
\item $ev_{\infty} \in X.$
\item there are exactly two non-constant components $u_0$ and $u_1$ in classes $B_0$ and
$B_1$ with $B_0 \cdot \D=\vec{\v}_1$ and $B_1\cdot \D=\vec{\v}_J$. Furthermore, each of these components is $\D$-regular and $z_0$ lies in $u_0$ while $z_1$ lies in $u_1$.
\item There are either no constant components, or one constant component containing the marked point $z_\infty.$
\end{itemize} 

\begin{lem} \label{lem:evproper} For $J$ sufficiently near $J_0$, the evaluation map
$$ev_{\infty}:
\overline{\mathcal{M}}_{0,3}(M,\mathbf{D},\vec{\v}_{1},\vec{\v}_{J})^{o} \to X$$
is proper. \end{lem} 
\begin{proof}  Consider a sequence of curves in $\mathcal{M}_{0,3}(M,\mathbf{D},\vec{\v}_{1},\vec{\v}_{J})^{o}$ such that $ev_{\infty}$ lies in some fixed compact set $K.$ Then by (B4) (a), there are two possible cases for bubbling: \vskip 5 pt \emph{Case I}: There is only one non-constant component $u$ with some constant bubbles attached. Then two of the marked points would have to lie on the same constant component. But this is impossible because $z_0$, $z_1$, $z_\infty$ all map to pairwise disjoint subsets of $M$(recall that by assumption $D_1 \cap D_J=\emptyset $).
\vskip 5 pt \emph{Case II}: There are two non-constant components $u_0$ and $u_1$ with homology classes $B_0 \cdot \D=\vec{\v}_1$ and $B_1\cdot \D=\vec{\v}_J$. Because $ev_\infty$ lies in $K$ and hence in $X$, we know that one of the components must be $\D$ regular and from intersection considerations, we see that $u_0$ and $u_1$  can only meet in $X$. It follows that both components are $\D$-regular. As in \emph{Case I}, constant components can only contain (exactly) one of the marked points and hence there is at most one constant component which is glued to $u_0$ at one nodal point and glued to $u_1$ at another nodal point (the constant component has a total of three special points). This potential constant component must therefore map to a point where $u_0$ and $u_1$ intersect, which must lie in $X$. It follows that the third special point (the marked point) on the constant component must be $z_\infty.$ \end{proof}

For generic $J$ near $J_0$, we therefore have that $\overline{\mathcal{M}}_{0,3}(M,\mathbf{D},\vec{\v}_{1},\vec{\v}_{J})^{o}$ defines a relative pseudocycle of codimension zero \begin{equation}
    \underline{GW(\vec{\v}_1,\vec{\v}_J)}
\end{equation}
when intersected with $\bar{X}$. 

\begin{thm} 
    \label{thm: generalquasi} 
    Let $(M,\D)$ be a pair as above which satisfies (B1)-(B4). Choose a Hamiltonian $H_M^{\lambda}$ with $\lambda>\kappa_1+\hbar$ (as in Lemma \ref{lem:compactness}) and choose a bounding cochain $Z_b$ for
    $\beta_1t^{\vec{\v}_1}$. Suppose further that the cohomology class
    $GW(\vec{\v}_1,\vec{\v}_J) \in H^0(X)^{\times}$, and let $\beta_{1,c}$, $\alpha_{1,c}$ be the chain-level representatives of $\beta_1$ and $\alpha_1$ chosen above. Then the pair
    \begin{align} (\mathfrak{c}_{\lambda,\infty} \circ \PSSlog^{\lambda}
        (\beta_{1,c}t^{\vec{\v}_1},Z_b), \PSSlog
        (\alpha_{1,c}t^{\vec{\v}_1}))\end{align} defines a quasi-dilation. 
\end{thm}
\begin{proof}
This follows immediately from the fact that on cohomology $\Delta(\PSSlog^{\lambda}
(\beta_{1,c}t^{\vec{\v}_1},Z_b))=\PSSlog^{\lambda}(\alpha_{1,c}t^{\vec{\v}_1})$ (by Corollary \ref{corpssBV}, noting that $\beta_{1,c}$ and $\alpha_{1,c}$ as chosen tautologically satisfy $[\Delta(\beta_{1,c}t^{\vec{\v}_1})] = [\alpha_{1,c}t^{\vec{\v}_1}]$; see \eqref{logBVoperator} for the definition of $\Delta$ on log cohomology),
together with the identity proven in Lemma \ref{lem: multiply} below (which by the hypothesis $GW(\vec{\v}_1,\vec{\v}_J) \in H^0(X)^{\times}$ implies that $\PSSlog (\alpha_{1,c}t^{\vec{\v}_1})$ is a unit).
\end{proof}
\begin{lem} 
    \label{lem: multiply} Under the assumptions of Theorem \ref{thm: generalquasi}, let $\alpha_{1,c}$ and $\alpha_{J,c}$ denote the chain-level representatives of $\alpha_1$ and $\alpha_J$ chosen above. Then there is a cohomological equality
    \[\PSSlog(\alpha_{1,c}t^{\vec{\v}_1}) \cdot \PSSlog(\alpha_{J,c}t^{\vec{\v}_J})=\PSS(GW(\vec{\v}_1,\vec{\v}_J)).\]
\end{lem} 
\begin{proof} 
We define an auxiliary moduli space which will allow us to prove that the PSS
map defined above preserves the ring structure. We work over a parameter space
$q \in [b,\infty)$ with $b\gg 0$. We consider the surfaces $S_{q,2}$ whose underlying domain is
$\mathbb{C}P^1 \setminus \lbrace 0 \rbrace$ with a negative cylindrical end
as before, but this time with two distinguished marked points at $z_1=q$
and $z_2=-q$. Let $\beta$ be a subclosed 1-form on $S_{q,2}$ satisfying

\begin{itemize} 
    \item The form $\beta$ restricts to $2dt$ on the cylindrical end;
    \item $\beta=0$ in neighborhoods of $z_1=-q$ and $z_2=q$.  
\end{itemize} \vskip 10 pt 

Define \[
\mathcal{M}_q( M, \vec{\mathbf{v}}_1, \vec{\mathbf{v}}_J  ; x_0)\] 
to be the moduli space of pairs $(q,u)$, with 
$q \in  [b,\infty)$ and 
\[ u: S_{q,2} \to M \]
as usual satisfying
\begin{equation}
(du - X_{H_m^\lambda} \otimes \beta)^{0,1} = 0
\end{equation}
with asymptotic condition
\begin{equation}
\lim_{s \ra -\infty} u(\epsilon(s,t)) = x_0
\end{equation}
and tangency/intersection conditions
\begin{align}
u(x) &\notin \mathbf{D} \textrm{ for $x \neq z_i$}; \\
 u(z_1) &\textrm{ intersects $D_i$ with multiplicity $\vec{\v}_1$}.\\
 u(z_2) &\textrm{ intersects $D_i$ with multiplicity $\vec{\v}_J$ }.
\end{align}

When $|x_0|=0$, $\mathcal{M}_q( M, \vec{\mathbf{v}}_1, \vec{\mathbf{v}}_J  ; x_0)$ has dimension 1 for generic choices. These moduli spaces then admit a Gromov compactification, which has two distinct types of boundary strata:  

\emph{(Type I:)} This case corresponds to the usual cylindrical breaking and sphere bubbling at some finite $q \in [b,\infty).$ Cylindrical breaking gives rise to strata: \begin{align} \label{eq:anotherhomotopyterm} \bigsqcup_{|y|=-1} \mathcal{M}_q( M, \vec{\mathbf{v}}_1, \vec{\mathbf{v}}_J  ; y) \times \mathcal{M}(x_0, y) \end{align}

On the other hand, sphere bubbling does not arise at finite $q$. To see why this is the case, suppose that we are given a configuration consisting of a collection sphere bubbles attached to a solution $u_{\operatorname{thimble}}: S_{q,2} \to M$. Note that by (B4)(b) the total homology class of all sphere bubbles must lie in a class $B$ which satisfies \eqref{eq: constrainingspheresIII} and that there can be at most two non-constant bubble components. Suppose there is a sphere bubble $u_{\operatorname{sphere}}$ which attaches at a point in $S_{q,2}$ which is not $z_1$ or $z_2$. We have: \begin{align} \label{eq:randomthing55} u_{\operatorname{thimble}} \cdot \D + u_{\operatorname{sphere}} \cdot \D \leq \vec{\v}_1+\vec{\v}_J \end{align}

Then because there can be at most two non-constant bubble components, there must be at least one marked point ($z_1$ or $z_2$) where no sphere bubble attaches. Assume without loss of generality that this point is $z_1$, in which case we have that  $u_{\operatorname{thimble}}$ intersects $\D$ with multiplicity $\vec{\v}_1$ at $z_1$. By \eqref{eq:randomthing55} (and \eqref{eq: constrainingspheresIII}), it follows that $u_{\operatorname{sphere}}$ is the only sphere bubble, which means there is no sphere bubbling at $z_2$ either. Hence we have \begin{align} u_{\operatorname{thimble}} \cdot \D \geq \vec{\v}_1+\vec{\v}_J \end{align} We then have that $u_{\operatorname{thimble}} \cdot \D + u_{\operatorname{sphere}} \cdot \D > \vec{\v}_1+\vec{\v}_J$, which is a contradiction. It follows that there is no sphere bubbling away from $z_1$ and $z_2$. Moreover, the only sphere bubbles which occur at finite $q$ are $\D$-regular sphere bubbles in classes which intersect $\D$ with multiplicity $\vec{\v}_1$ and $\vec{\v}_J$. However, these bubble configurations are of codimension at least two and hence do not arise in our one dimensional moduli space. \vskip 5 pt

\emph{(Type II:)} The compactification also incorporates limits as the parameter $q \to \infty$. This stratum is given by the fiber product: 
\begin{equation}\label{eq:fiberproductGW}
\mathcal{M}_{0,3}(M,\mathbf{D},\vec{\v}_{1},\vec{\v}_{J})^o \times
_{ev_{\infty}} \mc{M}(\vec{\v}_{\emptyset}, x_0) 
\end{equation}

 Similar arguments to those given in the \emph{(Type I)} case explain why there are no configurations with spherical components mapped entirely into $\D.$ The operation associated to counting rigid configurations of
\eqref{eq:fiberproductGW} for varying $x_0$
is by definition the composition 
\begin{align}
    \label{eq: fiberprod} \operatorname{PSS}(GW(\vec{\v}_1,\vec{\v}_J))
\end{align} 
Next consider the moduli space $\mathcal{M}_b( M, \vec{\mathbf{v}}_1,
\vec{\mathbf{v}}_J ; x_0)$ which is the restriction of the above moduli space
to domains $S_{b, 2}$. We have that the operation defined by $\mathcal{M}_b(M,
\vec{\mathbf{v}}_1, \vec{\mathbf{v}}_J ; x_0)$ is homotopic to \eqref{eq:
fiberprod} (as usual, counting configurations of the form \eqref{eq:anotherhomotopyterm} shows that the difference between these two operations is a Floer coboundary). Let $\Sigma$ denote the pair of pants equipped with its three
standard cylindrical ends as in \eqref{eq:FloerS}. We consider the nodal domain
$S_{n}$ of the form 
\[S \cup_\epsilon \Sigma \cup_\epsilon S \]
where the negative cylindrical ends of $S$ are glued to the positive cylindrical ends of $\Sigma$, a pair of pants.  Maps from $S_{n} \to M$ are given by the fiber product of moduli spaces given by: 
\[ 
\coprod_{x_1,x_2} \mathcal{M}( \vec{\mathbf{v}}_1, \alpha_1, x_1) \times \mathcal{M}(\Sigma, x_0, x_1, x_2 ) \times \mathcal{M}( \vec{\mathbf{v}}_J, \alpha_J, x_2) 
\]

We construct a homotopy between this moduli space and $\mathcal{M}_b( M,
\vec{\mathbf{v}}_1, \vec{\mathbf{v}}_J; x_0)$ in two steps. First, we perform a
finite connect sum along the cylindrical ends. Then, we can further homotopy
the complex structure and Floer datum to the domain $S_{b, 2}$ above. We thus
reach the desired conclusion. \end{proof} 

We now turn to constructing a wide class of examples. Let $X^o$ be a smooth affine
variety with trivial canonical bundle $\Lambda^{\operatorname{top}} T^*X^{o} \cong \mathcal{O}_{X^o}$ and $Z^o \hookrightarrow X^o$ be a principal, smooth
hypersurface. 

\begin{defn} 
\label{defn: goodcomp} 
We say that a pair $(\bar{M},\bar{\D})$ with $\bar{\D}=\sum_i \bar{D}_i$  $i \in \lbrace 3,\cdots, r+2 \rbrace$ (the reason for this choice of numbering will be clear momentarily) and $X^o=\bar{M}\setminus \bar{\D}$ is a {\em good compactification} of $Z^o \hookrightarrow X^o$ if 
\begin{enumerate} 
    \item $\bar{M}$ is equipped with an ample line bundle $\mathcal{O}(\sum_i \kappa_i \bar{D}_i)$ such that each $\kappa_i > 0.$
   \item The closure of $Z^o$ in $\bar{M}$, $Z$, is a smooth hypersurface which intersects each stratum $\bar{\D}_I=\cap_{i \in I}\bar{D}_i$ transversely. 
    \item $\bar{\D}$ supports a canonical divisor i.e. $\Lambda^{\operatorname{top}} \bar{M} \cong \mathcal{O}_{\bar{M}}(\sum_i -a_i D_i)$ for some $a_i \in \mathbb{Z}$. 
 \end{enumerate} 
\end{defn}

\begin{lem} 
    Let $Z^o \hookrightarrow X^o$ be a smooth hypersurface in a smooth affine variety with trivial canonical bundle. Then
    there is always a good compactification $(\bar{M}, \bar{\D})$ of $Z^o \hookrightarrow X^o$. 
\end{lem} 

\begin{proof}  

    Because $X^o$ is a smooth affine variety, we can use Lemma \ref{lem:smoothcomp} to find a simple-normal crossings compactification $(\bar{M}_o,\bar{\D}_o)$ of $X^o$ such that $\bar{\D}_o$ supports an effective ample divisor. Let $Z'$ be the compactification of $Z^o$ in $\bar{M}_o$. Next, we can do an embedded resolution of singularities (Theorem \ref{thm:Hironaka}) to the divisor $\bar{\D}_o \cup Z'$ to obtain a pair $(\bar{M}, \bar{\D})$ so that the proper transform of $Z'$, $Z$, is smooth and such that the
   divisor $\bar{\D} \cup Z$ is simple-normal crossings. The smooth centers of the sequence of blow-ups involved in the resolution will all lie over $\bar{\D}_o$. As in the proof of Lemma \ref{lem:smoothcomp}, we can therefore use Lemma \ref{lem:randomhart} to show that $\bar{\D}$ supports an effective ample divisor $F$. Again, to get one which has positive coefficients on each component $\bar{D}_i$, one can take $\sum_i\bar{D}_i + mF$ which is ample for $m$ sufficiently large by \cite[Exercise 7.5(b) of Chapter 2]{Hartshorne}.

Finally, to address (3), note that by applying Proposition 6.5 of \cite[Chapter 2]{Hartshorne} inductively, one obtains an exact sequence $$ \bigoplus_{i=3,\cdots, r+2} \mathbb{Z} \to \operatorname{Pic}(\bar{M}) \to \operatorname{Pic}(X^o) \to 0 $$  where the first map denotes the inclusion of the divisors $\mathcal{O}(\bar{D}_i)$ into the Picard group of $\bar{M}$ and the second is restricting the line bundle to $X^0$. As the canonical bundle restricted to $X^o$ is trivial, it follows that $\Lambda^{\operatorname{top}} \bar{M} \cong \mathcal{O}_{\bar{M}}(\sum_i -a_i \bar{D}_i)$ for $a_i \in \mathbb{Z}.$
\end{proof}

As $Z^o \hookrightarrow X^o$ is principal, it is the zero set of a regular function $f: X^o \to \mathbb{C}.$ We define the \emph{affine conic bundle} 
to be the smooth affine variety $X$ defined by the equation: \begin{align} 
    \label{eq: conicbundle2} X=\lbrace (u,v, \bar{x}) \in \mathbb{C}^2 \times X^o |\; uv=f(\bar{x}) \rbrace 
\end{align}  This variety depends only on the hypersurface $Z^o$. We we will use a good compactification $(\bar{M},\bar{D})$ to produce a compactification $(\hat{M}, \hat{D})$ of the affine conic bundle \eqref{eq: conicbundle2}. Let
$\mathbb{C}P^1$ be equipped with its standard toric polarization $H_i$,
$i=\lbrace 1,2 \rbrace.$ Consider the blow up of $M:=\bar{M} \times
\mathbb{C}P^1$ at $Z \times H_1$ which we denote by \[\pi: \hat{M} \to M.\] On
$M$, consider divisors $D_i=\bar{M} \times H_i$ for $i=\lbrace 1,2 \rbrace$ and $D_i=\bar{D}_i \times
\mathbb{C}P^1$ for $i= \lbrace 3, \cdots, r+2 \rbrace$ and let $\D$ denote the normal crossings collection of divisors
given by \[ \D= D_1, D_{2}, D_{3}, \cdots, D_{r+2} \] Let $\hat{D}_i$ denote
the proper transform of all of the divisors $D_i$. Then the divisors $\hat{\D}$
are normal crossings and $\hat{M}$ has a canonical divisor with $a_1=a_2=1$.
We choose a sufficiently small rational number $\epsilon_Z= 1/m$ for $m\gg 0$ and equip the blow
up with the $\mathbb{Q}$-divisor given by \begin{align} \label{eq:conicpolar} \epsilon_Z \hat{D}_{1} +(1-\epsilon_Z) \hat{D}_{2} + \sum_{i=3}^{r+2}a_i \hat{D}_i  \end{align}
Note that the exceptional divisor $E$ of the blow up is linearly equivalent to $\hat{D}_2-\hat{D}_1$ and hence \eqref{eq:conicpolar} is linearly equivalent to $-\epsilon_Z E + \hat{D}_{2} + \sum_{i=3}^{r+2}a_i \hat{D}_i.$ Therefore \eqref{eq:conicpolar} defines a rational ample line bundle for sufficiently small $\epsilon_Z$ again using Lemma \ref{lem:randomhart}. We say that $(\hat{M}, \hat{D})$ is the {\em conic bundle} associated to the good compactification $(\bar{M}, \bar{D})$. The affine conic bundle from \eqref{eq: conicbundle2} is then the complement $X=\hat{M}\setminus \hat{D}.$ We may perturb the exceptional divisor $E$ so that it is orthogonal to all of the $\D$ and take the complex structure $J_0$ to be any complex structure in $\mathcal{J}(V)$ which preserves $E$. 

\begin{lem}\label{lem:goodcompactificationGW} 
    Let $(\hat{M}, \hat{D})$ be a conic bundle associated to a good compactification $(\bar{M},\bar{D})$. Set $J=\lbrace 2 \rbrace$. Then the pair $(\hat{M}, \hat{D})$ (together with the above choice of complex structure $J_0$) satisfy (B1)-(B4) above.  Moreover $GW(\vec{\v}_1,\vec{\v}_J)=1$. 
\end{lem}
An immediate Corollary of the above Lemma and Theorem \ref{thm: generalquasi} is:
\begin{cor}
    If $X$ is any conic bundle over an affine variety $X^o$ with trivial canonical bundle as in \eqref{eq: conicbundle2}, then $X$ admits an integral quasi-dilation. \qed
\end{cor}

\begin{proof}[Proof of Lemma \ref{lem:goodcompactificationGW}] 
    To prove (B1),  note that
    in the complement of $\hat{D}_2$, the divisor $E+\hat{D}_1$ is a principal divisor. In particular, restricting to $\hat{D}_1$, we have that $N\hat{D}_1 \simeq \mathcal{O}(-E)_{|\hat{D}_{1}}.$  $E \cap \mathring{\hat{D}}_1$ can be identified with a copy of $Z^o$ inside of $X^o$. Because $Z^o$ is principal we have that $\mathcal{O}(-E)_{|\mathring{\hat{D}}_{1}}$ is trivial and (B1) follows. To prove (B2), we check that $\beta_1t^{\vec{\v}_1}$ is admissible with vanishing obstruction class. First note that the class $B_0$ of a $J_0$-holomorphic sphere with $u \cdot \D= \vec{\v}_1$ has minimal symplectic area ($\epsilon_Z=1/m$) and thus cannot bubble. Next, observe that via the argument of Lemma \ref{lem:gweq6} the obstruction class can be calculated using a Gromov-Witten invariant. To calculate this, we may use the standard integrable complex structure where the the moduli
    space of spheres  is empty except in the class $B_0=[E_{\mathbb{C}P^1}]$ of the fibers of the projection $E \to \bar{M}$(exceptional spheres). Moreover, the spheres of class $[E_{\mathbb{C}P^1}]$ are precisely the collection of exceptional
    spheres. These spheres are checked to be regular by applying Lemma 3.3.1 of
    \cite{McDuff:2004aa}. As the restriction of the exceptional divisor $E$ to
    $X$ is principal, it follows
    that $[E]=0 \in H^2(X)$. 

We take $J_0 \in \mathcal{J}(V)$ to be any almost complex structure preserving $E$. To prove (B3) and in particular Equation \eqref{eq: compactness}, let $B_1$ be such a homology class. Note that any sphere meeting $\hat{D}_2$
with multiplicity 1 must intersect $E$ non-negatively. As a consequence, it can
meet $\hat{D}_1$ at most once. It must therefore have non-negative weighted
intersection with the remaining divisors and thus the total area of this sphere
is either $1-\epsilon_Z$ or $1$.  The latter case is impossible as spheres in
class $B_1$ have area $1-\epsilon_Z$. 

To prove (B4)(a), assume $B' \cdot \D \neq \v_1$ or $\v_J$. Observe that it cannot
intersect some $\hat{D}_i$ for $i \geq 3$ with positive multiplicity. Otherwise, its symplectic
area (with respect to the rational polarization) would be at least 1, which
is the same as the area of $B$. It follows that $B' \cdot \D = d\v_1$ with $d>1$. However, then the spheres in class $B''$ must have a component which meets $\hat{D}_2$ with multiplicity one, which by the preceeding paragraph must have area at least $1-\epsilon_Z$. Hence, there can be no such bubbling. Condition (B4)(b) is immediate from simple area considerations--- if the class $B$ had positive area and strictly negative intersection with one of the divisors $\hat{D}_i, i \geq 3$, it would need to have intersection strictly bigger than one with either $\hat{D}_1$ or $\hat{D}_2$ to compensate.  The final
claim about the Gromov-Witten invariant can again be checked using the
integrable complex structure. All spheres in the class $B$ lie in the fiber of
the conic fibration $\pi_{\bar{M}}:  \hat{M} \to \bar{M}$ and there is a unique
such sphere (bubbled into two components in classes $B_0$ and $B_1$ where the
conic fibration is singular) passing through every point.  
\end{proof} 

\begin{rem} 
 One may generalize this example somewhat by considering the blow up of $M:=\bar{M} \times \mathbb{C}P^s$ at $Z \times \cap_{i=2}^{s+1} H_i$. 
\end{rem}

\subsection{Applications to embedding problems} 

We finally turn to the proofs of our last two applications. We first need a
lemma (Lemma \ref{lem: genquasi}) which is a variant of  Theorem 1.7 of \cite{Seidel:2014aa} for
quasi-dilations. We will now summarize some of the ingredients from \cite{Seidel:2010uq}, \cite{Seidel:2014aa}
which will be used in the proof of Lemma \ref{lem: genquasi} and the remaining facts required from these papers are reviewed in Appendix \ref{section:AppendixB}. Let $(M,\D)$ be a pair and $\Psi \in HF^1(X \subset M ; H_{m}^{2\lambda})$ (the motivation for using Hamiltonians of slope $2\lambda$ will be made clear momentarily).  Then Section 4(e) of \cite{Seidel:2014aa} introduced certain $\Psi$-twisted Floer cohomology groups $\tilde{H}^*(X)$ which fit into a long exact sequence \begin{align} \label{eq:leseq} \cdots \tilde{H}^{*-1}(X) \to HF^{*-1}(X \subset M ; H_{m}^{-\lambda}) \xrightarrow{-\Psi \cdot} HF^*(X \subset M ; H_{m}^{\lambda}) \to \tilde{H}^*(X) \to  \cdots  \end{align}

where $\Psi \cdot$ denotes multiplication with $\Psi$.

 Suppose $Q$ is an exact (closed, oriented) Lagrangian submanifold in $X$. We assume that $Q$ is equipped with the additional choices needed to make it an object of the $\mathbb{Z}$-graded Fukaya category of $X$ and let $CF^*(Q,Q)$ denote the Lagrangian Floer cochain complex of $Q$. For any $\lambda>0$, there is a closed-open map (see Section 2 of \cite{Seidel:2010uq} for a review) \begin{align} \label{eq:zerooc} \phi^0_Q: CF^*(X \subset M ; H_{m}^{2\lambda})  \to CF^*(Q,Q) \end{align}

\begin{defn} \label{defn:sseqs}
Let $\underline{\Psi}$ be a cocycle in $CF^1(X \subset M ; H_{m}^{2\lambda})$. A $\underline{\Psi}$-equivariant structure on $Q$ is a choice of $c_Q \in CF^0(Q,Q)$ such that $\partial(c_Q)= \phi^0_Q(\underline{\Psi})$. Two $c_Q$ are equivalent if their difference is a coboundary.  
\end{defn}   

Given a cohomology class $\Psi \in HF^1(X \subset M ; H_{m}^{2\lambda}),$ a $\Psi$-equivariant structure on $Q$ is the choice of a cochain level lift $\underline{\Psi}$ of $\Psi$ together with a $\underline{\Psi}$-equivariant structure on $Q$. Note that if $\underline{\Psi}= \underline{\Psi}' +\partial_{CF}(a)$  the space of $\underline{\Psi}$ and $\underline{\Psi'}$ equivariant structures up to equivalence are canonically identified. 

In what follows, we will sometimes use $\tilde{Q}$ to denote a Lagrangian with an equivariant structure when we wish to supress the particular choice of $c_Q$. Given a pair of equivariant Lagrangians $\tilde{Q}_0, \tilde{Q}_1$, we can introduce certain canonical operators (see \eqref{eq:operation} for more details)  \begin{align} \Phi_{Q_0,Q_1}: HF^*(Q_0,Q_1) \to HF^*(Q_0,Q_1). \end{align} 
Let us assume that the coefficient ring $\K$ is a field and let $\bar{\K}$ denote the algebraic closure of $\K$. Then we can decompose $HF^*(Q_0,Q_1) \otimes_\K \bar{\K}$ into a direct sum (over $\sigma  \in \bar{\K}$) of its generalized $\sigma$-eigenspaces $HF^*(Q_0,Q_1)_\sigma$ for $\Phi_{Q_0,Q_1}$. We can define the Lefschetz super-trace of $\tilde{Q}_0, \tilde{Q}_1$ by the formula \begin{align} \label{eq: lefschetztrace} \tilde{Q}_0 \cdot \tilde{Q}_1= \sum_\sigma  \sigma \cdot \chi(HF^*(Q_0,Q_1)_\sigma) \in \K  \end{align} where $\sigma$ runs over elements of $\bar{\K}$ and $\chi(HF^*(Q_0,Q_1)_\sigma))$ is the Euler characteristic of the $\sigma$-eigenspace. 

 The following is the main Floer theoretic result of \cite{Seidel:2014aa}:

\begin{thm} \label{thm:SeidelFloer} The groups $\tilde{H}^*$ have a canonical non-degenerate pairing $I$. Furthermore, for any equivariant Lagrangian $\tilde{Q}$, there is a canonical element $[[\tilde{Q}]] \in \tilde{H}^*$ such that for any pair of equivariant Lagrangians $\tilde{Q}_0, \tilde{Q}_1$, we have that $$I([[\tilde{Q}_0]], [[\tilde{Q}_1]])= (-1)^{n(n+1)/2} \tilde{Q}_0 \cdot \tilde{Q}_1$$ \end{thm}

 This should be thought of as a kind of Cardy relation in the (infinitesimally) equivariant context. We are now in a position to prove:  

\begin{lem} 
    \label{lem: genquasi}  
     Let $(M, \D)$ be a pair with
    $n=dim_\mathbb{C} M$ odd. Suppose that there exists a slope $\lambda$ such
    that: \begin{enumerate}[(a)] \item  there exists an element $\Psi \in HF^{1}(X \subset M ;
    H_{m}^{2\lambda})$ such that  $\mathfrak{c}_{\lambda,\infty}(\Psi) \in
    SH^1(X,\mathbb{Z})$ defines the degree 1 piece of an integral quasi-dilation.
\item the natural map $H^*(X,\mathbb{Z}) \to HF^{*}(X \subset M ; H_{m}^\lambda)$ is an
    isomorphism.  \end{enumerate} Suppose further
    that $Q_1,\cdots, Q_r$ is a collection of embedded Lagrangian spheres which
    are pairwise disjoinable. Then  for any field $\mathbf{k}$, the classes $[Q_1],\cdots,[Q_r]$ span a
    subspace of $H_n(X,\mathbf{k})$ which has rank at least $r/2$.   
\end{lem} 

\begin{proof}
The proof follows the proof of
 Theorem 1.7 of \cite{Seidel:2014aa} almost verbatim. (The only modification we must make is that in our situation, there are
 non-trivial signs appearing in the computation (3.17) of \emph{loc. cit}.) More precisely, fix a field $\K$. We can transfer the element $\Psi$ to obtain a quasi-dilation over $\K$. For the remainder of the proof all (Floer) cohomology groups will be taken to have $\K$-coefficients.  Because the $Q_i$ are spheres and $H^1(Q_i,\K)=0$, all $Q_i$ can be given a $\Psi$-equivariant structure $\tilde{Q}_i$. We use these $\Psi$-equivariant structures to obtain endomorphisms $\Phi_{Q_i,Q_i}$ on $HF^*(Q_i,Q_i)$ according to the formula from \eqref{eq:operation}. By Equation \eqref{eq:actiondegn}, we have that this operation acts by $\pm 1$ on the degree $n$ piece (here we use assumption (a) and in particular that our quasi-dilation is defined over
 $\mathbb{Z}$). 
 By \eqref{eq:zzzzzsphere},  we therefore have that the Lefschetz traces from \eqref{eq: lefschetztrace} are given by (compare with Equation (3.17) of \cite{Seidel:2014aa}) 
\begin{align} 
    \tilde{Q}_i \cdot \tilde{Q}_j = \left\{
     \begin{aligned} & \pm 1 \quad \text{for} \quad i=j \\ & 0 \quad
             \text{for} \quad i \neq j \\ \end{aligned}
\right. 
 \end{align} Theorem \eqref{thm:SeidelFloer} then implies that the elements $[[\tilde{Q}_i]] \in \tilde{H}^*(X)$ are linearly independent and generate a subspace of dimension $r$. 

The crucial point in Seidel's proof is that the image of $HF^{*}(X \subset M ; H_{m}^{\lambda}) \to \tilde{H}^*(X)$ from  \eqref{eq:leseq} is isotropic for the pairing $I$.\footnote{This is true for any $\Psi \in HF^1(X \subset M ; H_{m}^{2\lambda})$, not just dilations or quasi-dilations.} Hence, the intersection of the subspace generated by $[[\tilde{Q}_i]]$ with this subspace has at most rank $r/2$. Using the exact sequence \eqref{eq:leseq}, it therefore follows that if we let $c$ denote the natural map $c: \tilde{H}^*(X) \to HF^{*}(X \subset M ; H_{m}^{-\lambda})$, $c(\operatorname{Span}\lbrace [[\tilde{Q}_i]]\rbrace)$ has dimension $\geq r/2.$ Now by dualizing assumption (b), we have that the natural map $\PSS^{\vee}: HF^{*}(X \subset M ; H_{m}^{-\lambda}) \to H_{n}(X,\K)$ is an isomorphism. Lastly, by combining Equation (3.11) of \cite{Seidel:2014aa} with the discussion below (2.39) of \emph{loc. cit.}, $\operatorname{PSS}^{\vee} \circ c([[\tilde{Q}_i]])= [Q_i]$. It follows that $\operatorname{Span}\lbrace [Q_i] \rbrace \in  H_{n}(X,\K)$ has dimension $\geq r/2$ as desired. 
\end{proof}

\begin{thm} \label{thm: app1sec6} 
Let $\mathbf{k}$ be a field and $n \geq 3$ be an odd integer. Suppose that $X$ is an affine conic bundle of total
dimension $n$ over an affine variety $X^o$ with trivial canonical bundle  and
that $Q_1,\cdots, Q_r$ is a collection of embedded Lagrangian spheres in $X$ which are
pairwise disjoinable. Then the classes $[Q_1],\cdots,[Q_r]$ span a subspace of
$H_n(X,\mathbf{k})$ which has rank at least $r/2$.   
\end{thm}
\begin{proof} 
    Choose a compactification $(\hat{M},\hat{D})$ associated to a good compactification of $Z^o \hookrightarrow X^o.$ Choose the parameters $\epsilon$ and $\delta$
    appearing in the definition of $\bar{X}$ to be sufficiently small. Then because the coefficient of $\hat{D}_1$ in \eqref{eq:conicpolar} is the smallest,
    the period computations in (\cite{McLean2} Theorem 5.16) show that the Reeb orbit
    which winds once around $\hat{D}_1$ is the Reeb orbit of lowest period
    $\lambda_0$. Choose a generic slope $\lambda$ such that
    $2\lambda>\lambda_0>\lambda$. Then $\Psi
    :=\PSSlog^{2\lambda}(\beta_{1,c}t^{\vec{\v}_1},Z_b) \in
    HF^{*}(X \subset \hat{M} ; H_{m}^{2\lambda})$ defines the degree 1 part of a
    quasi-dilation by combining Theorem  \ref{thm: generalquasi}  and Lemma \ref{lem:goodcompactificationGW}. Moreover, we have an isomorphism $H^*(X) \cong
    HF^{*}(X \subset \hat{M} ; H_{m}^\lambda)$. Now apply Lemma \ref{lem: genquasi}.
\end{proof}

\begin{cor} 
    \label{cor: primitiveclass} 
For any embedded Lagrangian sphere $Q \hookrightarrow X$, the class $[Q] \in H_n(X,\mathbb{Z})$ is non-zero and primitive.
\end{cor} 
\begin{proof} 
    Letting $r=1$ in Theorem \ref{thm: app1sec6}, we see that the class $[Q]
    \in H_n(X,\mathbf{k})$ is non-zero for every field $\mathbf{k}$. This
    implies the statement of the corollary.  
\end{proof}

\subsubsection{Lagrangians in three-dimensional affine conic bundles}

We now turn to our other main application involving 3-dimensional conic bundles over affine surfaces: 

\begin{thm} \label{thm: app2sec6gen} 
    Let $X^o$ be an affine surface with trivial canonical bundle and let $X$ be a 3-dimensional conic bundle
    over $X^o$ of the form \eqref{eq: conicbundle2}. Suppose that $j: Q
    \hookrightarrow X$ is a closed, oriented, exact Lagrangian submanifold of
    $X$. Then $Q$ is diffeomorphic to $S^1\times B$ for $B$ of genus $\geq 1$ or $\#_n S^1 \times S^2$ (as usual $n=0$ corresponds to $S^3$ by convention). 
\end{thm}
\begin{proof} 
    This follows immediately by combining Corollary \ref{cor: Vitdil} and Lemma \ref{lem:goodcompactificationGW}. 
\end{proof}

The classification of Theorem \ref{thm: app2sec6gen} is in a sense optimal. To explain this, note first that if $B$ is an exact embedded genus $g$ surface in some $X^o$, then $S^1 \times B$ is exactly embedded in $\mathbb{C}^* \times X^o$ (one can construct similar examples when $Z^o$ is non-empty). To describe how one constructs exact Lagrangian $\#_n S^1 \times S^2$ in conic bundles, it is useful to recall that a semi-free group action on a manifold $Q$ is a non-trivial $S^1$ action such that the isotropy subgroup of every point is either $\lbrace id \rbrace$ or all of $S^1$. Let $Q$ be a closed, oriented 3 manifold with a semi-free $S^1$ action such that the quotient space $Q/S^1=:\Sigma_{g,h}$ is an oriented surface of genus $g$ with $h$ boundary components.  Whenever $h>0$, we have that  $Q \cong \#_n S^1\times S^2$ where $n=2g+h-1$ (See Theorem 1 of \cite{MR219086}). Turning to symplectic topology, one can arrange that the Kahler form on $X$ is $S^1$-invariant under the action $$ e^{i\theta}\cdot (u,v,\bar{x})=(e^{i\theta}u, e^{-i\theta}v,\bar{x}).$$  The fixed point locus of this action occurs where $u=v=0$, which we will assume lies at moment level set $\mu=\epsilon_Z$. Given an embedded exact Lagrangian surface $j_o: \Sigma_{g,h} \to X^o$ with $j_o^{-1}(Z^o)=\partial \Sigma_{g,h}$, we let $Q$ denote the set of points
\begin{align} \label{eq: suspendingLs} \lbrace x \subset X,  \pi_{X^o}(x) \subset
j_o(\Sigma_{g,h}), \mu(x)= \epsilon_Z \rbrace  \end{align} 
where $\pi_{X^o}$ is the natural projection $X \to X^o$. In nice examples, $Q$ is an embedded exact Lagrangian 3-fold which inherits a semi-free action of $S^1$ from the action on $X$ (the action has fixed points which lie over $\partial \Sigma_{g,h}$). We give some specific examples in Appendix \ref{section:AppendixA}.

For a given affine base $X^o$, it may be possible to improve this classification. For example, in view of the explicit constructions of exact Lagrangians outlined above, it is natural to ask: 
\begin{ques} Does there exist an exact embedding of $S^1 \times B$, for $B$ of genus $\geq 2$ inside of a conic bundle over $(\mathbb{C}^*)^2$? \end{ques}
We will now explain, using ideas from mirror symmetry, that the answer to this question is ``no." To recall how the mirror to $X$ is constructed following \cite{AAK, CPU}, let $\mathcal{P}$ denote the Newton polytope of
$f(\bar{x})$ in $M_\mathbb{R}$ and let $\Sigma$ be the fan in $M_\mathbb{R}
\oplus \mathbb{R}$ associated with a coherent unimodular triangulation of
$\mathcal{P}$. Set $\bar{X}^{\vee}$ to be the associated toric variety which by definition carries a dense algebraic torus
$(\mathbb{C}^*)^3.$ 

A function on $\bar{X}^{\vee}$ is determined by its restriction to $(\mathbb{C}^*)^3$ and hence can be described by a unique character $(\mathbf{n},k) \in N \oplus \mathbb{Z}$ (here
$N$ denotes the dual lattice to $M_\mathbb{Z} \subset M_\mathbb{R}$ as is standard in the literature on toric varieties). We set $\chi_{\mathbf{n},k} : \bar{X}^{\vee} \to
\mathbb{C}$ to be the function associated to a given character. Of course not every function on $(\mathbb{C}^*)^3$ extends to a function on $\bar{X}^{\vee}$, but the definition of the toric variety  $\bar{X}^{\vee}$ implies that functions which do extend are determined in a straightforward way by the Newton polytope. Namely if $A$
denotes the set of lattice points inside the polytope, we have that $\Gamma (\mathcal{O}_{\bar{X}^{\vee}})= \bigoplus_{\mathbf{n},k \in \mathcal{C}} \mathbb{C} \cdot \chi_{\mathbf{n},k}$ where 
\begin{align} \mathcal{C}:= \lbrace (\mathbf{n},k) \in N \oplus \mathbb{Z} \text{ such that } \mathbf{n}(m)+k \geq 0 \text{ for all } m \in A. \rbrace \end{align} 

In particular, for any $\mathbf{n}$, there is always a $k_0$ for which $\chi_{\mathbf{n},k}$ defines a regular function on $\bar{X}^{\vee}$ whenever $k \geq k_0$. Note that multiplication of functions in this ring is just given by addition in the character lattice.  Set $p= \chi_{0,1}-1$ and let $H$ be the (anticanonical) divisor in $\bar{X}^{\vee}$ defined by $p^{-1}(0)$. Then motivated by the geometry of Lagrangian torus fibrations, \cite{AAK} have predicted that $X^{\vee} := \bar{X}^{\vee}
\setminus H$ should be a mirror to $X$. The assertion that these two spaces are mirror implies a number of predictions concerning the symplectic topology of $X.$ Perhaps the simplest such prediction was recently confirmed by the following calculation from \cite{CPU}: 
\begin{lem}\label{lem:CPUlemma} (Theorem 8.2 of \cite{CPU})  There is an isomorphism of rings: \begin{align} \label{eq: CPUiso} \Gamma(\mathcal{O}_{X^{\vee}}) \cong SH^0(X,\mathbb{C}). \end{align} Under this isomorphism, the weight $\mathbf{n}$ above is identified with the homology class of a Hamiltonian orbit in $H_1(X,\mathbb{Z}) \cong \mathbb{Z}^2.$ \end{lem}

In the proof of Proposition \ref{thm: app2sec6}, we will only use the following two properties/consequences of this isomorphism:
  \begin{enumerate}
\item Pick any non-zero vector $\mathbf{n}$ and let $k$ be a positive integer so that $\chi_{\mathbf{n},k}, \chi_{-\mathbf{n},k}$ both extend to functions on $\bar{X}^{\vee}$ and hence $X^{\vee}$.
 Then we have that 
\begin{align} \label{eq: identityI} \chi_{\mathbf{n},k} \cdot \chi_{-\mathbf{n},k}= \chi_{0,2k}=(1+p)^{2k} \end{align}
We may use \eqref{eq: CPUiso} to obtain classes in $SH^0(X,\mathbb{C})$ (which we give the same name) satisfying the same equation.
\item The only units corresponding to contractible orbits under \eqref{eq: CPUiso} are all of the form $a \cdot p^d$ for some $d$ and $a \in \mathbb{C}^*$.
\end{enumerate}

\begin{prop} \label{thm: app2sec6} 
    Let $X$ be a 3-dimensional conic bundle
    over $(\mathbb{C}^*)^2$ of the form \eqref{eq: conicbundle2} and such that the discriminant locus $Z^o$ is connected. Let $j: Q
    \hookrightarrow X$ be a closed, oriented, exact Lagrangian submanifold of
    $X$. Then $Q$ is diffeomorphic to either $T^3$ or $\#_n S^1 \times S^2$. 
\end{prop}
\begin{proof} 
  In view of Theorem \ref{thm: app2sec6gen}, it remains to rule out those
  manifolds $S^1\times B$ with
  genus of $B \geq 2$. Let $j: Q \hookrightarrow X$ be a putative exact embedding of $S^1 \times B$ with $g \geq 2.$ Recall from \eqref{eq:abschmap} that we have a canonical isomorphism: \begin{align} \label{eq:refrecall} SH^*(T^*Q) \cong \mathbb{H}^*(\mathcal{L}Q). \end{align} By \eqref{eq:refrecall} and Lemma \ref{lem:centre} (which applies because $Q$ is aspherical)  we have that \begin{align} \label{eq:isoforref} SH^0(T^*Q) \cong \mathbb{H}^0(\mathcal{L}Q) \cong  \mathcal{Z}(\mathbb{C}[\pi_1(Q)]). \end{align} Again using the fact that  $\mathcal{Z}(\mathbb{C}[\pi_1(B)])= \mathbb{C}$ (see e.g. page 564 of \cite{MR2379052}), the center of $\mathbb{C}[\pi_1(Q)]$ is of the form  \begin{align} 
\label{eq:centerb2} \mathcal{Z}(\mathbb{C}[\pi_1(Q)]) \cong \mathbb{C}[h,h^{-1}] 
\end{align} 
where $h$ is a generator of $\pi_1(S^1).$  For a possibly inhomogeneous element $z \in SH^*(T^*Q),$ let $$[z]_0 \in SH^0(T^*Q)$$ denote its degree zero piece.  As noted in the discussion following \eqref{eq:loopdecomp}, because $Q$ is aspherical, we have that $\mathbb{H}^*(\mathcal{L}Q)$ is concentrated in non-negative degree (and hence the same is true for $SH^*(T^*Q)$ in view of \eqref{eq:refrecall}). This implies that if $z$ is an invertible (possibly inhomogeneous) element of $SH^*(T^*Q)$, $[z]_0$ is invertible in $SH^0(T^*Q).$

For the remainder of the proof, we identify elements of $SH^0(X)$ with functions on $X^{\vee}$ using \eqref{eq: CPUiso} and $SH^0(T^*Q) \cong \mathbb{C}[h,h^{-1}]$ using the composition of \eqref{eq:isoforref} and \eqref{eq:centerb2}.  Since $p \in SH^0(X)$ is invertible, then as noted above, we have that $[j^!(p)]_0$ is also invertible, where $j^!$ is the Viterbo restriction map \eqref{eq:Viterbo} from $SH^*(X)$ to $SH^*(T^*Q)$. Then \begin{align} \label{eq:jrestp} [j^!(p)]_0= b \cdot h^s \end{align} for some integer $s$ and $b \in \mathbb{C}^*$ because these are the only units in a Laurent polynomial ring. Next, recall the quasi-dilation $$(\mathfrak{c}_{\lambda,\infty} \circ \PSSlog^{\lambda}
        (\beta_{1,c}t^{\vec{\v}_1},Z_b), \PSSlog
        (\alpha_{1,c}t^{\vec{\v}_1}))$$ constructed in Theorem \ref{thm: generalquasi}.  We also have that 
\begin{align}\label{eq:delta} \Delta(\mathfrak{c}_{\lambda,\infty} \circ \PSSlog^{\lambda}
        (\beta_{1,c}t^{\vec{\v}_1},Z_b))=\PSSlog(\alpha_{1,c}t^{\vec{\v_1}})=a \cdot p^d \end{align} for some $d$ and $a \in \mathbb{C}^*$ because the only units corresponding to contractible orbits under \eqref{eq: CPUiso} are all of this form.  

This implies that $s$ from  \eqref{eq:jrestp} is \emph{non-zero} because $[j^!(p^d)]_0=([(j^!(p))]_0)^d$ is in the image of the BV-operator\footnote{We have $\Delta(\frac{1}{a}[j^{!}(\mathfrak{c}_{\lambda,\infty} \circ \PSSlog^{\lambda}
(\beta_{1,c}t^{\vec{\v}_1},Z_b))]_0)= [j^!(p^d)]_0$ by \eqref{eq:delta} and Lemma \ref{lem:Viterboalgebraic}.} and hence must lie in (a summand of symplectic cohomology corresponding to) a non-zero free homotopy class.  The fact that $s \neq 0$ in turn implies that $j_*(h)=0 \in H_1(X)$ because $p$ corresponds to the trivial free homotopy class.  Pick any non-zero vector $\mathbf{n}$ and let $k$ be a positive integer so that $\chi_{\mathbf{n},k}, \chi_{-\mathbf{n},k}$ both extend to functions on $\bar{X}^{\vee}$ and hence $X^{\vee}$. Then $s \neq 0$ also implies that \begin{align} [j^!((1+p)^{2k})]_0 =(1+ b \cdot h^s)^{2k} \neq 0.\end{align} Hence, by Equation \eqref{eq: identityI}, $[j^{!}(\chi_{\mathbf{n},k})]_0 \neq 0 \in SH^0(T^*Q) \cong \mathbb{C}[h,h^{-1}]$.
However, this is a contradiction as $[j^{!}(\chi_{\mathbf{n},k})]_0$ cannot be a Laurent polynomial in $h$ 
 because $\mathbf{n} \neq 0$ but $j_*(h)=0$. 
\end{proof}

 \begin{rem} The key equation \eqref{eq: identityI} 
    could also potentially be proven using the methods in this paper. Namely, consider the
   simplest case of conic bundles over $(\mathbb{C}^*)^2$ which compactify to
   conic bundles over $\mathbb{C}P^2$(in general $M$ above can be taken to be
    a blowup of $\mathbb{C}P^2$) with its three toric divisors $\bar{D}_j$.
    Next, let $\hat{D}_j$ be the corresponding divisors in $X$. Then the
    classes $\PSSlog(\alpha_{\hat{D}_1}t^{\vec{\v}_1})$ and
    $\PSSlog(\alpha_{\hat{D}_2 \cap \hat{D}_3}t^{\vec{\v}_J})$ corresponding to
    the fundamental chains on $\mathring{S}_1$ and $\mathring{S}_{[2,3]}$ can
   be shown to satisfy \eqref{eq: identityI} using a more complicated variant
    of Lemma \ref{lem: multiply}.  
  \end{rem}

\begin{rem} It is reasonable to expect that for a given Laurent polynomial $f$, there exists only finitely many $n$ for which
  $\#_n S^1 \times S^2$ embeds as an exact Lagrangian in the conic bundle $X$ determined by $f$. Proving this, however, seems to require new ideas. 
\end{rem}

\subsection{Knottedness of Lagrangian intersections}\label{sec:knottedness}

The main result of this section is Proposition \ref{unknottedness}, concerning
the rigidity of the knot types of Lagrangian 3-manifolds meeting cleanly in a
knot. This result is in the vein of Question \ref{ques:smith} stated in
the introduction and builds upon work of Evans-Smith-Wemyss
\cite{evanssmithwemyss}. First, some preliminaries:
Let $S_1$ and $S_2$ be two 3-spheres, and $S^1
\stackrel{\kappa_i}{\hookrightarrow} S_i$ a pair of knots together with a identification of normal bundles
$$\eta :\nu_{S^1}S_1 \cong \nu_{S^1}S_2 $$

It is well-known that the above data determines a Stein manifold,
\[
    W_\eta(\kappa_1,\kappa_2) = T^*S_1 \#_{S^1,\eta} T^* S_2
\]
the result of plumbing the cotangent bundles of $S_i$ along the knots
$\kappa_i$, using the identification $\eta$ of normal bundles of $\kappa_1$
with $\kappa_2$. We briefly recall the construction of $W_\eta(\kappa_1,\kappa_2)$, referring the reader to \cite{Abouzaid:2011fp}*{\S A} for complete details. Choose Riemannian metrics on $S_i$ (assume for simplicity they induce the same metric on $S^1$) and let $D^*S_i$ denote the disc cotangent bundle with respect to this metric. Then by Weinstein's isotropic embedding theorem there are open neighborhoods $US^1_i$ of $S^1 \subset D^*S_i$ which are symplectomorphic to a disc sub-bundle of the symplectic vector bundle $\nu_{S^1}S_i \otimes \mathbb{C}=\nu_{S^1}S_i \oplus \sqrt{-1}  \nu_{S^1}S_i $ pulled back to $D^*S^1$ (of  small radius). The Weinstein neighborhoods can be chosen so that the images of the real factors $\nu_{S^1}S_i$ over the zero section of $D^*S^1$ are precisely $US^1_i \cap S_i.$ $W_\eta$ is obtained by gluing $US^1_1$ to $US^1_2$ by the map $\eta \otimes \cdot \sqrt{-1}$ obtained from composing the given isomorphism of normal bundles with multiplication by $\sqrt{-1}$ in the complexified bundle.  

These manifolds $W_\eta(\kappa_1,\kappa_2)$ retract onto $S_{1}\cup_{S^1} S_2$, from which it follows using Mayer-Vietoris that they all satisfy \begin{align} H^2(W_\eta) \cong H_2(W_\eta) \cong \mathbb{Z} \\ H^3(W_\eta) \cong H_3(W_\eta) \cong \mathbb{Z}\oplus \mathbb{Z} \end{align} and all other (co)homology groups in positive degree vanish.

These Stein manifolds are the local models for spheres intersecting cleanly in
a circle. More precisely, if $(X, \omega)$ is any symplectic manifold which
contains a pair of Lagrangian spheres $S_i$ meeting cleanly along a circle,
then $\omega$ induces an isomorphism $\eta$ of the underlying (unoriented)
normal bundles. An open neighborhood of $S_1 \cup_{S^1} S_2$ is then symplectically
equivalent to $W_\eta(\kappa_1,\kappa_2)$. In the case that $X$ is itself a
Liouville domain, the inclusion of a small closed tubular neighborhood of
$S_{1}\cup_{S^1} S_2$ will be a Liouville embedding (this follows from the vanishing
of $H^1(W_\eta)$). 

The normal bundles $\nu_{S^1}S_1$ are trivial and the space of identifications
(up to homotopy) form a torsor over $\mathbb{Z}$. We next observe that it is
possible to fix a canonical base-point for this torsor. To do this, note that
there are two possible fibered {\em (Polterovich) Lagrangian surgeries}
of the zero sections $S_1$
and $S_2$. Observe that because $S^1 \subset W_\eta$ is isotropic, a
neighborhood $U(S^1) \subset W_\eta$ may be identified with the total space of
a disc sub-bundle of a rank four symplectic vector bundle $E$ over $D^*S^1$.
The normal bundles $\nu_{S^1}S_i$ give rise to transverse Lagrangian subbundles
of $E$, showing that $E$ is in fact a trivial symplectic bundle. In view of
this, we may perform the Polterovich surgery construction in each fiber, where
one performs one of the two possible surgeries of 2-planes $\mathbb{R}^2 \cup i
\mathbb{R}^2 \subset \mathbb{C}^2$.  The result are exact Maslov index 0
Lagrangian submanifolds $K$,$K'$ (see e.g., the discussion in
\cite{mak_wu_2018}*{\S 2.2.2} for more details).

 In the case $\kappa_i$ are both linearly embedded unknots, the resulting
 manifolds $K$ and $K'$ are given by Dehn surgery of integer slope (depending on the identification of the normal bundles) and hence
 give rise to a $S^1 \times S^2$, $S^3$ or a lens space. Let $W_n$ denote the
 unique plumbing for which $H_1(K,\mathbb{Z}) \cong \mathbb{Z}/ n\mathbb{Z}$.\footnote{\cite{evanssmithwemyss} show more generally that that there is a unique
 normal identification for which $H_1(K,\mathbb{Z}) \cong \mathbb{Z}/n$.
 They denote by $W_{n}(\kappa_1,\kappa_2)$, the resulting plumbing.} In
 their study of $W_1$ (and related plumbings $W_n$), Evans, Smith and
 Wemyss proved a result in the direction of a negative answer to Question
 \ref{ques:smith}:
\begin{prop}[Evans-Smith-Wemyss \cite{evanssmithwemyss}]\label{prop:esw}
If there is a Hamiltonian isotopy $\tilde{S}_1$ of $S_1$ and $\tilde{S}_2$ of
$S_2$ in $W_1$ so that $\tilde{S}_1$ meets $\tilde{S}_2$ cleanly along a knot,
then this knot must be the unknot in one component and the unknot or trefoil in the other component.
\end{prop}

The goal of this section is to prove the following:

\begin{prop}\label{unknottedness}
Given two Lagrangians spheres $\tilde{S}_1$ and $\tilde{S}_2$ in $W_1$ so that $\tilde{S}_1$
meets $\tilde{S}_2$ cleanly along a knot, then this knot must be the unknot in both components.
\end{prop}

In particular, our main result rules out the remaining trefoil possibility (and along the way
gives an alternate proof of Proposition \ref{prop:esw}). The key starting point
for us is that, as noted in \cite{evanssmithwemyss}, $W_1$ embeds into an explicit affine variety. Namely, let $M$ be a flag threefold, realized as a
hypersurface of bidegree $(1,1)$ in $\mathbb{P}^2 \times \mathbb{P}^2$ and $D$ be a smooth hyperplane $(1,1)$ section.  The following construction is due to \cite{evanssmithwemyss}: \begin{lem} \label{lem:eswembedding} There is an embedding $W_1 \hookrightarrow M\setminus D.$  \end{lem} \begin{proof} $M$ is explicitly the hypersurface given by $$ M:= \lbrace xx'+yy'+zz'=0  \subset \mathbb{P}^2 \times \mathbb{P}^2 \rbrace  $$ Consider the pencil of $(1,1)$ hypersurfaces in $M$ given by $$ D_{\mu_1, \mu_2}:= \lbrace \mu_1(xx'-zz') + \mu_2(yy')=0 \rbrace \cap M \subset M, \quad [\mu_1:\mu_2] \subset \mathbb{P}^1. $$ The fiber $D_{1,0}$ is a smooth complex surface and contains the base locus of the pencil which is the locus $\lbrace xx'=zz'=yy'=0 \rbrace \cap M$. Thus, setting $X=M\setminus D_{1,0}$, we obtain a fibration $W:X \to \mathbb{C}$ with general fiber $(\mathbb{C}^*)^2$. Furthermore, it is not difficult to check that this fibration is of Morse-Bott type with critical values $\lbrace -1,0,1 \rbrace$ and critical fibers isomorphic to $(\mathbb{C}^*) \times (\mathbb{C} \vee \mathbb{C})$. The map $W$ is $T^2$-equivariant with respect to the torus which acts on the ambient $\mathbb{P}^2 \times \mathbb{P}^2$ by the formula 
    $$ (\theta, \phi) \cdot ([x:y:z],[x':y':z'])= ( [e^{i\theta} x:y: e^{i\theta} z],[e^{-i\phi} x':y':e^{-i\phi} z']).$$
Consider the paths $\gamma_1:=[-1,0] \subset \mathbb{C}$ and $\gamma_2:=[0,1] \subset \mathbb{C}$.
For each of the paths $\gamma_i$, there is a $T^2$ equivariant Lagrangian 3-sphere $S_i$ which is the union of two Morse-Bott Lefschetz thimbles (a Morse-Bott matching cycle) fibering over that path.  Let $W_i$ denote the restriction of $W$ to $S_i$. For any $c \in \gamma_i$, the fibers $W_i^{-1}(c)$ are orbits of the group action and moreover this action is free unless $c \in \lbrace -1,0,1 \rbrace$. Over these endpoints, the action on $S_i$ has $S^1$ stabilizer. We have that $S_1 \cap S_2=W_1^{-1}(0)=W_2^{-1}(0) \cong S^1$. Taking a Weinstein tubular neighborhood of  the configuration $S_1\cup_{S^1} S_2 \hookrightarrow X$ gives the desired embedding of $W_1 \hookrightarrow X.$
\end{proof}

 \begin{rem} It is not difficult to see that the inclusion $W_1 \hookrightarrow M \setminus D$ is a homotopy equivalence.\footnote{In fact Evans, Smith, and Wemyss have shown that $W_1$ is deformation equivalent to $M \setminus D$ as Stein manifolds, but we do not need this stronger statement.} \end{rem}

 It is worth mentioning at this stage that $D$ is a del Pezzo surface of degree $6$ (this will account for the use of
characteristic 3 coefficients in our arguments). To check this, note that by the adjunction formula, we have that $\mathcal{O}(D)_{|D}= K_D^{-1}$ and so to calculate the degree we need to calculate $D_N \cdot D_N$ where $D_N$ is a smooth divisor representing $\mathcal{O}(D)_{|D}.$ Let $f:M \to \mathbb{P}^2 \times \mathbb{P}^2$ be the embedding of $M$ into $\mathbb{P}^2 \times \mathbb{P}^2$. Set $H_1:=H \times [\mathbb{P}^2]\in H_6(\mathbb{P}^2 \times \mathbb{P}^2)$ and $H_2:=[\mathbb{P}^2] \times H \in H_6(\mathbb{P}^2 \times \mathbb{P}^2)$ ($H$ denotes the hyperplane class in $H_2(\mathbb{P}^2)$) so that $[M]=H_1 + H_2 \in H_6(\mathbb{P}^2 \times \mathbb{P}^2)$. We have that $f_*([D])= f_*((H_1+H_2)_{|M})=(H_1+H_2)^2  \in H_4(\mathbb{P}^2 \times \mathbb{P}^2)$ and so \begin{align} \label{eq:degree6} D_N \cdot D_N=(H_1+H_2)^2 \cdot (H_1+H_2)^2=6. \end{align} Because $W_1$ embeds in $M \setminus D$, we can use a variant of the argument of Theorem \ref{thm: dilationcrit} (compare also with \cite{catdyn}*{Conjecture 18.6}) to prove:

\begin{prop}\label{dilationchar3}
$W_1$ admits a dilation over any field $\mathbf{k}$ of characteristic 3.
\end{prop}
\begin{proof} 
By Viterbo functoriality (see \eqref{eq:Viterbo}), it suffices to prove this for $M \setminus D.$ We explain how to modify Theorem \ref{thm: dilationcrit} and use the notation from that Theorem and its proof. We have that $m=2$ and we let $\alpha_0=1 \in H^0(M)$. Choose $D_N$ as before and we set $L:=Bl_{D_{N}}D \subset S_D$ to be the real-oriented blowup along $D_N$ whose boundary is $ \partial L :=\pi^{-1}(D_N) \subset S_D$.  Then $\mathcal{O}(D)$ restricted to $D_N$ has degree 6 by Equation \eqref{eq:degree6}. Choose a point $p$ on $D_N$ and let $K':= Bl_p(D_N)$ be the real-oriented blow up of $D_N$ at $p$, which is again a submanifold inside of $S\mathcal{O}(p)$, the circle bundle associated to $\mathcal{O}(p)$. Because $(S_D)_{|D_{N}}$ has degree 6, there is a natural 6:1 fiberwise covering $w: S\mathcal{O}(p) \to (S_D)_{|D_{N}}$ over $D_N$: \[ \begin{tikzcd}
S\mathcal{O}(p) \arrow{r}{w}  \arrow{rd}{\pi_p} 
  & (S_D)_{|D_{N}} \arrow{d}{\pi_D} \\
    & D_N
\end{tikzcd} \] We let $K$ be the pushforward of $K'$ to $S_D$ along this map.

Along its boundary, $K$ is a 6:1 covering onto the fibers of $(S_D)_{|D_{N}}$ over $p$ and hence $K$ represents a class in $H_2(S_D,\mathbf{k}) \cong H^3(S_D,\mathbf{k})$.  Moreover, the pushforward $j_{D,*}(D_N)=3(A_1+A_2) \in H_2(M)$, where $A_1$ and $A_2$ are the line classes coming from the two $\mathbb{P}^2$ factors. This implies, by a variant of Lemma \ref{lem:gweq6}, that Equation \eqref{eq:rk1coh} holds as well with $\mathbf{k}$ coefficients. Thus we can choose some bounding cochain $Z_b$ to define a symplectic cohomology class $\operatorname{PSS}_{log}^{\lambda}(Kt,Z_b)$. Finally, we note that $GW_{M}(\PD(D),\alpha_1,pt)=2$ (see e.g. Section 4 of \cite{Iritani} for a convenient reference) and so following Lemma \ref{lem:pssBV6}, we have   
$$\Delta (\operatorname{PSS}_{log}^{\lambda}(Kt,Z_b))= - \PSS_{\vec{\v}_{\emptyset}}(\underline{GW_{\vec{\v}_1}(L)}) =-2.$$  
\end{proof}
\begin{rem} 
    We note that although our typical ground ring convention in this paper is $\mathbf{k} =
    \Z$, $\C$, or $\mathbb{Q}$, the above argument works with $\mathbf{k}$ a field of
    characteristic 3. To explain both the reason for our usual convention and why
    the above argument encounters no related issues, see Remark \ref{rem:groundring}.
\end{rem}

\begin{lem} \label{lem: laginW} Let $i: Q \hookrightarrow W_1$ be an exact Lagrangian embedding of a closed orientable manifold. Then $Q$ admits a prime decomposition with prime summands $S^1 \times S^2$ and spherical spaces forms $Q_i$ such that $|\pi_1(Q_i)|$ is not divisible by 3. \end{lem} 

\begin{proof} This follows from Viterbo functoriality by combining Proposition \ref{dilationchar3} together with the claim (2) of Lemma \ref{lem:dilatp} (with $p=3$). \end{proof}

\begin{rem} \label{rem: Poincare} In particular, none of the summands $Q_i$ can be a Poincar\'e homology sphere $S^3/\pi$ with $|\pi|=120$. This is sufficient to rule out the case of the trefoil from Proposition \ref{prop:esw} because in the case of a trefoil the Polterovich surgery of $\tilde{S}_1$ and $\tilde{S}_2$ would be a Poincar\'e homology sphere. Note that Proposition \ref{unknottedness} is slightly stronger as it concerns \emph{a priori} arbitrary Lagrangian spheres $\tilde{S}_1$ and $\tilde{S}_2$ meeting cleanly in a circle. \end{rem}

We record one final preparatory Lemma before proceeding with the proof of Proposition \ref{unknottedness}.

\begin{lem} \label{lem: spherenot3} The only spherical space forms $Q$ such that $|\pi_1(Q)|$ is not divisible by $3$ and $|H_1(Q)| \leq 2$ are $S^3$ and $\mathbb{R}P^3.$  \end{lem} 
\begin{proof} This statement is obvious if one restricts to lens spaces. The general case follows from the classification of spherical space forms; see Theorem 2.2 of \cite{MR2484716} for a table listing the fundamental groups of all spherical space forms which are not lens spaces. In particular, inspecting this classification shows that all of the space forms (other than lens spaces) such that $|\pi_1(Q)|$ is not divisible by 3 are ``prism manifolds." There are two types of prism manifolds. The first type have fundamental group  \begin{align} Q_{4n}:= \langle x,y| x^n=y^2, xyx=y \rangle \end{align} for some $n \geq 1$. The abelianization ($=H_1(Q)$) of $Q_{4n}$ has order 4 (see the beginning of Section 3 of \emph{loc. cit.}). In the second case, the fundamental group is \begin{align} B_{2^{k}(2n+1)}:= \langle x,y| x^{2^{k}}=y^{2n+1}=1, xyx^{-1}=y^{-1} \rangle \end{align} for some $k \geq 2, n \geq 0$. As discussed in Section 5.2 of \emph{loc. cit.}, the abelianization of $B_{2^{k}(2n+1)}$ has order $2^k>2.$  \end{proof}

\begin{proof}[Proof of Proposition \ref{unknottedness}]
  Form a Polterovich surgery $Q$ of $\tilde{S}_1$ and $\tilde{S}_2$, which as we have seen will be an exact, orientable Lagrangian. If both knots are non-trivial, then $Q$ contains an incompressible torus and is irreducible by Proposition 2.3 of \cite{Hedden}. Lemma \ref{lem: laginW} shows that the only possible irreducible Lagrangians in $W_1$ are spherical space forms. As none of the spherical space forms contain an incompressible torus,\footnote{For instance because an incompressible torus gives rise to an injective map $\pi_1(T^2) \to \pi_1(Q).$} it follows that at least one of the knots is the unknot. Without a loss of generality we can assume this is $\kappa_1$, and the surgery is then a Dehn surgery on $\kappa_2$.

It is known that a manifold $Q$ which arises from surgery on a knot has a prime decomposition with at most three summands. If there are three summands, then two of the summands are lens spaces and the third a homology sphere (Corollary 5.3 of \cite{MR1915093}). Thus, because none of the summands arising in Lemma  \ref{lem: laginW} are homology spheres (see Remark \ref{rem: Poincare}), there are at most two summands. Note as well that $H_1(Q) \cong \mathbb{Z}/n\mathbb{Z}$ where $n$ is the surgery slope, which means that $Q$ cannot be one of the manifolds in Lemma \ref{lem: laginW} with an $S^1 \times S^2$ summand and a non-trivial spherical space form. It follows that $Q$ is either \begin{enumerate} \item a spherical space form with $|\pi_1(Q)|$ not divisible by 3, \item a connected sum of two of these types of spherical manifolds, \item or $Q=S^1 \times S^2$. \end{enumerate}

We will discuss each of these cases individually, beginning with: \vskip 5 pt

\emph{Case} $(1)$: Throughout the rest of the proof we let $\mathbf{k}$ be an algebraically closed field of characteristic 3. Suppose that $Q$ is a spherical space form with $|\pi_1(Q)|$ not divisible by 3 (note that this implies that $Q$ is a $\K$-homology sphere) and fix a generator for $H_1(Q,\mathbb{Z})=\mathbb{Z}/n\mathbb{Z}$. We may then equip $Q$ with $n$ rank one $\mathbf{k}$-local systems $\psi_m$ corresponding to $\rho^m, m\in \lbrace 0,\cdots,n-1\rbrace$ where $\rho$ is a primitive $n$-th root of unity.

 For $m \neq 0$, the cohomology groups $H^i(Q,\psi_m)=0$. This is obvious for $i=0$ and follows for $i=3$ using Poincar\'e duality in the form $H^i(Q,\psi_m) \cong H^{3-i}(Q,\psi_m^*)^*$. For $i=1,2$, we can again reduce the claim to $i=1$ using duality. To check this case, recall that cohomology with coefficients in this local coefficient system is by definition the (hyper-)ext $\operatorname{Ext}_{\mathbb{Z}[\pi_1(Q)]}^*(C_{-*}(\tilde{Q}), \K)$ where $\tilde{Q}$ is the universal cover of $Q$ and in the right-most argument we give $\K$ the module structure from the local coefficients. We have the hyper-Ext (or geometrically the Eilenberg-Moore) spectral sequence 
 \begin{align} \operatorname{Ext}_{\mathbb{Z}[\pi_1(Q)]}^p(H_q(\tilde{Q}),\K) \Rightarrow H^{p+q}(Q,\psi_m) 
 \end{align} 
which vanishes when $p=0,q=1$ because $\tilde{Q}$ is a universal cover and vanishes when $p=1,q=0$ because the group cohomology $H^1(\pi_1, \psi_m)=0$ (again we use that 3 does not divide $|\pi_1|$). 

In the Fukaya category of $W_1$, we have that $$\operatorname{Hom}((Q,
\psi_i),(Q,\psi_j)) \cong H^* (Q,\psi_{i-j})= 0 \textrm{ if $i \neq j$}.$$
Thus, the objects $\{(Q, \psi_i)\}$ are
pairwise disjoinable in the sense of \cite{Seidel:2014aa}; hence their homology
classes must span a subspace of rank at least $n/2$ by Theorem 1.7 of
{\em loc. cit.}, which applies because $Q$ is a $\K$-homology sphere.\footnote{Here, as
in Theorem \ref{thm: app1sec6}, we use the fact that our dilation arises
from a Hamiltonian which wraps around the divisor once; hence our manifold
has ``property (H)'' in the terminology of \cite{Seidel:2014aa}*{Definition
2.11}.} However, these objects all represent the same homology class and hence
$n\leq 2$; so by Lemma \ref{lem: spherenot3} we have that $Q=S^3$ or $\mathbb{R}P^3$. It follows from the main result of
\cite{MR2299739} that in either case $\kappa_2$ must be the unknot. \vskip 5 pt

\emph{Case} $(2)$: Essentially the same argument works for a connected sum of two of these types of manifolds $Q=Q_1\#Q_2.$  Namely, let $\psi_m$ denote a rank one $\K$ representation of $\pi_1(Q_1)$ considered in Case $(1)$. Note that $\pi_1(Q) \cong \pi_1(Q_1) \ast \pi_1(Q_2)$ and extend $\psi_m$ to a representation of $\pi_1(Q)$ by letting $\pi_1(Q_2)$ act trivially. We denote the resulting representation of $\pi_1(Q)$ by $\hat{\psi}_m.$ As in Case $(1)$, the key vanishing of the groups $H^i(Q,\hat{\psi}_m)$ can be reduced to the vanishing of the group cohomology $H^1(\pi_1(Q), \hat{\psi}_m)$ using Poincar\'e duality and the hyper-ext spectral sequence. We have that  $K(\pi_1(Q),1)$ is homotopy equivalent to the wedge sum: \begin{align} K(\pi_1(Q),1) \cong K(\pi_1(Q_1),1) \vee K(\pi_1(Q_2),1). \end{align}   By the Mayer-Vietoris sequence  \begin{align} H^1(\pi_1(Q),\hat{\psi}_m) \cong H^1(\pi_1(Q_1),\psi_{m}) \oplus H^1(\pi_2(Q_2),\K) = 0. \end{align}  (The point is that the restriction map $H^0(\pi_1(Q_2),\K) \to H^0(pt, \K)$ is an isomorphism because we have taken the representation to be trivial on the $\pi_1(Q_2)$ factor.) The same argument using Theorem 1.7 of \cite{Seidel:2014aa} shows that $Q_1$ is $\mathbb{R}P^3$ (or $S^3$ but then $Q$ is irreducible). By symmetry, we also deduce that $Q_2$ is $\mathbb{R}P^3$. However, then $H_1(Q)$ is not cyclic which is a contradiction. \vskip 5 pt

\emph{Case} $(3)$: Finally, if
$Q=S^1\times S^2$, then again using the result of \cite{MR2299739}, $\kappa_2$
must be the unknot (and the surgery slope is zero).  \end{proof}

\begin{rem}The idea of studying which diffeomorphism types of Lagrangians may arise as the surgery $Q$ of $\tilde{S}_1$ and $\tilde{S}_2$ is also borrowed from \cite{evanssmithwemyss}, who use a detailed classification of objects in the Fukaya category of $W_1$ (together with similar results from 3-manifold theory to those used above) to prove Proposition \ref{prop:esw}. \end{rem}

%%fakesection: starting appendix
\appendix
\section{Examples of Lagrangian connected sums} \label{section:AppendixA}

In this Appendix, we give some specific examples of exact Lagrangian connected sums in affine conic bundles. Let $X^o:= (\mathbb{C}^*)^2$ and $Z^o$ be the zero locus of $f(\bar{x})=x_1+x_2+\frac{1}{x_1x_2}-c$ for $c >3$ (recall \eqref{eq: conicbundle2}). In this case, we can take our good compactification (recall Definition \ref{defn: goodcomp}) so that $\bar{M}$ is a toric surface and $\bar{D}$ is its toric boundary. We will need to choose a specific K{\"a}hler form on the conic bundle. To do this, note that $\hat{M}$ is a hypersurface in the $\mathbb{C}P^1$ bundle, $\mathbb{P}(\mathcal{O}(Z) \oplus \mathcal{O})$ over $M:=\bar{M} \times \mathbb{C}P^1$ (in an abuse of notation we let $\mathcal{O}(Z)$ be the line bundle associated to the hypersurface $Z \times \mathbb{C}P^1 \subset \bar{M} \times \mathbb{C}P^1$). Let $\bar{\omega}$ be a toric K{\"a}hler form on $\bar{M}$ and lift this to a K{\"a}hler form on $M$  by taking product with the standard symplectic structure on $\mathbb{C}P^1$. Then a choice of Hermitian metric on the line bundle $\mathcal{O}(Z)$ induces a K{\"a}hler form on $\mathbb{P}(\mathcal{O}(Z) \oplus \mathcal{O})$ so that the $\mathbb{C}P^1$ fibers have area $\epsilon_Z$. We can then restrict this form to the hypersurface $\hat{M}$ to obtain a K{\"a}hler form which we denote by $\hat{\omega}.$ Away from $E$ we have that \begin{align} \hat{\omega}= \pi^* \omega_M +i\epsilon_Z \partial \bar{\partial} \psi \end{align} for some potential $\psi.$

Following Section 3.2. of \cite{AAK}, we will modify this K{\"a}hler form to a K{\"a}hler form $\omega$ so that $\omega$ agrees with $\pi^* \omega_M$ outside of (the preimage of) a tubular neighborhood $U$ of $Z \times 0 \subset \bar{M} \times \mathbb{C}P^1$. Namely, we choose a cut-off function $\chi: M \to [0,1]$ which is supported in $U$ and which is $S^1$ invariant with respect to the rotation action on $\mathbb{C}P^1$. We require that $\chi=1$ in a smaller open set about $Z \times 0$. We will also assume that the support of $\chi$ is disjoint from the (preimage of the) exact Lagrangian torus $\bar{L}_0 \subset (\mathbb{C}^*)^2$ given by \begin{align} |x_1|=|x_2|=1 \end{align}  On the complement of $E$, we set \begin{align} \hat{\omega}= \pi^* \omega_M +i\epsilon_Z \partial \bar{\partial} (\chi \psi) \end{align}
This extends to a form on all of $\hat{M}$ which again K{\"a}hler (after possibly shrinking $\epsilon_Z$). 

Let $\bar{L}_1$ denote be the ``positive real Lefschetz thimble," that is the portion of  $(\mathbb{R}^{\geq 0})^2$ which fibers over the interval $[3,c]$, \begin{align} \bar{L}_1:= \lbrace (\mathbb{R}^{\geq 0})^2, f(\bar{x}) \in [3,c]) \rbrace. \end{align} 

 $\bar{L}_0$ and $\bar{L}_1$ intersect at exactly one point p(where $x_1=x_2=1$) and we can form surgery at this intersection point to obtain an exact Lagrangian $\Sigma_{1,1} \subset (\mathbb{C}^*)^2$ with boundary on the hypersurface $Z^0$.  We let $L_0, L_1, Q$ be the subspaces obtained using the constructions \eqref{eq: suspendingLs} from $\bar{L}_0, \bar{L}_1, \Sigma_{1,1}$. 

\begin{lem} For a suitable choice of primitive $\theta$, $L_0, L_1,Q$ are all exact Lagrangian submanifolds of $X$.  \end{lem} 
\begin{proof} It is obvious that $L_0$ is an embedded submanifold and the fact that $L_1$ and $Q$ are embedded submanifolds is a local calculation near the boundary of $\bar{L}_1$. Because $\chi$ is supported away from $L_0$, it is also obvious that it is a Lagrangian submanifold(diffeomorphic to $T^3$).  It is also not difficult to make the Lagrangian $L_0$ exact by chosing a primitive which agrees with a suitable product primitive on $\mathbb{C}^* \times (\mathbb{C}^*)^2$.  Let $X\setminus E/\!\!/S^1$ denote the symplectic reduction of $X \setminus E$ at moment level set $\mu=\epsilon_Z$. $X \setminus E /\!\!/S^1$ is holomorphically identified with $(\mathbb{C}^*)^2 \setminus Z^o$ by the projection map, and is equipped with a K{\"a}hler form $\omega_{red}$ on $(\mathbb{C}^*)^2 \setminus Z^o$.\footnote{The K\"{a}hler form does not extend smoothly to all of $(\mathbb{C}^*)^2$; see Section 4.1 of \cite{AAK}. However, this is not relevant for us.}

 As $\bar{L}_1$ is fixed by an anti-holomorphic involution, it defines a Lagrangian with respect to $\omega_{red}$. It is then an elementary fact in the theory of symplectic reduction that $L_1$ is a Lagrangian in the total space $X$(which is automatically exact because it is diffeomorphic to $S^3$). The same argument applies to $Q$ because it agrees with $L_1$ when $\chi \neq 0$. Finally, the fact that $Q$ is exact as well for this choice of primitive follows from the fact that we have an isomorphism $H_1(\bar{L}_0 \setminus p) \cong H_1(Q)$ and in particular we can choose representatives for the generators of $H_1(Q)$ which lie far away from the surgery locus (which is a local construction). Exactness of $Q$ therefore follows from exactness of $L_0.$ \end{proof}

We can construct examples with $n=k+1$ by passing to $k$-fold covers of this conic bundle.  

\begin{rem} Tropical geometry provides a general method for constructing Lagrangian surfaces in $(\mathbb{C}^*)^2$(see e.g. \cite{MatessiLag}). It should be possible to ``suspend" some of these tropical Lagrangians to 3-dimensional Lagrangians in conic bundles, potentially giving rise to a wide class of examples. \end{rem} 

\section{More on \texorpdfstring{$\Psi$}{Psi}-equivariant structures} \label{section:AppendixB}

In general, unlike Hamiltonian Floer cohomology, Lagrangian Floer cohomology is not commutative. However, classes in the image of the maps  \eqref{eq:zerooc} are central elements of Lagrangian Floer cohomology. To be precise, for any pair of exact Lagrangians $Q_0, Q_1$ (with brane structure) there is a canonically defined homotopy (Equation (2.12) of \cite{Seidel:2010uq}): \begin{align} \phi_1: CF^*(X \subset M ; H_{m}^{2\lambda})\otimes CF^*(Q_0,Q_1) \to CF^*(Q_0,Q_1)[-1] \end{align} such that for any $b \in CF^*(X \subset M ; H_{m}^{2\lambda})$ and $a \in CF^*(Q_0,Q_1),$ 
$$\mu_2(\phi_{Q_{1}}(b),a)- (-1)^{|a||b|} \mu_2(a,\phi_{Q_{0}}(b)) = \mu_1(\phi^1(b,a)) +\phi^1(db,a)+ (-1)^{|b|}\phi^1(b,\mu_1(a)). $$ 

Given $\underline{\Psi}-$ equivariant structures on  $Q_0,Q_1$ (recall Definition \ref{defn:sseqs}), we can, following Equation (4.4) of \cite{Seidel:2010uq}, use this chain homotopy to define a chain map: 
\begin{equation} 
\begin{aligned} 
\label{eq:operation} \Phi_{Q_0,Q_1}: CF^*(Q_0,Q_1) \to CF^*(Q_0,Q_1) \\ 
\Phi_{Q_0,Q_1}(a)= \phi_1(\underline{\Psi},a) -\mu_2(c_{Q_{1}},a)+ \mu_2(a,c_{Q_{0}}). 
\end{aligned} 
\end{equation}

If $Q_0=Q_1=Q$, it is possible to calculate the endomorphism induced by \eqref{eq:operation} on cohomology (which we will in a slight abuse of notation also denote by $\Phi_{Q,Q}$) in certain degrees. Namely, by Corollary 3.6 of \cite{Seidel:2014aa}, we have that $\Phi_{Q,Q}$ acts by 0 on $HF^0(Q,Q).$ More interesting is the action on $HF^n(Q,Q)$, where we have by Equation (4.16) of
\cite{Seidel:2010uq} (see also discussion following Lemma 18.1 of \cite{catdyn}) that for any class $a \in HF^n(Q,Q)$ \begin{align} \label{eq:actiondegn} \Phi_{Q,Q}(a)= \mu_2(\Phi^0_Q(\Delta \Psi), a). \end{align} 
Note that when $\tilde{Q}$ is an equivariant sphere, these two calculations determine $\Phi_{Q,Q}$ completely. In particular, by \eqref{eq:actiondegn}, we have that if $\tilde{Q}$ is a sphere \begin{align} \label{eq:zzzzzsphere} \tilde{Q} \cdot \tilde{Q}= (-1)^n \phi^0_Q(\Delta \Psi) \in \K \end{align}
where $\tilde{Q} \cdot \tilde{Q}$ is defined in \eqref{eq: lefschetztrace}.
%fakesection:bibliography
\bibliography{shbib}
\bibliographystyle{apalike} 
%\printbibliography

\end{document}